\documentclass[11pt,oneside]{amsart}
\usepackage{fouriernc} 

\usepackage{Environments}   
\usepackage{Definitions}    
\usepackage{PageSetup}      

\usepackage{AuthorInfo}     

\newcommand{\proofof}[1]{} 
\newcommand{\Anglearrow}[1]{\rotatebox{#1}{$\Rightarrow$}}

\SetWatermarkText{}

\hypersetup{
 pdfkeywords={K-theory,2-categories,symmetric monoidal bicategory,connective spectra},
 pdfauthor={Nick Gurski,Niles Johnson,Angélica Osorno},
}


\title{$K$-theory for 2-categories}
\authorGurski
\authorJohnson
\authorOsorno
\date{05 December, 2017}

\usepackage{fancyhdr}
\usepackage{multicol}

\begin{document}

\begin{abstract}
  We establish an equivalence of homotopy theories between symmetric
monoidal bicategories and connective spectra.  For this, we develop
the theory of $\Gamma$-objects in 2-categories.  In the course of the proof we
establish strictification results of independent interest for symmetric
monoidal bicategories and for diagrams of 2-categories.


\end{abstract}

\maketitle
\tableofcontents

\section{Introduction}

The classifying space functor gives a way of constructing topological
spaces from categories. Extra structure on the categorical side often
results in extra structure on the corresponding spaces, and categorical coherence
theorems give rise to homotopical coherence. This is particularly the
case of symmetric monoidal categories.  Up to group completion, the
classifying space of a symmetric monoidal category is an infinite loop
space, as proven independently by Segal \cite{Seg74Categories} (using $\Ga$-spaces) and May \cite{may72geo} (using operads). The
corresponding spectrum is called the $K$-theory of the symmetric
monoidal category.  The name comes from the following key
example: when the input is the category of finitely generated
projective modules over a ring $R$, the corresponding spectrum is the
algebraic $K$-theory spectrum of the ring whose homotopy groups are then the $K$-groups of $R$.

Classically, $K$-theory developed as a strategy for converting
algebraic data to homotopical data.  Thomason \cite{Tho95Symmetric}
proved that $K$-theory establishes an equivalence of homotopy
categories between connective spectra and symmetric monoidal
categories.  Mandell gives an alternate approach in
\cite{Man10Inverse} combining an equivalence of homotopy categories
between $\Ga$-spaces and permutative categories with the Quillen
equivalence between connective spectra and $\Ga$-spaces of Bousfield
and Friedlander \cite{BF1978}.

In this work we prove that the $K$-theory functor defined in
\cite{Oso10Spectra} induces an equivalence of homotopy theories
between symmetric monoidal bicategories and connective spectra.  
One motivation for this level of generality is that
a number of interesting symmetric monoidal structures are naturally
2-categorical.  Among these are the 2-category of finite categories,
the cobordism bicategory, and bicategories arising from rings,
algebras, and bimodules.  A $K$-theory for 2-categories allows us to
work directly with this 2-dimensional algebra.

One important general example is that of the bicategory
$\mathpzc{Mod}_{R}$ of finitely generated modules over a bimonoidal
category $R$, defined in \cite{BDR04Two}. By results of
\cite{BDRR11Stable} and \cite{Oso10Spectra}, the bicategorical
$K$-theory of $\mathpzc{Mod}_R$ is equivalent to the
$K$-theory of the ring spectrum $KR$. In the case when
$R$ is the topological category of finite dimensional
complex vector spaces, we get that the $K$-theory spectrum of
connective complex topological $K$-theory, $K(ku)$, is equivalent to
$K(\mathpzc{2Vect}_{\bC})$, where $\mathpzc{2Vect}_{\bC}$ is the
bicategory of 2-vector spaces. Thus, $K(ku)$ is the classifying
spectrum for 2-vector bundles (see \cite{BDR04Two}).

We are also interested in the converse application, namely using
$K$-theory to provide algebraic models for stable homotopical
phenomena.  This gives another motivation to develop $K$-theory for
2-categories, namely that 2-dimensional homotopical data is most
naturally reflected in 2-dimensional algebra.  The first Postnikov
invariant of a connective spectrum is readily discernable in the
algebra of permutative (1-)categories (see \cite{JO12Modeling}), but
the most natural algebraic models for $n$th Postnikov invariants are
$n$-categorical.

In particular, we are interested in the homotopy theory of stable
homotopy 2-types and a characterization in terms of the algebra of
grouplike symmetric monoidal 2-groupoids.  We develop this in a
sequel, and explore applications to the Brauer theory of commutative
rings.

The primary goal of this paper is to prove the following theorem.
\begin{thm}
  \label{thm:main-equiv-ho-thy}
  There are equivalences of homotopy theories
  \[
  (\GaIICat,\mathpzc{S}) \simeq
  (\PIICat,\mathpzc{S}) \simeq
  (\PGM,\mathpzc{S})
  \]
  and therefore induced equivalences of stable homotopy categories.
\end{thm}

On the far left, we have $\Ga$-objects in the category of 2-categories
and 2-functors, and it is a relatively simple exercise in model
category theory (see \cref{prop:quillen-equiv-sset-iicat}) to show
that this is yet another model for the homotopy theory of connective spectra.  On the far
right, we have permutative Gray-monoids (\cref{defn:pgm}) which model
symmetric monoidal bicategories categorically, i.e., up to symmetric monoidal biequivalence.
In the middle, we have permutative 2-categories
(\cref{defn:perm-2-cat}).  These model symmetric monoidal bicategories
homotopically but not categorically, and admit a $K$-theory functor
landing in $\Ga$-2-categories which is simpler and better behaved than
the $K$-theory functor defined on all permutative Gray-monoids.  In each case, we have a notion of
stable equivalence, denoted $\mathpzc{S}$, and the resulting homotopy
theories are constructed using relative categories and complete Segal
spaces.  Thus we see that \cref{thm:main-equiv-ho-thy} gives an
equivalence of homotopy theories, and not just homotopy categories,
between connective spectra and symmetric monoidal bicategories.

These same techniques can be applied in the 1-categorical case to
generalize \cite[Theorem 1.3]{Man10Inverse} to an equivalence of
homotopy theories between $\Ga$-categories and symmetric monoidal
categories with their respective stable equivalences.

\addtocontents{toc}{\SkipTocEntry}
\subsection*{Outline}

In this section we outline the contents of the paper and explain how
they combine to prove the main result.  

In \cref{sec:preliminaries} we review some preliminary homotopy
theory.  In \cref{sec:css-basics} we recall the basic theory
of complete Segal spaces necessary for our work.  In
\cref{sec:gamma-basics} we define $\Ga$-objects in the category of
2-categories and show that these model all connective spectra.

In \cref{sec:SM-2-cats} we describe a number of different symmetric
monoidal structures on 2-categories or bicategories.  We show that
every symmetric monoidal bicategory is (categorically) equivalent to a
permutative Gray-monoid (\cref{defn:pgm}) and every symmetric monoidal
pseudofunctor can be replaced, up to a zigzag of strict monoidal
equivalences, by a strict monoidal functor between permutative
Gray-monoids.  We also define permutative 2-categories
(\cref{defn:perm-2-cat}), a stricter notion of primarily homotopical
interest.  We will eventually show that every permutative Gray-monoid
is \emph{weakly} equivalent to a permutative 2-category---see 
\cref{prop:epzo-strictification,thm:main-css-2}.

\cref{sec:diagrams-of-2-cats} establishes the groundwork we need for
general diagrams of 2-categories.  The essential concepts are a notion
of lax maps between diagrams and a Grothendieck construction on
symmetric monoidal diagrams.  The latter is necessary to define the
inverse to $K$-theory, and the former is necessary to define the lax
unit for $K$-theory and its inverse.  We end with a construction which
takes as input lax maps and produces spans of strict maps.  This is
the key construction which allows us to replace the lax unit of \cref{sec:construction-of-eta} with a
zigzag of strict equivalences.

In \cref{sec:permutative-2-cats-from-Gamma-2-categories} we define the
inverse to our $K$-theory functor.  It might seem strange to
define an inverse to $K$-theory before defining $K$-theory itself, but
the majority of our methods in this section are entirely standard
techniques in enriched category theory or 2-category theory.  The results
follow from our work in \cref{sec:diagrams-of-2-cats}.

In \cref{sec:K-theory-constructions} we define the $K$-theory
functor, or more precisely two different $K$-theory functors, one for
permutative Gray-monoids and one for permutative 2-categories.  We
also study the composite of $K$-theory followed by its inverse, and
show this is naturally weakly equivalent to the identity functor on
permutative Gray-monoids.

\cref{sec:eta-and-triangle-ids} studies the composite of the inverse $K$-theory functor followed by
$K$-theory, and in particular a unit for a putative adjunction between
these two functors.  Such a unit exists only as a lax map of diagrams,
but we show that this lax unit is a stable equivalence and use this to
prove our main theorem in two parts.
\Cref{thm:equiv-HoPGM-and-HoP2Cat-part2} establishes the right-hand
equivalence, and \cref{thm:main-css-1} establishes the left-hand
equivalence. 

\cref{appendix} contains a diagram of all the main categories used in
the paper, together with the functors that relate them. It also
includes a list of the categories, functors and natural
transformations that are key components of the proof of
\cref{thm:main-equiv-ho-thy}, along with precise references to where
they are introduced. We see this list as an aid to readers as they
navigate what the referee called ``a densely interconnected network of
(relative) categories connected by a bird's nest of constructed
functors, natural transformations and weak equivalence notions.'' In
particular, we hope that this appendix will serve as a useful guide
which the reader can reference as necessary to help place the
structures in context.

\addtocontents{toc}{\SkipTocEntry}
\subsection*{Acknowledgements} 

The authors gratefully acknowledge the support and hospitality of the
Mathematical Sciences Research Institute in Berkeley, California: this
material is partially based on work completed while the first and
second authors visited and the third author was in residence at the
MSRI during the Spring 2014 semester.  This work was also partially
supported by a grant from the Simons Foundation (3599449, Ang\'elica
Osorno), and by the EPSRC (EP/K007343/1 Nick Gurski).  We also thank
the Ohio State University at Newark, Reed College, and the University
of Sheffield for hospitality hosting various combinations of the
authors over the past three years.  For helpful conversations, the
authors wish to particularly thank Peter May, Mike Mandell, and Chris
Schommer-Pries. We thank the referees carefully reading such a long
paper, and for suggesting we include the Appendix.

\section{Preliminaries}
\label{sec:preliminaries}

This section will introduce some of the background homotopy theory we
will use throughout the paper.  Our main result,
\cref{thm:main-equiv-ho-thy}, establishes equivalences between three
different homotopy theories, and for this we use the machinery of
complete Segal spaces \cite{Rez01Model}.  We give the relevant
definitions and some very basic results in \cref{sec:css-basics}.
Since our ultimate goal is to show that connective spectra can be
modeled categorically, we must choose how to represent these spectra
for the comparison.  Here we have chosen to use $\Ga$-spaces, together
with their class of stable equivalences, as our basic notion of the
homotopy theory of connective spectra.  Therefore the second goal of
this section is to show that we can use $2$-categories with their
Thomason model structure instead of spaces when thinking about
$\Ga$-objects, and we explain this in \cref{sec:gamma-basics}.
We also recall the category $\Ga$ itself, together with some related
and useful categories, and fix some notation.

\subsection{Complete Segal spaces}\label{sec:css-basics}

Complete Segal spaces are bisimplicial sets satisfying certain
homotopical properties, and the complete Segal space model structure
on bisimplicial sets is constructed as a localization of the Reedy
model structure.  We will assume some familiarity with these concepts;
the interested reader can consult \cite{Rez01Model,Hir03Model}.  It
should be noted that this theory can be easily taken as a black box:
\cref{cor:rel-cats-sufficient} suffices on its own for all of our
applications.

\begin{defn}\label{defn:css}
A bisimplicial set $X$ is a \emph{complete Segal space} \cite{Rez01Model} if
\begin{itemize}
\item it is fibrant in the Reedy model structure on bisimplicial sets,
\item for each $n$, the Segal map $X(n) \to X(1) \times_{X(0)} \cdots \times_{X(0)} X(1)$ is a weak equivalence of simplicial sets, and
\item the map $\textrm{Map}(E,X) \to \textrm{Map}(\Delta[1],X) \cong X(1)$, induced by the inclusion of the arrow category into the free living isomorphism (where we take nerves, and then treat the resulting simplicial sets as discrete bisimplicial sets), is a weak equivalence of simplicial sets.
\end{itemize}
\end{defn} 

One of the main results of \cite{Rez01Model} is the following.

\begin{thm}\label{thm:rezk}
There is a model structure on the category of bisimplicial sets which is a left Bousfield localization of the Reedy model structure such that the fibrant objects are precisely the complete Segal spaces.
\end{thm}

In practice, one constructs complete Segal spaces from model
categories, or more generally relative categories.

\begin{defn}\label{defn:relcat}
A \emph{relative category} is a pair $(\cC,\cW)$ in which $\cC$ is a
category and 
$\cW$ is a subcategory of $\cC$ containing all of the objects.  A
\emph{relative functor} $F \cn (\cC,\cW) \to (\cC',\cW')$ is a functor
$F \cn \cC \to \cC'$ such that $F$ restricts to a functor
$\cW \to \cW'$.
\end{defn}

\begin{notn}\label{notn:rel-nerve}
Let $(\cC,\cW)$ be a relative category, and let $\cA$ be any category.  Then $(\cC,\cW)^{\cA}$ will denote the category whose objects are functors $\cA \to \cC$ and whose morphisms are those natural transformations with components in $\cW$.
\end{notn}

We let $N$ denote the usual nerve functor $N \cn \Cat \to \sSet$.

\begin{defn}\label{defn:classification-diag}
  Let $(\cC,\cW)$ be a relative category.  The \emph{classification
    diagram} of $(\cC,\cW)$ is the bisimplicial set $\cN (\cC,\cW)$ given by
  \[
  n \mapsto N\big( (\cC,\cW)^{\Delta[n]} \big),
  \]
  where as usual $\Delta[n]$ is the category of $n$ composable arrows.
\end{defn}

The classification diagram of a relative category is rarely a complete
Segal space, one must usually take a fibrant replacement.  This can
often be done using only a \emph{Reedy} fibrant replacement (see
\cite{Rez01Model,Hir03Model}), but for an arbitrary relative category
one sometimes needs to take a fibrant replacement in the complete Segal space model
structure.

\begin{defn}\label{defn:css-from-relcat}
  Let $(\cC,\cW)$ be a relative category.  The \emph{homotopy theory of
    $(\cC,\cW)$} is given by a fibrant replacement of the classification
  diagram of $(\cC,\cW)$ in the complete Segal space model structure.  We
  write such a fibrant replacement as $R\cN (\cC,\cW)$.  A relative
  functor $F\cn(\cC,\cW)\to (\cC',\cW')$ is an \emph{equivalence of homotopy
    theories} if the morphism $R\cN F$ is a weak equivalence in the
  complete Segal space model structure.
\end{defn}

Recall the hammock localization \cite{DK80Calculating} of a relative
category $(\cC,\cW)$ is a simplicially-enriched category $L^H(\cC,\cW)$.  One
key property of hammock localization is that the category of
components $\pi_0L^H(\cC,\cW)$ is equivalent to the categorical
localization $\cC[\cW^{-1}]$. A DK-equivalence of simplicially-enriched
categories is a simplicially-enriched functor which induces weak
equivalences on mapping simplicial sets and for which the induced
functor on categories of components is an equivalence of categories.

\begin{prop}[{\cite[1.8]{BK12Characterization}}]
  \label{prop:Rezk-equiv-is-DK-equiv}
  A relative functor is an equivalence of homotopy theories if and
  only if it induces a DK-equivalence on hammock localizations.
\end{prop}

One can actually say something stronger using \cite{Toen05Vers}, namely that the classification
diagram of a relative category is equivalent to the hammock
localization, so that both give the same notion of homotopy category.

The following lemma is an immediate consequence of the complete Segal space model structure being a localization of the Reedy model structure on bisimplicial sets.

\begin{lem}\label{lem:levelwise-eq-css}
Let $F\cn (\cC,\cW) \to (\cC',\cW')$ be a relative functor with the property that composition with $F$ induces a weak equivalence of categories
\[
N\big( (\cC,\cW)^{\Delta[n]} \big) \to N\big( (\cC',\cW')^{\Delta[n]} \big)
\]
for each $n$.  Then $F$ is an equivalence of homotopy theories.
\end{lem} 

\begin{cor}\label{cor:rel-cats-sufficient}
Let $F\cn (\cC,\cW) \to (\cC',\cW'), G\cn (\cC',\cW') \to (\cC,\cW)$ be relative functors and let $\al \cn FG \impl \id, \be \cn GF \impl \id$ be natural transformations.  If the components of $\al$ are all in $\cW'$ and the components of $\be$ are all in $\cW$, then $F,G$ induce weak equivalences between $R\cN (\cC,\cW)$ and  $R\cN (\cC',\cW')$ in the complete Segal space model structure.  This conclusion remains valid if $\al, \be$ are replaced with any finite zigzag of transformations between relative functors such that the components of the zigzag replacing $\al$ are all in $\cW'$ and the components of the zigzag replacing $\be$ are all in $\cW$.
\end{cor}
\begin{proof}\proofof{cor:rel-cats-sufficient}
The transformations $\al, \be$ induce weak homotopy equivalences between the nerves of $\cC,\cC'$, and thus of $\cC^{\Delta[n]},\cC'^{\Delta[n]}$ as well.  The hypotheses on the components of $\al,\be$ ensure that they restrict to natural transformations when the morphisms of $\cC^{\Delta[n]},\cC'^{\Delta[n]}$ are restricted to those maps which are componentwise in $\cW,\cW'$, respectively.
\end{proof}

\subsection{\texorpdfstring{$\Ga$}{Γ}-2-categories}\label{sec:gamma-basics}

For a natural number $m \ge 0$ we let $\ul{m}$ be the set with $m$
elements $\{1, \ldots m\}$ where $\ul{0} = \emptyset$.  We let be
$\ul{m}_+$ the pointed set $\{0, 1, \ldots, m\}$ with basepoint $0$.

\begin{itemize}
  \item $\sN$ denotes the category with objects $\ul{m}$ and maps of
    sets.
  \item $\sF$ denotes the category with objects $\ul{m}_+$ for $m \ge
    0$ and pointed maps.
\end{itemize}

\begin{notn}\label{notn:iicats}
Throughout, we let $\IICat$
denote the (1-)category of 2-categories and strict 2-functors.  We let
$\IICat_2$ denote the 2-category of 2-categories, 2-functors, and
2-natural transformations. 
\end{notn}

\begin{defn}\label{defn:gammaobjs}
  For any category $\cC$ with a terminal object $\ast$, a
  \emph{$\Ga$-object} $X$ in $\cC$ is a functor $X\cn \sF \to \cC$
  such that $X(\ul{0}_+)=\ast$.  The term \emph{reduced} is also used
  for this condition on $X(\ulp{0})$.  We will denote by $\Ga\mh \cC$
  the category of $\Ga$-objects and natural transformations.  We will
  denote by $\sF \mh \cC$ the category of all functors $\sF \to \cC$,
  and then $\Ga\mh\cC$ is the full subcategory of $\sF \mh \cC$
  consisting of reduced $\sF$-diagrams in $\cC$.
\end{defn}

We note that Segal's category $\Ga$ is the opposite of $\sF$ although he did not originally define it that way.  Our interest in $\Ga$-2-categories is that they will be the output of the various $K$-theory functors.

There are several possible definitions of nerve of a 2-category.
However there are natural weak equivalences connecting any pair of
such definitions (see \cite{CCG10Nerves} or \cite{CHR2015Bicategorical}) so it is not necessary to
choose a particular classifying space functor in order to make the
next definition.

\begin{defn}
  \label{defn:equivs-and-weak-equivs}
  Let $\cA, \cB$ be 2-categories, and $F:\cA \to \cB$ be a 2-functor
  between them.
  \begin{enumerate}
  \item $F$ is an \emph{equivalence} of 2-categories if each object $b
    \in \cB$ is equivalent to $Fa$ for some $a \in \cA$, and if each
    functor $\cA(a,a') \to \cB(Fa, Fa')$ is an equivalence of
    categories.
  \item $F$ is a \emph{weak equivalence} if the induced map on nerves
    $NF: N\cA \to N\cB$ is a weak equivalence of simplicial sets.
  \end{enumerate}
\end{defn}

\begin{rmk}
  For the definition of equivalence of 2-categories above, it is
  useful to recall that two objects $b,b'$ in a 2-category $\cB$ are
  equivalent, or internally equivalent, if there exist 1-cells $f:b
  \to b', g: b' \to b$ and isomorphism 2-cells $\al:fg \cong \id_{b'},
  \be: gf \cong \id_{b}$.  Similarly, a 1-cell $f:b \to b'$ is an
  internal equivalence in $\cB$ if $g, \al, \be$ exist as above.  One
  can then check that the definition of equivalence of 2-categories
  given above amounts to what is usually called a biequivalence.  It
  is not the case, though, that if a 2-functor $F:\cA \to \cB$ is a
  biequivalence then there is necessarily a 2-functor $G:\cB \to \cA$
  which is a biequivalence and a weak inverse for $F$.  In general,
  one can only produce such a $G$ as a pseudofunctor.
\end{rmk}

There are two model structures on the category of 2-categories and
2-functors.  The first has equivalences as its class of weak
equivalences, and was established by Lack
\cite{Lac02Quillen2,Lac04Quillenb}.  The second, called the Thomason
model structure, has the weak equivalences (as given in the definition
above) for its class of weak equivalences (in the model categorical
sense), and was established by \cite{WHPT07Model} with corrections by
\cite{AM13Towards}.
By standard arguments, every equivalence is a weak equivalence, so any
functor which inverts weak equivalences also inverts equivalences.
Further, any functor inverting weak equivalences must identify
2-functors with a 2-natural (or lax, or oplax) transformation between
them by \cite{CCG10Nerves}.  In \cref{sec:D-trans} we will express
this fact using a functorial path object so that we can apply the same
technique to diagrams of 2-categories.

\begin{notn}\label{notn:ho-F-2cat}
  Let $\ho \sF\mh \IICat$ denote the category obtained from
  $\sF\mh\IICat$ by inverting the maps which are levelwise weak
  equivalences, and $\ho \Ga\mh \IICat$ denote the category obtained
  from $\Ga\mh\IICat$ by inverting the maps which are levelwise weak
  equivalences.  
\end{notn}

\begin{rmk}
  Since the Thomason model structure on $\IICat$ is cofibrantly
  generated, we can produce the homotopy category $\ho \sF\mh \IICat$
  using the generalized Reedy model structure in which weak
  equivalences are levelwise.  As noted in the passing remark of
  \cite[p. 95]{BF1978}, this restricts to a generalized Reedy model
  structure on $\Ga\mh\IICat$ because for each $n \ge 1$ the category
  of $\Sigma_n$-objects in $\IICat$ has the projective model
  structure.
\end{rmk}

The Quillen equivalence between $\sSet$ and $\IICat$ of
\cite{WHPT07Model,AM13Towards} yields a Quillen equivalence between
the respective generalized Reedy model structures on $\sF\mh\sSet$ and
$\sF\mh\IICat$.  Since both functors of the adjunction between $\sSet$
and $\IICat$ preserve the terminal object, the corresponding
adjunction between $\sF\mh\sSet$ and $\sF\mh\IICat$ restricts to give
the following.
\begin{prop}
  \label{prop:quillen-equiv-sset-iicat}
  There is a Quillen equivalence between $\Ga\mh\sSet$ and
  $\Ga\mh\IICat$, both with the model structures described above.
\end{prop}

Thus we can transport definitions directly from the context of
$\Ga$-spaces (here seen as $\Ga$-simplicial sets) to
$\Ga$-2-categories. \Cref{prop:quillen-equiv-sset-iicat} implies, in
particular, that $\ho \Ga\mh \IICat$ is locally small.

\begin{defn}
\label{defn:special-very-special}
Let $X$ be a $\Ga$-2-category.
\begin{enumerate}
\item Let $i_{k}:\ul{n}_{+} \to \ul{1}_{+}$ denote the unique map
  sending only $k$ to the non-basepoint element in $\ul{1}_{+}$.  Then
  $X$ is \emph{special} if the map
  $X(\ul{n}_{+}) \to X(\ul{1}_{+})^{n}$ induced by all the $i_{k}$ is
  a weak equivalence of 2-categories.
\item A special $\Ga$-2-category $X$ is \emph{very special} when
  $\pi_{0}X(\ul{1}_{+})$ is a group under the operation
  \[
  \pi_{0}X(\ul{1}_{+}) \times \pi_{0}X(\ul{1}_{+}) \cong \pi_{0}X(\ul{2}_{+}) \to \pi_{0}X(\ul{1}_{+})
  \]
  where the isomorphism is induced by the weak equivalence
  $X(\ul{2}_{+}) \to X(\ul{1}_{+})^{2}$ and the final map is induced by the morphism $\ul{2}_{+} \to \ul{1}_{+}$ which sends both non-basepoints in the source to the non-basepoint in the target.
\end{enumerate}
\end{defn}

\begin{defn}
  \label{defn:stable-equiv}
  A map $f:X \to Y$ in $\Ga\mh\IICat$ is a \emph{stable equivalence}
  if the function
  \[
  f^{*}:\ho \Ga\mh \IICat(Y,Z) \to \ho \Ga\mh \IICat(X,Z)
  \]
  is an isomorphism for every very special $\Ga$-2-category $Z$.
\end{defn}

\begin{notn}\label{notn:Ho-Ga-2cat}
  Let $\Ho \GaIICat$ denote the category obtained from $\GaIICat$ by
  inverting the stable equivalences, and $\Ho \Ga\mh\sSet$ denote the
  category obtained from $\Ga\mh\sSet$ by inverting the stable
  equivalences.
\end{notn}

\begin{defn}\label{defn:W-Ga2cat}
Let $\cW$ denote the collection of weak equivalences in $\GaIICat$ and
$\cS$ the collection of stable equivalences.  Then $(\GaIICat,\cW)$
and $(\GaIICat,\cS)$ are the respective relative categories.  Similarly, $(\Ga\mh\sSet, \cW)$ is the relative category of $\Ga$-simplicial sets and levelwise weak equivalences, and $(\Ga\mh\sSet, \cS)$ is the relative category of $\Ga$-simplicial sets and stable equivalences.
\end{defn}

Note that $\Ho \Ga\mh\sSet$ is (equivalent to) the full subcategory of
the stable homotopy category whose objects are connective spectra \cite{BF1978}.
The remarks above prove the following theorem. 

\begin{thm}
  Applying the adjunction between $\IICat$ and $\sSet$ levelwise, we
  obtain an adjunction between $\GaIICat$ and $\Ga\mh\sSet$.  This
  induces an equivalence of homotopy theories
  \[
  (\Ga\mh\IICat, \cS) \fto{\hty} (\Ga\mh\sSet, \cS)
  \]
  and therefore an equivalence of stable homotopy categories
  \[
  \Ho \Ga\mh\IICat \fto{\hty} \Ho \Ga\mh\sSet.
  \]
\end{thm}


\section{Symmetric monoidal 2-categories}
\label{sec:SM-2-cats}

In this section we give basic results on symmetric monoidal structures
for 2-categories and bicategories.  We first outline the general theory of
symmetric monoidal bicategories in \cref{sec:SM-bicats} and then
describe the stricter notions of permutative Gray-monoid and
permutative 2-category in
\cref{sec:permutative-Gray-monoids,sec:perm-2-cat}, respectively.
Both descriptions make use of the Gray tensor product, which we recall
in \cref{sec:2-cats-and-Gray-tensor-product}.

\subsection{Symmetric monoidal bicategories}
\label{sec:SM-bicats}

While our main objects of study, permutative Gray-monoids and
permutative 2-categories, are 2-categories (see \cref{defn:pgm} and
\cref{defn:perm-2-cat}), it is important for us to use some of the
more general theory of symmetric monoidal bicategories.  Since we use
only some main results of this theory and not the particular details
of many definitions, we only sketch the notions of symmetric monoidal
bicategory together with functors and transformations between them.
Good references for these details include 
\cite{McCru00Balanced,Lac10Icons,Sch2011Classification,CG2014Iterated}.  In
this section, functor always means pseudofunctor, transformation
always means pseudonatural transformation (which we will only indicate
via components), and equivalence always means pseudonatural adjoint
equivalence.

\begin{defn}[Sketch, see {\cite[Defn. 2.3]{Sch2011Classification}} or \cite{CG2014Iterated}]
\label{defn:smb}
A \textit{symmetric monoidal bicategory} consists of
\begin{itemize}
\item a bicategory $\cB$,
\item a tensor product functor $\cB \times \cB \to \cB$,
\item a unit object $e \in \ob \cB$,
\item an associativity equivalence $(xy)z \simeq x(yz)$,
\item unit equivalences $xe \simeq x \simeq ex$
\item four invertible modifications between composites of the unit and
  associativity equivalences 
\item a braid equivalence $\beta:xy \simeq yx$,
\item two invertible modifications (denoted $R_{--|-}, R_{-|--}$)
  which correspond to two instances of the third Reidemeister move,
\item and an invertible modification (the syllepsis, $v$, which has
  components indexed by pairs of objects) between $\beta \circ \beta$
  and the identity
\end{itemize}
satisfying three axioms for the monoidal structure, four axioms for
the braided structure, two axioms for the sylleptic structure, and one
final axiom for the symmetric structure.
\end{defn}

\begin{rmk}
  In the theory of monoidal categories, there are notions of braided
  and symmetric structures, but a sylleptic structure is a new
  intermediate structure that only appears in the 2-categorical
  context.  The syllepsis axioms compare $v_{xy,z}$ with the composite
  of $v_{x,z}$ and $v_{y,z}$ (and similarly in the second variable),
  while the symmetry axiom ensures that there is a unique isomorphism
  between $\beta^{3}$ and $\beta$. Moreover, the coherence theorem for symmetric monoidal bicategories in \cite{GO12infinite} states that the axioms in this definition imply there is a unique structural isomorphism between any compositions of the associativity, unit, and braiding equivalences which represent the same permutation of objects.
\end{rmk}

\begin{defn}[Sketch, see {\cite[Defn. 2.5]{Sch2011Classification}}]\label{defn:symm-mon-psfun}
  A \textit{symmetric monoidal pseudofunctor} $F:\cB \to \cC$ consists
  of
  \begin{itemize}
  \item a functor $F:\cB \to \cC$,
  \item a unit equivalence $F(e_{\cB}) \simeq e_{\cC}$,
  \item an equivalence for the tensor product $Fx Fy \simeq F(xy)$,
  \item three invertible modifications between composites of the unit
    and tensor product equivalences, and
  \item an invertible modification $U$ comparing the braidings in
    $\cB$ and $\cC$
  \end{itemize}
  satisfying two axioms for the monoidal structure, two axioms for the
  braided structure, and one axiom for the symmetric (and hence
  subsuming the sylleptic) structure.
\end{defn}

\begin{defn}[Sketch, see {\cite[Defn. 2.7]{Sch2011Classification}}]\label{defn:sym-mon-transf}
  A \textit{symmetric monoidal transformation} $\al: F \rtarr G$
  consists of
  \begin{itemize}
  \item a transformation $\al: F \rtarr G$ and
  \item two invertible modifications concerning the interaction
    between $\al$ and the unit objects on the one hand and the tensor
    products on the other
  \end{itemize}
  satisfying two axioms for the monoidal structure and one axiom for
  the symmetric structure (and hence subsuming the braided and
  sylleptic structures).
\end{defn}

The following is verified in \cite{Sch2011Classification}.  Note that we
have not defined symmetric monoidal modifications as we will not have
any reason to use them in any of our constructions.

\begin{lem}\label{lem:smb3}
  There is a tricategory, $\SMBicat_3$, of symmetric monoidal bicategories, symmetric
  monoidal pseudofunctors, symmetric monoidal transformations, and
  symmetric monoidal modifications.
\end{lem}

We will need to know when symmetric monoidal pseudofunctors or
transformations are invertible in the appropriate sense.

\begin{defn}\label{defn:sym-mon-biequiv}
  A \textit{symmetric monoidal biequivalence} $F: \cB \to \cC$ is a
  symmetric monoidal pseudofunctor such that the underlying functor
  $F$ is a biequivalence of bicategories.
\end{defn}

\begin{defn}
  \label{defn:symm-mon-equiv}
  A \textit{symmetric monoidal equivalence} $\al: F \rtarr G$ between
  symmetric monoidal pseudofunctors $F,G: \cB \to \cC$ is a symmetric monoidal
  transformation $\al: F \rtarr G$ such that the underlying
  transformation $\al$ is an equivalence.  This is logically
  equivalent to the condition that each component 1-cell
  $\al_{b}:Fb \to Gb$ is an equivalence 1-cell in $\cC$.
\end{defn}

The results of \cite{Gur2012Biequivalences} can be used to easily prove
the following lemma, although the first part is also verified by
elementary means in \cite{Sch2011Classification}.

\begin{lem}
  Let $F,G: \cB \to \cC$ be symmetric monoidal pseudofunctors, and
  $\al:F \rtarr G$ a symmetric monoidal transformation between them.
  \begin{itemize}
  \item $F$ is a symmetric monoidal biequivalence as above if
    and only if it is an internal biequivalence in the tricategory
    $\SMBicat_3$.
  \item $\al$ is a symmetric monoidal equivalence if and
    only if it is an internal equivalence in the bicategory
    $\SMBicat_3(\cB,\cC)$.
  \end{itemize}
\end{lem}

We now come to the definition and results from
\cite{Sch2011Classification} that are most important for our construction
of $K$-theory later.

\begin{defn}[{\cite[Def 2.28]{Sch2011Classification}}]
  \label{defn:quasi-strict}
  A symmetric monoidal bicategory $\cB$ is a \textit{quasi-strict
    symmetric monoidal 2-category} if
  \begin{itemize}
  \item the underlying monoidal bicategory of $\cB$ is a Gray-monoid
    (see \cref{defn:gray-monoid}),
  \item the braided structure is strict in the sense of Crans
    \cite{Cra98Generalized}, and
  \item the following three additional axioms hold.
  \end{itemize}
  \begin{itemize}
  \item[QS1] The modifications $R_{--|-}, R_{-|--}, v$ are all identities.
  \item[QS2] The naturality 2-cells
    \[
    \xy
    (0,0)*+{ab}="00";
    (30,0)*+{a'b}="10";
    (0,-12)*+{ba}="01";
    (30,-12)*+{ba'}="11";
    {\ar^{f1} "00"; "10" };
    {\ar^{\beta} "10"; "11" };
    {\ar_{\beta} "00"; "01" };
    {\ar_{1f} "01"; "11" };
    (15,-6)*{\cong \, \beta_{f1}};
    (60,0)*+{ab}="22";
    (90,0)*+{ab'}="32";
    (60,-12)*+{ba}="23";
    (90,-12)*+{b'a}="33";
    {\ar^{1g} "22"; "32" };
    {\ar^{\beta} "32"; "33" };
    {\ar_{\beta} "22"; "23" };
    {\ar_{g1} "23"; "33" };
    (75,-6)*{\cong \, \beta_{1g}};
    \endxy
    \]
    for the pseudonatural transformation $\beta$ are identities.
  \item[QS3] The 2-cells $\Sigma_{\beta, g}, \Sigma_{f, \beta}$ (see
    \cref{defn:graytensor}) are the identity for any 1-cells $f,g$.
\end{itemize}
\end{defn}

\begin{rmk}\ 
  \begin{enumerate}
  \item The first axiom above implies that the tensor product is given
    by a cubical functor, and that this operation is strictly unital
    and associative.  It is the functoriality isomorphism for the
    tensor product as a cubical functor that gives rise to the 2-cells
    $\Sigma$ that appear in (QS3).  We will discuss cubical functors
    and their properties in the next section.
  \item Our version of (QS3) is not exactly as it is presented in
    \cite{Sch2011Classification}, but it is equivalent using the cubical
    functor axioms.  We find this version more amenable to our later
    work.
  \end{enumerate}
\end{rmk}

There is also a stricter notion of morphism that will be of interest
to us.

\begin{defn}\label{strict-functor}
  A \textit{strict functor} $F:\cB \to \cC$ between symmetric monoidal
  bicategories is a strict functor of the underlying bicategories that
  preserves all of the structure strictly, and for which all of the
  constraints are either the identity (if this makes sense) or unique
  coherence isomorphisms from $\cC$.
\end{defn}

\begin{rmk}
  There is a monad on the category of 2-globular sets whose algebras
  are symmetric monoidal bicategories.  Strict functors can then be
  identified with the morphisms in the Eilenberg-Moore category for
  this monad, and in particular symmetric monoidal bicategories with
  strict functors form a category.  This point of view is crucial to
  the methods employed in \cite{Sch2011Classification}.
\end{rmk}

\begin{notn}\label{notn:smb-qst}
  We let $\SMBicat$ denote the 1-category of symmetric monoidal
  bicategories and strict functors.  We let $\qsSMIICat$ denote the
  full subcategory of $\SMBicat$ consisting of quasi-strict symmetric
  monoidal bicategories.
\end{notn}

The following strictification theorem of \cite{Sch2011Classification} 
enables us to restrict attention to quasi-strict symmetric monoidal
2-categories.  Its proof relies heavily on the coherence theorem for
symmetric monoidal bicategories in \cite{GO12infinite}.
\begin{thm}[{\cite[Thm. 2.97]{Sch2011Classification}}]
  \label{cohqs2cats}
  Let $\cB$ be a symmetric monoidal bicategory.
  \begin{enumerate}
  \item There are two endofunctors, $\cB \mapsto \cB^{c}$ and
    $\cB \mapsto \cB^{qst}$, of the category of symmetric monoidal
    bicategories and strict functors between them.  Any symmetric
    monoidal bicategory of the form $\cB^{qst}$ is a quasi-strict
    symmetric monoidal 2-category.
  \item There are natural transformations
    $(-)^{c} \impl \id, (-)^{c} \impl (-)^{qst}$.  When evaluated at a
    symmetric monoidal bicategory $\cB$, these give natural strict
    biequivalences
    \[
    \cB \leftarrow \cB^{c} \to \cB^{qst}.
    \]
  \item For a symmetric monoidal pseudofunctor $F:\cB \to \cC$, there
    are strict functors $F^{c}:\cB^{c} \to \cC^{c}$,
    $F^{qst}:\cB^{qst} \to \cC^{qst}$ such that the right hand square
    below commutes and the left hand square commutes up to a symmetric
    monoidal equivalence.
    \[
    \xy
    (0,0)*+{\cB}="00";
    (30,0)*+{\cB^{c}}="10";
    (60,0)*+{\cB^{qst}}="20";
    (0,-12)*+{\cC}="01";
    (30,-12)*+{\cC^{c}}="11";
    (60,-12)*+{\cC^{qst}}="21";
    {\ar^{F^{c}} "10"; "11" };
    {\ar "10"; "20" };
    {\ar^{F^{qst}} "20"; "21" };
    {\ar "10"; "00" };
    {\ar_{F} "00"; "01" };
    {\ar "11"; "01" };
    {\ar "11"; "21" };
    (15,-6)*{\simeq}
    \endxy
    \]
  \end{enumerate}
\end{thm}

We will need a strengthening of the above result in the special case
that we begin with quasi-strict symmetric monoidal 2-categories
instead of the more general symmetric monoidal bicategories.

\begin{thm}\label{thm:cohqs2cats2}
  Let $\cB$ be a quasi-strict symmetric monoidal 2-category.
  \begin{enumerate}
  \item There is a strict functor $\nu \cn \cB^{qst} \to \cB$ such
    that
    \[
    \xy
    (0,0)*+{\cB^{c}}="0";
    (40,0)*+{\cB^{qst}}="1";
    (20,-12)*+{\cB}="2";
    {\ar "0"; "1"};
    {\ar^{\nu} "1"; "2"};
    {\ar "0"; "2"};
    \endxy
    \]
    commutes, where the unlabeled morphisms are those from
    \cref{cohqs2cats}.  In particular, $\nu$ is a strict symmetric
    monoidal biequivalence.
  \item For a symmetric monoidal pseudofunctor $F:\cB \to \cC$ with
    $\cB, \cC$ with both quasi-strict, the square below commutes up to
    a symmetric monoidal equivalence.
    \[
    \xy
    (0,0)*+{\cB^{qst}}="00";
    (40,0)*+{\cB}="10";
    (0,-12)*+{\cC^{qst}}="01";
    (40,-12)*+{\cC}="11";
    {\ar^{\nu} "00"; "10" };
    {\ar^{F} "10"; "11" };
    {\ar_{F^{qst}} "00"; "01" };
    {\ar_{\nu} "01"; "11" };
    (20,-6)*{\simeq}
    \endxy
    \]
  \end{enumerate}
\end{thm}
\begin{proof}\proofof{thm:cohqs2cats2}
  The first statement follows from the general theory of computads
  developed in \cite{Sch2011Classification} and the calculus of mates
  \cite{KS74Review}.  We conclude that $\nu$ is a strict symmetric
  monoidal biequivalence since the other two maps in the triangle are,
  and biequivalences satisfy the 2-out-of-3 property.  The second
  claim follows directly from the third part of \cref{cohqs2cats}.
\end{proof}

In \cref{p2isoqs2} we repackage the definition of quasi-strict
symmetric monoidal 2-category using the Gray tensor product.  Before
doing so, we give a basic review of the Gray tensor product in the
next section.

\subsection{2-categories and the Gray tensor product}
\label{sec:2-cats-and-Gray-tensor-product}

Let $\IICat$ denote the category of strict 2-categories and strict
2-functors between them.  We will often be concerned with a monoidal
structure on $\IICat$ which is not the cartesian structure, but
instead uses the tensor product defined below.  For further reference,
see \cite{Gra74Formal,GPS95Coherence,Gurski13Coherence}.

\begin{defn}\label{defn:graytensor}
  Let $\cA, \cB$ be 2-categories.  The Gray tensor product of $\cA$
  and $\cB$, written $\cA \otimes \cB$ is the 2-category given by
  \begin{itemize}
  \item 0-cells consisting of pairs $a \otimes b$ with $a$ an object
    of $\cA$ and $b$ an object of $\cB$;
  \item 1-cells generated by basic 1-cells of the form
    $f \otimes 1: a \otimes b \to a' \otimes b$ for $f:a \to a'$ in
    $\cA$ and $1 \otimes g: a \otimes b \to a \otimes b'$ for
    $g:b \to b'$ in $\cB$, subject to the relations
    \begin{itemize}
    \item $(f \otimes 1) (f' \otimes 1) = (ff') \otimes 1$,
    \item $(1 \otimes g)(1 \otimes g') = 1 \otimes (gg')$
    \end{itemize}
    and with identity 1-cell $\id \otimes 1 = 1 \otimes \id$; and
  \item 2-cells generated by basic 2-cells of the form
    $\al \otimes 1$, $1 \otimes \de$, and
    $\Si_{f,g}: (f \otimes 1)(1 \otimes g) \cong (1 \otimes g)(f
    \otimes 1)$ subject to the relations
    \begin{itemize}
    \item
      $(\al \otimes 1) \cdot (\al' \otimes 1) = (\al \cdot \al')
      \otimes 1$,
    \item
      $(1 \otimes \de) \cdot (1 \otimes \de') = 1 \otimes (\de \cdot
      \de')$,
    \end{itemize}
    where $\cdot$ can be taken as either vertical or horizontal
    composition of 2-cells, together with the axioms below for 2-cells
    of the form $\Si_{f,g}$.
  \end{itemize}

  \[
  \def\objectstyle{\scriptstyle}
  \def\labelstyle{\scriptstyle}
  \xy0;/r.18pc/:
  (0,0)*+{\scriptstyle a_{1} \otimes a_{2}}="A"; 
  (38,0)*+{\scriptstyle a_{1} \otimes a_{2}'}="B";
  (0,-25)*+{\scriptstyle a_{1}' \otimes a_{2}}="C";
  (38,-25)*+{\scriptstyle a_{1}' \otimes a_{2}'}="D";
  (68,0)*+{\scriptstyle a_{1} \otimes a_{2}}="E";
  (106,0)*+{\scriptstyle a_{1} \otimes a_{2}'}="F";
  (68,-25)*+{\scriptstyle a_{1}' \otimes a_{2}}="G";
  (106,-25)*+{\scriptstyle a_{1}' \otimes a_{2}'}="H";
  {\ar@/^2pc/^{\scriptstyle 1 \otimes f_{2}} "A"; "B"};
  {\ar_{\scriptstyle 1 \otimes g_{2}} "A"; "B" };
  {\ar@/_2pc/_{\scriptscriptstyle g_{1}\otimes 1} "A"; "C" };
  {\ar@/^1pc/^{\scriptstyle f_{1}\otimes 1} "A"; "C" };
  {\ar^{\scriptscriptstyle f_{1}\otimes 1} "B"; "D" };
  {\ar_{\scriptstyle 1 \otimes g_{2} } "C"; "D" };
  {\ar_{\scriptscriptstyle g_{1} \otimes 1} "E"; "G" };
  {\ar@/^2pc/^{\scriptscriptstyle f_{1} \otimes 1} "F"; "H" };
  {\ar^{\scriptstyle 1 \otimes f_{2}} "E"; "F" };
  {\ar@/_2pc/_{\scriptstyle 1 \otimes g_{2} } "G"; "H" };
  {\ar@/_1pc/_{\scriptstyle g_{1} \otimes 1 } "F"; "H" };
  {\ar^{\scriptstyle 1 \otimes f_{2}} "G"; "H" };
  (55,-9)*{=};
  (19,5)*{\Downarrow 1 \otimes\alpha_{2}};
  (-2,-12)*{\stackrel{\Leftarrow}{\alpha_{1}\otimes 1}};
  (25,-12)*{\Downarrow \Sigma};
  (80,-12)*{\Downarrow \Sigma};
  (87,-29)*{\Downarrow 1\otimes\alpha_{2}};
  (108,-12)*{\stackrel{\Leftarrow}{\alpha_{1}\otimes 1}}
  \endxy
  \]
  
  \[
  \def\objectstyle{\scriptstyle}
  \def\labelstyle{\scriptstyle}
  \xy0;/r.20pc/:
  (0,0)*+{a_{1} \otimes a_{2}}="p00";
  (38,0)*+{a_{1} \otimes a_{2}'}="p10";
  (0,-25)*+{a_{1}' \otimes a_{2}}="p01";
  (38,-25)*+{a_{1}' \otimes a_{2}'}="p11";
  (0,-50)*+{a_{1}'' \otimes a_{2}}="p02";
  (38,-50)*+{a_{1}'' \otimes a_{2}'}="p12";
  {\ar^{1 \otimes f_{2}} "p00"; "p10" };
  {\ar^{f_{1} \otimes 1} "p10"; "p11" };
  {\ar^{h_{1} \otimes 1} "p11"; "p12" };
  {\ar_{f_{1} \otimes 1} "p00"; "p01" };
  {\ar_{h_{1} \otimes 1} "p01"; "p02" };
  {\ar_{1 \otimes f_{2}} "p02"; "p12" };
  {\ar^{1 \otimes f_{2}} "p01"; "p11" };
  (57,-25)*{=};
  (78,0)*+{a_{1} \otimes a_{2}}="q00";
  (116,0)*+{a_{1} \otimes a_{2}'}="q10";
  (78,-50)*+{a_{1}'' \otimes a_{2}}="q01";
  (116,-50)*+{a_{1}'' \otimes a_{2}'}="q11";
  {\ar^{1 \otimes f_{2}} "q00"; "q10" };
  {\ar_{1 \otimes f_{2}} "q01"; "q11" };
  {\ar_{(h_{1}f_{1}) \otimes 1} "q00"; "q01" };
  {\ar^{(h_{1}f_{1}) \otimes 1} "q10"; "q11" };
  (19,-11)*{\Downarrow \Sigma};
  (19,-36)*{\Downarrow \Sigma};
  (97,-25)*{\Downarrow \Sigma}
  \endxy
  \]
  \[
  \def\objectstyle{\scriptstyle}
  \def\labelstyle{\scriptstyle}
  \xy0;/r.20pc/:
  (0,0)*+{a_{1} \otimes a_{2}}="p00";
  (40,0)*+{a_{1} \otimes a_{2}'}="p10";
  (80,0)*+{a_{1} \otimes a_{2}''}="p20";
  (0,-25)*+{a_{1}' \otimes a_{2}}="p01";
  (40,-25)*+{a_{1}' \otimes a_{2}'}="p11";
  (80,-25)*+{a_{1}' \otimes a_{2}''}="p21";
  {\ar^{1 \otimes f_{2}} "p00"; "p10" };
  {\ar^{1 \otimes h_{2}} "p10"; "p20" };
  {\ar^{f_{1} \otimes 1} "p20"; "p21" };
  {\ar_{f_{1} \otimes 1} "p00"; "p01" };
  {\ar_{1 \otimes f_{2}} "p01"; "p11" };
  {\ar_{1 \otimes h_{2}} "p11"; "p21" };
  {\ar^{f_{1} \otimes 1} "p10"; "p11" };
  {\ar@{=} (40,-35)*{}; (40,-39)*{} };
  (0,-50)*+{a_{1} \otimes a_{2}}="q00";
  (80,-50)*+{a_{1} \otimes a_{2}''}="q10";
  (0,-75)*+{a_{1}' \otimes a_{2}}="q01";
  (80,-75)*+{a_{1}' \otimes a_{2}''}="q11";
  {\ar_{f_{1} \otimes 1} "q00"; "q01" };
  {\ar^{f_{1} \otimes 1} "q10"; "q11" };
  {\ar^{1 \otimes (h_{2}f_{2})} "q00"; "q10" };
  {\ar_{1 \otimes (h_{2}f_{2})} "q01"; "q11" };
  (20,-11)*{\Downarrow \Sigma};
  (60,-11)*{\Downarrow \Sigma};
  (40,-62)*{\Downarrow \Sigma}
  \endxy
  \]
\end{defn}

The basic properties of the Gray tensor product are summarized below.
\begin{enumerate}
\item The unit object for this monoidal structure is the terminal
  2-category, just as it is for the cartesian monoidal structure.
\item This monoidal structure is closed, with corresponding internal
  hom functor denoted $[\cB,\cC]$.  The 0-cells of $[\cB,\cC]$ are the
  2-functors $F:\cB \rightarrow \cC$, the 1-cells $\al:F \to G$ are
  the pseudonatural transformations from $F$ to $G$, and the 2-cells
  $\Ga:\al \Rightarrow \be$ are the modifications.  In particular,
  each functor $- \otimes \cB$ has a right adjoint $[\cB, -]$ and
  therefore preserves all colimits.
\item This monoidal structure is symmetric, with
  $\tau:\cA \otimes \cB \to \cB \otimes \cA$ sending $a \otimes b$ to
  $ b \otimes a$, $f \otimes 1$ to $1 \otimes f$, $1 \otimes g$ to
  $g \otimes 1$, similarly on 2-cells, and then extended on arbitrary
  1- and 2-cells in these generators by 2-functorialiy.  On the cell
  $\Si_{f,g}$, we then have
  \[
  \tau(\Si_{f,g}) = \Si_{g,f}^{-1}.
  \]
\item This monoidal structure also equips $\IICat$ with the structure
  of a monoidal model category using the canonical model structure
  developed in \cite{Lac02Quillen2}, although we will not make use of
  this structure here.
\end{enumerate}

For the reader new to the Gray tensor product, perhaps the most
helpful feature to remark upon is that the 1- and 2-cells are
generated by certain basic cells subject to relations.  In particular,
an arbitrary 1-cell of $\cA \otimes \cB$ has the form
\[
(f_{1} \otimes 1)(1 \otimes g_{1}) \cdots (f_{n} \otimes 1)(1 \otimes g_{n})
\]
for some natural number $n$ and some 1-cells $f_{i}$ of $\cA$ and
$g_{i}$ of $\cB$.

As with all closed, symmetric monoidal categories, we not only get an
isomorphism of hom-sets
\[
\IICat(\cA \otimes \cB, \cC) \cong \IICat(\cA, [\cB,\cC]),
\]
but an isomorphism of hom-objects, in this case 2-categories, of the
form
\[
[\cA \otimes \cB, \cC] \cong [\cA, [\cB,\cC]].
\]
In particular one can transport pseudonatural transformations and
modifications between the tensor product and hom-2-category side.

\begin{defn}
  \label{defn:gray-monoid}
  A \emph{Gray-monoid} is a monoid object in $\IICat$ with the Gray
  tensor product, $\otimes$.  This consists of a 2-category $\cC$, a
  2-functor
  \[
  \oplus\cn \cC \otimes \cC \to \cC,
  \]
  and an object $e$ of $\cC$ satisfying associativity and unit axioms.
\end{defn}

The Gray tensor product has another universal property relating it to
the notion of cubical functor.

\begin{defn}\label{defn:cubical}
  A normal pseudofunctor
  $F: \cA_{1} \times \cA_{2} \times \cdots \times \cA_{n} \rightarrow
  \cB$ is
  \textit{cubical} if the following condition holds:\\
  if $(f_{1}, f_{2}, \ldots, f_{n}), (g_{1}, g_{2}, \ldots, g_{n})$ is
  a composable pair of morphisms in the 2-category
  $\cA_{1} \times \cA_{2} \times \cdots \times \cA_{n}$ such that for
  all $i > j$, either $g_{i}$ or $f_{j}$ is an identity map, then the
  comparison 2-cell
  \[
  \phi: F(f_{1}, f_{2}, \ldots, f_{n})\circ F(g_{1}, g_{2}, \ldots, g_{n}) 
  \Rightarrow 
  F \Big( (f_{1}, f_{2}, \ldots, f_{n})\circ(g_{1}, g_{2}, \ldots, g_{n})
  \Big)
  \]
  is an identity. 
\end{defn}

When $n=1$, it is easy to check that a cubical functor is the same
thing as a 2-functor.  When $n=2$, we have the following
characterization, a proof of which appears in either
\cite{GPS95Coherence} or \cite{Gurski13Coherence}.

\begin{prop}\label{2var}
  A cubical functor $F: \cA_{1} \times \cA_{2} \rightarrow \cB$
  determines, and is
  uniquely determined by\\
  \begin{enumerate}
  \item For each object $a_{1} \in \textrm{ob}\cA_{1}$, a strict
    2-functor $F_{a_{1}}\cn \cA_2 \to \cB$, and for each object
    $a_{2} \in \textrm{ob}\cA_{2}$, a strict 2-functor
    $F_{a_{2}}\cn \cA_1 \to \cB$, such that for each pair of objects
    $a_{1}, a_{2}$ in $\cA_{1}, \cA_{2}$, respectively, the equation
    \[
    F_{a_{1}}(a_{2}) = F_{a_{2}}(a_{1}) := F(a_{1}, a_{2})
    \]
    holds; 
  \item For each pair of 1-cells $f_{1}: a_{1} \rightarrow a_{1}'$,
    $f_{2}: a_{2} \rightarrow a_{2}'$ in $\cA_{1}, \cA_{2}$,
    respectively, a 2-cell isomorphism
    \[
    \xy
    {\ar^{F_{a_{1}}(f_{2})} (0,0)*+{F(a_{1}, a_{2})}; (35,0)*+{F(a_{1}, a_{2}')}};
    {\ar^{F_{a_{2}'}(f_{1})} (35,0)*+{F(a_{1}, a_{2}')}; (35,-20)*+{F(a_{1}', a_{2}')} };
    {\ar_{F_{a_{2}}(f_{1})} (0,0)*+{F(a_{1}, a_{2})}; (0,-20)*+{F(a_{1}', a_{2})} };
    {\ar_{F_{a_{1}'}(f_{2})} (0,-20)*+{F(a_{1}', a_{2})}; (35,-20)*+{F(a_{1}', a_{2}')} };
    {\ar@{=>}_{\Sigma_{f_{1}, f_{2}}}^{\cong} (21,-7)*+{}="s"; "s"+(-5,-5) }
    \endxy
    \]
    which is an identity 2-cell if either $f_{1}$ or $f_{2}$ is an
    identity 1-cell;
  \end{enumerate}
  subject to the following 3 axioms for all diagrams of the form
  \[
  \xy
  {\ar@/^1pc/^{(f_{1}, f_{2})} (0,0)*+{(a_{1}, a_{2})}; (40,0)*+{(a_{1}', a_{2}')} };
  {\ar@/_1pc/_{(g_{1}, g_{2})} (0,0)*+{(a_{1}, a_{2})}; (40,0)*+{(a_{1}', a_{2}')} };
  {\ar@{=>}^{(\alpha_{1}, \alpha_{2})} (17,3)*+{}; (17,-3)*+{} };
  {\ar^{(h_{1}, h_{2})} (40,0)*+{(a_{1}', a_{2}')};
    (80,0)*+{(a_{1}'', a_{2}'')} }
  \endxy
  \]
  in $\cA_{1} \times \cA_{2}$.
  \[
  \def\objectstyle{\scriptstyle} \def\labelstyle{\scriptstyle}
  \xy0;/r.18pc/: {\ar@/^2pc/^{\scriptstyle F_{a_{1}}(f_{2})}
    (0,0)*+{\scriptstyle F(a_{1}, a_{2})}; (38,0)*+{\scriptstyle
      F(a_{1}, a_{2}')} }; {\ar_{\scriptstyle F_{a_{1}}(g_{2})}
    (0,0)*+{\scriptstyle F(a_{1}, a_{2})}; (38,0)*+{\scriptstyle
      F(a_{1}, a_{2}')} }; {\ar@/_2pc/_{\scriptscriptstyle
      F_{a_{2}}(g_{1})} (0,0)*+{\scriptstyle F(a_{1}, a_{2})};
    (0,-25)*+{\scriptstyle F(a_{1}', a_{2})} };
  {\ar@/^1pc/^{\scriptstyle F_{a_{2}}(f_{1})} (0,0)*+{\scriptstyle
      F(a_{1}, a_{2})}; (0,-25)*+{\scriptstyle F(a_{1}', a_{2})} };
  {\ar^{\scriptscriptstyle F_{a_{2}'}(f_{1})} (38,0)*+{\scriptstyle
      F(a_{1}, a_{2}')}; (38,-25)*+{\scriptstyle F(a_{1}', a_{2}')} };
  {\ar_{\scriptstyle F_{a_{1}'}(g_{2})} (0,-25)*+{\scriptstyle
      F(a_{1}', a_{2})}; (38,-25)*+{\scriptstyle F(a_{1}', a_{2}')} };
  {\ar_{\scriptscriptstyle F_{a_{2}}(g_{1})} (68,0)*+{\scriptstyle
      F(a_{1}, a_{2})}; (68,-25)*+{\scriptstyle F(a_{1}', a_{2})} };
  {\ar@/^2pc/^{\scriptscriptstyle F_{a_{2}'}(f_{1})}
    (106,0)*+{\scriptstyle F(a_{1}, a_{2}')}; (106,-25)*+{\scriptstyle
      F(a_{1}', a_{2}')} }; {\ar^{\scriptstyle F_{a_{1}}(f_{2})}
    (68,0)*+{\scriptstyle F(a_{1}, a_{2})}; (106,0)*+{\scriptstyle
      F(a_{1}, a_{2}')} }; {\ar@/_2pc/_{\scriptstyle
      F_{a_{1}'}(g_{2})} (68,-25)*+{\scriptstyle F(a_{1}', a_{2})};
    (106,-25)*+{\scriptstyle F(a_{1}', a_{2}')} };
  {\ar@/_1pc/_{\scriptstyle F_{a_{2}'}(g_{1})} (106,0)*+{\scriptstyle
      F(a_{1}, a_{2}')}; (106,-25)*+{\scriptstyle F(a_{1}', a_{2}')}
  }; {\ar^{\scriptstyle F_{a_{1}'}(f_{2})} (68,-25)*+{\scriptstyle
      F(a_{1}', a_{2})}; (106,-25)*+{\scriptstyle F(a_{1}', a_{2}')}
  }; (55,-9)*{=}; (19,3)*{\Downarrow F_{a_{1}}\alpha_{2}};
  (-2,-12)*{\stackrel{\Leftarrow}{F_{a_{2}}\alpha_{1}}};
  (25,-12)*{\Downarrow \Sigma}; (80,-12)*{\Downarrow \Sigma};
  (87,-29)*{\Downarrow F_{a_{1}'}\alpha_{2}};
  (108,-12)*{\stackrel{\Leftarrow}{F_{a_{2}'}\alpha_{1}}}
  \endxy
  \]
  \[
  \def\objectstyle{\scriptstyle}
  \def\labelstyle{\scriptstyle}
  \xy0;/r.20pc/:
  {\ar^{F_{a_{1}}(f_{2})} (0,0)*+{F(a_{1}, a_{2})}; (38,0)*+{F(a_{1}, a_{2}')} };
  {\ar^{F_{a_{2}'}(f_{1})} (38,0)*+{F(a_{1}, a_{2}')}; (38,-25)*+{F(a_{1}', a_{2}')} };
  {\ar^{F_{a_{2}'}(h_{1})} (38,-25)*+{F(a_{1}', a_{2}')}; (38,-50)*+{F(a_{1}'', a_{2}')} };
  {\ar_{F_{a_{2}}(f_{1})} (0,0)*+{F(a_{1}, a_{2})}; (0,-25)*+{F(a_{1}', a_{2})} };
  {\ar_{F_{a_{2}}(h_{1})} (0,-25)*+{F(a_{1}', a_{2})}; (0,-50)*+{F(a_{1}'', a_{2})} };
  {\ar_{F_{a_{1}''}(f_{2})} (0,-50)*+{F(a_{1}'', a_{2})}; (38,-50)*+{F(a_{1}'', a_{2}')} };
  {\ar^{F_{a_{1}'}(f_{2})} (0,-25)*+{F(a_{1}', a_{2})}; (38,-25)*+{F(a_{1}', a_{2}')} };
  (57,-25)*{=};
  {\ar^{F_{a_{1}}(f_{2})} (78,0)*+{F(a_{1}, a_{2})}; (116,0)*+{F(a_{1}, a_{2}')} };{\ar_{F_{a_{1}''}(f_{2})} (78,-50)*+{F(a_{1}'', a_{2})}; (116,-50)*+{F(a_{1}'', a_{2}')} };
  {\ar_{\scriptstyle F_{a_{2}}(h_{1}f_{1})} (78,0)*+{F(a_{1}, a_{2})}; (78,-50)*+{F(a_{1}'', a_{2})} };
  {\ar^{F_{a_{2}'}(h_{1}f_{1})} (116,0)*+{F(a_{1}, a_{2}')}; (116,-50)*+{F(a_{1}'', a_{2}')} };
  (19,-11)*{\Downarrow \Sigma};
  (19,-36)*{\Downarrow \Sigma};
  (97,-25)*{\Downarrow \Sigma}
  \endxy
  \]
  \[
  \def\objectstyle{\scriptstyle}
  \def\labelstyle{\scriptstyle}
  \xy0;/r.20pc/:
  {\ar^{F_{a_{1}}(f_{2})} (0,0)*+{F(a_{1}, a_{2})}; (40,0)*+{F(a_{1}, a_{2}')} };
  {\ar^{F_{a_{1}}(h_{2})} (40,0)*+{F(a_{1}, a_{2}')}; (80,0)*+{F(a_{1}, a_{2}'')} };
  {\ar^{F_{a_{2}''}(f_{1})} (80,0)*+{F(a_{1}, a_{2}'')}; (80,-25)*+{F(a_{1}', a_{2}'')} };
  {\ar_{F_{a_{2}}(f_{1})} (0,0)*+{F(a_{1}, a_{2})}; (0,-25)*+{F(a_{1}', a_{2})} };
  {\ar^{F_{a_{1}'}(f_{2})} (0,-25)*+{F(a_{1}', a_{2})}; (40,-25)*+{F(a_{1}', a_{2}')} };
  {\ar^{F_{a_{1}'}(h_{2})} (40,-25)*+{F(a_{1}', a_{2}')}; (80,-25)*+{F(a_{1}', a_{2}'')} };
  {\ar^{F_{a_{2}'}(f_{1})} (40,0)*+{F(a_{1}, a_{2}')}; (40,-25)*+{F(a_{1}', a_{2}')} };
  {\ar@{=} (40,-35)*{}; (40,-39)*{} };
  {\ar_{F_{a_{2}}(f_{1})} (0,-50)*+{F(a_{1}, a_{2})}; (0,-75)*+{F(a_{1}', a_{2})} };
  {\ar^{F_{a_{2}''}(f_{1})} (80,-50)*+{F(a_{1}, a_{2}'')}; (80,-75)*+{F(a_{1}', a_{2}'')} };
  {\ar^{F_{a_{1}}(h_{2}f_{2})} (0,-50)*+{F(a_{1}, a_{2})}; (80,-50)*+{F(a_{1}, a_{2}'')} };
  {\ar_{F_{a_{1}'}(h_{2}f_{2})} (0,-75)*+{F(a_{1}', a_{2})}; (80,-75)*+{F(a_{1}', a_{2}'')} };
  (20,-11)*{\Downarrow \Sigma};
  (60,-11)*{\Downarrow \Sigma};
  (40,-62)*{\Downarrow \Sigma}
  \endxy
  \]
\end{prop}

The following proposition is easy to check, and appears in
\cite{Gurski13Coherence}.

\begin{prop}
  There is a multicategory $\IICat_{c}$ whose objects are
  2-categories and for which the set
  \[
  \IICat_{c}(\cA_{1}, \cA_{2}, \ldots, \cA_{n}; \cB)
  \]
  consists of the cubical functors
  $\cA_{1} \times \cdots \times \cA_{n} \rightarrow \cB$.
\end{prop}

\begin{thm}\label{cubicalmulticat}
  Let $\cA$, $\cB$, and $\cC$ be 2-categories.  There is a cubical
  functor
  \[
  c:\cA \times \cB \rightarrow \cA \otimes \cB,
  \]
  natural in $\cA$ and $\cB$, such that composition with $c$ induces
  an isomorphism
  \[
  \IICat_{c}(\cA, \cB; \cC) \cong \IICat(\cA \otimes \cB, \cC).
  \]
\end{thm}
\begin{proof}[Sketch proof, see \cite{Gurski13Coherence}]
  We define $c$ using \cref{2var}.  We define the 2-functor $c_{a}$
  by\\
  \[
  \begin{array}{c}
    c_{a}(b) = (a,b) \\
    c_{a}(f) = (1_{a}, f) \\
    c_{a}(\alpha) = (1_{1_{a}}, \alpha);
  \end{array}
  \]
  the 2-functor $c_{b}$ is defined similarly.  The 2-cell isomorphism
  $\Sigma_{f,g}$ is the same $\Sigma_{f,g}$ that is part of the data for
  $\cA \otimes \cB$.
  
  To prove that this cubical functor has the claimed universal
  property, assume that $F: \cA \times \cB \rightarrow \cC$ is a
  cubical functor.  We define a strict 2-functor
  $\overline{F}: \cA \otimes \cB \rightarrow \cC$ by the following
  formulas.\\
  \[
  \begin{array}{c}
    \overline{F}(a,b) = F(a,b) \\
    \overline{F}(f,1) = F_{b}(f) \\
    \overline{F}(1,g) = F_{a}(g) \\
    \overline{F}(\alpha, 1) = F_{b}(\alpha) \\
    \overline{F}(1, \beta) = F_{a}(\beta) \\
    \overline{F}(\Sigma^{\cA \otimes \cB}_{f,g}) = \Sigma^{F}_{f,g}
  \end{array}
  \]
\end{proof}

For our discussion of permutative Gray-monoids later, we will need the
notion of an opcubical functor.
\begin{defn}\label{defn:opcubical}
  A pseudofunctor
  $F: \cA_{1} \times \cA_{2} \times \cdots \times \cA_{n} \rightarrow
  \cB$ is
  \textit{opcubical} if the following condition holds:\\
  if $(f_{1}, f_{2}, \ldots, f_{n}), (g_{1}, g_{2}, \ldots, g_{n})$ is
  a composable pair of morphisms in the 2-category
  $\cA_{1} \times \cA_{2} \times \cdots \times \cA_{n}$ such that for
  all $i < j$, either $g_{i}$ or $f_{j}$ is an identity map, then the
  comparison 2-cell
  \[
  \phi: F(f_{1}, f_{2}, \ldots, f_{n})\circ F(g_{1}, g_{2}, \ldots,
  g_{n})\Rightarrow F \Big( (f_{1}, f_{2}, \ldots, f_{n})\circ(g_{1}, g_{2}, \ldots, g_{n})
  \Big)
  \]
  is an identity.
\end{defn}

As the difference between a cubical functor and an opcubical one is
merely a matter of ordering, we immediately get the following lemma.
\begin{lem}\label{cubtoopcub}
There is a bijection, natural in all variables, between the set of cubical functors $\cA \times \cB \to \cC$ and the set of opcubical functors $\cB \times \cA \to \cC$.  In particular, if the pseudofunctor $F:\cA \times \cB \to \cC$ is cubical, then
\[
\cB \times \cA \cong \cA \times \cB \stackrel{F}{\longrightarrow} \cC
\]
is opcubical.
\end{lem}

If $F:\cA \times \cB \rightarrow \cC$ is cubical, we can produce
another opcubical functor $F^{*}$, this time with source
$\cA \times \cB$, by defining
\[
F^{*}(f,g) = F(f,1) F(1,g)
\]
and replacing the necessary structure 2-cells with their inverses.
This process was introduced in \cite{GPS95Coherence} and is called
nudging.  One can check that $F^{*} \cong F$ as pseudofunctors in the
2-category of 2-categories, pseudofunctors, and icons
\cite{Lac10Icons}, and therefore we obtain an isomorphism between the
set of cubical functors $\cA \times \cB \rightarrow \cC$ and the set
of opcubical functors $\cA \times \cB \rightarrow \cC$.  This
procedure, together with the lemma above, is one method for giving the
symmetry isomorphism for the Gray tensor product.  Moreover, one can
show that the map from the Gray tensor product to the cartesian one
$\cA \otimes \cB \to \cA \times \cB$ induced by the identity 2-functor
$\cA \times \cB \to \cA \times \cB$ viewed as a cubical functor is a
map of symmetric monoidal structures. It is not hard to check, though,
that the universal cubical functor
$c:\cA \times \cB \to \cA \otimes \cB$ is only symmetric up to an
invertible icon.

\begin{lem}\label{cubtoopcub2}
  Let $F: \cA \times \cB \to \cC$ be a cubical functor, and let
  $\overline{F}:\cA \otimes \cB \to \cC$ be the associated 2-functor.
  Then the cubical functor associated to the 2-functor
  \[
  \cB \otimes \cA \stackrel{\tau_{\otimes}}{\rtarr} \cA \otimes \cB \stackrel{\overline{F}}{\rtarr} \cC
  \]
  is $F^{*} \tau_{\times}$, where $F^{*}$ is the opcubical functor
  associated to $F$ as defined above.
\end{lem}
\begin{proof}\proofof{cubtoopcub2}
  We must show that
  $\overline{F^{*} \tau_{\times}} = \overline{F} \tau_{\otimes}$ by
  checking that the two agree on generating cells, including cells of
  the form $\Si_{f,g}$.  On generating cells arising directly from one
  of the copies of $\cA$, this is obvious.  For the $\Si_{f,g}$'s, we
  have
  \[
  \begin{array}{rcl}
    \overline{F} \tau_{\otimes}(\Si_{f,g}) & = & \overline{F}((\Si_{g,f})^{-1}) \\
    & = & (\overline{F}(\Si_{g,f}))^{-1} \\
    & = & (\Si_{g,f}^{F})^{-1} \\
    & = & \overline{F^{*} \tau_{\times}}(\Si_{f,g})
  \end{array}
  \]
  by the definition of $\tau_{\otimes}$, 2-functors of the form
  $\overline{G}$ for a cubical functor $G$, and the opcubical functor
  $F^{*}$.
\end{proof}

\begin{notn}
  \label{notn:sum-of-one-cells-gray-monoid}
  Given 1-cells $f, g$ in a Gray-monoid, we let $f \oplus g$ denote
  $\ol{\oplus}(f,g)$, where $\ol{\oplus}$ is the cubical functor
  associated to $\oplus$.  Concretely,
  $f \oplus g = (\id \oplus g )\circ (f \oplus \id)$.  Similarly, for
  1-cells $f_1, \ldots, f_q$ we let
  $\oplus_i f_i = \ol{\oplus}( f_1, \ldots, f_q )$, where now
  $\ol{\oplus}$ is the cubical functor associated to the iterated sum
  $\oplus \cn \cC^{\otimes q} \to \cC$.
\end{notn}

The next two subsections will study two different strict notions of
what one might consider a symmetric monoidal
2-category.  The first is that of a permutative Gray-monoid which we
will show is equivalent to the notion of quasi-strict symmetric
monoidal 2-category introduced in \cite{Sch2011Classification}.  We
believe that this repackaging of the definition sheds conceptual
light, and helps to motivate the second definition, that of a
permutative 2-category.  While permutative 2-categories are not
equivalent to permutative Gray-monoids in the categorical (or
bicategorical, or even tricategorical) sense, we will later show that
they do have the same homotopy theory.

\subsection{Permutative Gray-monoids}
\label{sec:permutative-Gray-monoids}

We begin with a reminder of the definition of a permutative category.

\begin{defn}\label{defn:perm-cat}
  A \textit{permutative category} $C$ consists of a strict monoidal
  category $(C, \oplus, e)$ together with a natural isomorphism,
  \[
  \xy
  (0,0)*+{C \times C}="00";
  (25,0)*+{C \times C}="10";
  (12.5,-10)*+{C}="01";
  {\ar^{\tau} "00"; "10"};
  {\ar^{\oplus} "10"; "01"};
  {\ar_{\oplus} "00"; "01"};
  (12.5,-4)*{\Anglearrow{40} \beta}
  \endxy
  \]
  where $\tau \cn C \times C \to C \times C$ is the symmetry
  isomorphism in $\Cat$, such that the following axioms hold for all
  objects $x,y,z$ of $C$.
  \begin{itemize}
  \item $\beta_{y,x} \beta_{x,y} = \id_{x \oplus y}$
  \item $\beta_{e,x} = \id_{x} = \beta_{x,e}$
  \item $\beta_{x, y \oplus z} = (y \oplus \beta_{x,z}) \circ (\beta_{x,y} \oplus z)$
  \end{itemize}
\end{defn}

\begin{rmk}
  It is relatively easy to check that this definition is logically
  equivalent to the definition of a symmetric monoidal category with
  underlying monoidal structure strict.
\end{rmk}

Each of the axioms for a permutative category can be expressed in a
purely diagrammatic form and thus studied in other contexts.  We do so
now in the context of 2-categories using the Gray tensor product.

\begin{defn}\label{defn:pgm}
  A \textit{permutative Gray-monoid} $\cC$ consists of a Gray-monoid
  $(\cC, \oplus, e)$ together with a 2-natural isomorphism,
  \[
  \xy
  (0,0)*+{\cC \otimes \cC}="00";
  (25,0)*+{\cC \otimes \cC}="10";
  (12.5,-10)*+{\cC}="01";
  {\ar^{\tau} "00"; "10"};
  {\ar^{\oplus} "10"; "01"};
  {\ar_{\oplus} "00"; "01"};
  (12.5,-4)*{\Anglearrow{40}\beta}
  \endxy
  \]
  where $\tau \cn \cC \otimes \cC \to \cC \otimes \cC$ is the symmetry
  isomorphism in $\IICat$ for the Gray tensor product, such that the
  following axioms hold.
  \begin{itemize}
  \item The following pasting diagram is equal to the identity
    2-natural transformation for the 2-functor $\oplus$.
    \[
    \xy
    (0,0)*+{\cC \otimes \cC}="00";
    (25,0)*+{\cC \otimes \cC}="10";
    (50,0)*+{\cC \otimes \cC}="20";
    (25,-15)*+{\cC}="11";
    {\ar^{\tau} "00"; "10"};
    {\ar^{\tau} "10"; "20"};
    {\ar_{\oplus} "00"; "11"};
    {\ar_{\oplus} "10"; "11"};
    {\ar^{\oplus} "20"; "11"};
    {\ar@/^1.5pc/^1 "00"; "20"};
    (14.5,-4)*{\scriptstyle \Anglearrow{40} \beta};
    (35.5,-4)*{\scriptstyle \Anglearrow{0} \beta}
    \endxy
    \]
  \item The following pasting diagram is equal to the identity
    2-natural transformation for the canonical isomorphism
    $1 \otimes \cC \cong \cC$.
    \[ \xy
    (0,0)*+{1 \otimes \cC}="00";
    (30,0)*+{\cC \otimes \cC}="10";
    (60,0)*+{\cC \otimes \cC}="20";
    (30,-12)*+{\cC}="11";
    {\ar^{e \otimes \id} "00"; "10"};
    {\ar^{\tau} "10"; "20"};
    {\ar^{\oplus} "20"; "11"};
    {\ar_{\oplus} "10"; "11"};
    {\ar_{\cong} "00"; "11"};
    (17,-4)*{\scriptstyle =};     (40, -4)*{\scriptstyle \Anglearrow{0} \beta}
    \endxy
    \]
  \item The following equality of pasting diagrams holds where we have abbreviated the tensor product to concatenation when labeling 1- or 2-cells.
    \[ 
    \xy
    (0,0)*+{\cC^{\otimes 3}}="00";
    (18,0)*+{\cC^{\otimes 3}}="10";
    (36,0)*+{\cC^{\otimes 3}}="20";
    (54,0)*+{\cC^{\otimes 2}}="30";
    (27,-15)*+{\cC^{\otimes 2}}="11";
    (18,-30)*+{\cC^{\otimes 2}}="12";
    (54,-22)*+{\cC}="33";
    (69,0)*+{\cC^{\otimes 3}}="40";
    (87,0)*+{\cC^{\otimes 3}}="50";
    (105,0)*+{\cC^{\otimes 3}}="60";
    (123,0)*+{\cC^{\otimes 2}}="70";
    (87,-30)*+{\cC^{\otimes 2}}="52";
    (123,-22)*+{\cC}="73";
    (105,-17)*+{\cC^{\otimes 2}}="63";
    {\ar^{\scriptstyle \tau \id} "00"; "10"};
    {\ar^{\scriptstyle \tau \id} "40"; "50"};
    {\ar^{\scriptstyle \id \tau } "10"; "20"};
    {\ar^{\scriptstyle \id \tau } "50"; "60"};
    {\ar^{\scriptstyle \oplus \id} "20"; "30"};
    {\ar^{\scriptstyle \oplus \id} "60"; "70"};
    {\ar^{\scriptstyle \oplus} "30"; "33"};
    {\ar^{\scriptstyle \oplus} "70"; "73"};
    {\ar_{\scriptstyle \oplus \id} "00"; "12"};
    {\ar_{\scriptstyle \oplus \id} "40"; "52"};
    {\ar_{\scriptstyle \oplus} "12"; "33"};
    {\ar_{\scriptstyle \oplus} "52"; "73"};
    {\ar_{\scriptstyle \id \oplus} "00"; "11"};
    {\ar^{\scriptstyle \tau} "11"; "30"};
    {\ar_{\scriptstyle \oplus} "11"; "33"};
    {\ar^{\scriptstyle \oplus \id} "50"; "52"};
    {\ar_{\scriptstyle \id \oplus} "50"; "63"};
    {\ar^{\scriptstyle \id \oplus} "60"; "63"};
    {\ar^{\scriptstyle \oplus} "63"; "73"};
    (27,-7.5)*{=}; (18,-18)*{=}; (114,-7.5)*{=}; (96,-20)*{=}; (61.5,-15)*{=};
    (45,-13)*{\scriptstyle \Anglearrow{40} \beta}; (80,-8)*{\scriptstyle \Anglearrow{40} \beta \id}; (99,-5)*{\scriptstyle \Anglearrow{40} \id \beta};
    \endxy
    \]
  \end{itemize}
\end{defn}

\begin{rmk}\label{rmk:no2monadforpgm}
  Although the definition of permutative Gray-monoid is analogous to
  the definition of permutative category, there is an important
  difference.  Permutative categories can be described as the algebras
  for a 2-monad on $\Cat$, but no such description can be made for
  permutative Gray-monoids.  This is because a 2-monad for permutative
  Gray-monoids would have as its underlying 2-category that of
  2-categories, 2-functors, and 2-natural transformations.  But there
  is no way to extend the Gray tensor product as a functor of
  categories $\otimes: \IICat \times \IICat \to \IICat$ to a 2-functor
  of 2-categories with the same objects and 1-cells but 2-natural
  transformations as 2-cells: one can easily verify that there is no
  way to make a Gray tensor product of a pair of 2-natural
  transformations itself into a 2-natural transformation, it is only
  possible to produce a pseudonatural one.  Therefore such a 2-monad
  does not exist.  Nevertheless, we borrow heavily from the strategies
  in 2-dimensional algebra in dealing with permutative Gray-monoids,
  using such notions as strict and (op)lax functors.
\end{rmk}

\begin{rmk}
  The definition of permutative Gray-monoid makes no mention of
  permutations of more than 3 objects, so the reader might wonder
  about the existence and uniqueness of something deserving to be
  called $\beta \oplus \beta: x \oplus y \oplus z \oplus w \iso y
  \oplus x \oplus w \oplus z$.  \it{A priori}, there are two such 1-cells,
  $(\id \oplus \beta) \circ (\beta \oplus \id)$ and $(\beta \oplus \id)
  \circ (\id \oplus \beta)$.  However in \cref{qs3} we will see that
  $\Sigma_{\beta,\beta}$ is the identity, so there is a unique 1-cell
  isomorphism for any permutation of objects.
\end{rmk}

\begin{defn}
  \label{defn:strict-functor-gray-mon}
  A \textit{strict functor} $F:\cC \to \cD$ of permutative
  Gray-monoids is a 2-functor $F:\cC \to \cD$ of the underlying
  2-categories satisfying the following conditions.
  \begin{itemize}
  \item $F(e_\cC) = e_\cD$, so that $F$ strictly preserves the unit
    object.
  \item The diagram
    \[
    \xy
    (0,0)*+{\cC \otimes \cC}="00";
    (30,0)*+{\cD \otimes \cD}="10";
    (0,-15)*+{\cC}="01";
    (30,-15)*+{\cD}="11";
    {\ar^{F \otimes F} "00"; "10"};
    {\ar^{\oplus_{\cD}} "10"; "11"};
    {\ar_{\oplus_{\cC}} "00"; "01"};
    {\ar_{F} "01"; "11"};
    \endxy
    \]
    commutes, so that $F$ strictly preserves the sum.
  \item The equation
    \[
    \beta^{\cD} * (F \otimes F) = F * \beta^{\cC}
    \]
    holds, so that $F$ strictly preserves the symmetry.  This equation
    is equivalent to requiring that
    \[
    \beta^{\cD}_{Fx,Fy} = F(\beta^{\cC}_{x,y})
    \]
    as 1-cells from $Fx \oplus Fy = F(x \oplus y)$ to
    $Fy \oplus Fx = F(y \oplus x)$.
  \end{itemize}
\end{defn}

\begin{prop}
  \label{prop:defn-PGM}
  There is a category $\PGM$ of permutative Gray-monoids and strict
  functors between them. The underlying 2-category functor
  $\PGM \to \IICat$ is monadic in the usual, 1-categorical sense.
\end{prop}

\begin{defn}\label{defn:weak-equiv-pgm}
  Let $\cC, \cD$ be a pair of permutative Gray-monoids.  A strict
  functor of permutative Gray-monoids $F \cn \cC \to \cD$ is a
  \emph{weak equivalence} if the underlying 2-functor is a weak
  equivalence of 2-categories.  We let $(\PGM, \cW)$ denote the
  relative category of permutative Gray-monoids with weak
  equivalences,
\end{defn}

We also have the notion of lax functor.

\begin{defn}\label{laxfundef}
  A \textit{lax functor} $F:\cC \to \cD$ of permutative Gray-monoids
  consists of
  \begin{itemize}
  \item a 2-functor $F:\cC \to \cD$ between the underlying
    2-categories,
  \item a 1-cell $\tha_{0}:e_{\cD} \to F(e_{\cC})$, and
  \item a 2-natural transformation
    \[
    \xy
    (0,0)*+{\cC \otimes \cC}="00";
    (30,0)*+{\cD \otimes \cD}="10";
    (0,-15)*+{\cC}="01";
    (30,-15)*+{\cD}="11";
    {\ar^{F \otimes F} "00"; "10"};
    {\ar^{\oplus_{\cD}} "10"; "11"};
    {\ar_{\oplus_{\cC}} "00"; "01"};
    {\ar_{F} "01"; "11"};
    (15,-6)*{\Downarrow \tha}
    \endxy
    \]
  \end{itemize}
  subject to the requirement that the following diagrams commute for
  all objects $x,y,z \in \cC$.
  \[
  \xy
  (0,0)*+{e \oplus Fx}="00";
  (30,0)*+{Fe \oplus Fx}="10";
  (30,-10)*+{F(e \oplus x)}="11";
  (30,-20)*+{Fx}="12";
  {\ar^{\tha_{0} \oplus 1} "00"; "10"};
  {\ar^{\tha} "10"; "11"};
  {\ar@{=} "11"; "12"};
  {\ar@{=} "00"; "12"};
  (50,0)*+{Fx \oplus e}="20";
  (80,0)*+{Fx \oplus Fe}="30";
  (80,-10)*+{F(x \oplus e)}="31";
  (80,-20)*+{Fx}="32";
  {\ar^{1 \oplus \tha_{0}} "20"; "30"};
  {\ar^{\tha} "30"; "31"};
  {\ar@{=} "31"; "32"};
  {\ar@{=} "20"; "32"};
  (20,-30)*+{Fx \oplus Fy \oplus Fz}="55";
  (60,-30)*+{F(x \oplus y) \oplus Fz}="65";
  (20,-45)*+{Fx \oplus F(y \oplus z)}="56";
  (60,-45)*+{F(x \oplus y \oplus z)}="66";
  {\ar^{\tha \oplus 1} "55"; "65"};
  {\ar^{\tha} "65"; "66"};
  {\ar_{1 \oplus \tha} "55"; "56"};
  {\ar_{\tha} "56"; "66"};
  (20,-60)*+{Fx \oplus Fy }="85";
  (60,-60)*+{F(x \oplus y) }="95";
  (20,-75)*+{Fy \oplus Fx}="86";
  (60,-75)*+{F(y \oplus x )}="96";
  {\ar^{\tha } "85"; "95"};
  {\ar^{F(\beta)} "95"; "96"};
  {\ar_{\be} "85"; "86"};
  {\ar_{\tha} "86"; "96"};
  \endxy
  \]
\end{defn}

Just as one can compose lax monoidal functors between monoidal
categories, it is possible to compose lax functors between permutative
Gray-monoids.  If $(F, \tha_0, \tha)$ is a lax functor $\cC \to \cD$
and $(G, \psi_0, \psi)$ is a lax functor $\cD \to \cA$, then the
composite $GF$ is given the structure of a lax functor with 1-cell
\[
e \stackrel{\psi_0}{\rtarr} Ge \stackrel{G\tha_0}{\rtarr} GFe
\]
and 2-natural transformation with components
\[
GFx \oplus GFy  \stackrel{\psi}{\rtarr} G(Fx \oplus Fy) \stackrel{G\tha}{\rtarr} GF(x \oplus y).
\]

There are two additional variants one might consider.
\begin{defn}\label{defn:ps-oplax-functor}
  A \textit{pseudofunctor} $F:\cC \to \cD$ between permutative
  Gray-monoids is a lax functor in which both the structure 1-cell for
  the unit object and the 2-natural transformation are isomorphisms.
  An \textit{oplax functor} $F:\cC \to \cD$ consists of
  \begin{itemize}
  \item a 2-functor $F:\cC \to \cD$ between the underlying
    2-categories,
  \item a 1-cell $\tha_{0}:F(e_{\cC}) \to e_{\cD}$, and
  \item a 2-natural transformation $\tha$ with components
    $F(x \oplus y) \to Fx \oplus Fy$
  \end{itemize}
  subject to axioms such as those in \cref{laxfundef} with all arrows
  reversed.
\end{defn}

It is clear that pseudofunctors are closed under composition, and one
can define the composition of oplax functors in much the same way as
it is defined for lax ones.

\begin{rmk}
  A pseudofunctor in this sense is not the weakest possible notion, as
  it has an underlying 2-functor.
\end{rmk}

\begin{defn}\label{defn:normal-functor}
  A \textit{normal functor} $F:\cC \to \cD$ of permutative
  Gray-monoids is a functor (lax, oplax, or pseudo) for which
  $\tha_{0}$ is the identity.
\end{defn}

Since the underlying morphism of a functor (of any kind) between permutative Gray-monoids is itself a 2-functor, the composite of two normal functors of the
same kind will be another normal functor.  Note also that every strict
functor is already normal.

\begin{prop}\label{prop:PGM-variants}
  There are categories of permutative Gray-monoids with lax functors,
  oplax functors, and pseudofunctors; we denote these with the
  subscripts \emph{l, op,} and \emph{ps} respectively.  We also have
  categories whose maps are the normal variants of each, denoted with
  respective subscripts \emph{nl, nop,} and \emph{nps}.  We have
  canonical inclusions
\[
\xy
(0,0)*+{\PGM}="0";
(30,0)*+{\PGMnps}="1";
(60,0)*+{\PGMps}="2";
(50,10)*+{\PGMnl}="3";
(50,-10)*+{\PGMnop}="4";
(80,10)*+{\PGMl}="5";
(80,-10)*+{\PGMop}="6";
{\ar "0"; "1" };
{\ar "1"; "2" };
{\ar "1"; "4" };
{\ar "1"; "3" };
{\ar "2"; "5" };
{\ar "2"; "6" };
{\ar "3"; "5" };
{\ar "4"; "6" };
\endxy
\]
which commute with the forgetful functors to $\IICat$.
\end{prop}

We return now to a discussion of plain 2-categories without any
monoidal structure in order to prove the equivalence between
quasi-strict symmetric monoidal 2-categories and permutative
Gray-monoids.

\begin{defn}\label{defn:strict-transformation}
  Let $\cA$ and $\cB$ be 2-categories and let $F,G\cn \cA \to \cB$ be
  pseudofunctors. A pseudonatural transformation $\al \cn F \impl G$
  is a \textit{strict transformation} if for each 1-cell $f \cn a \to
  b$, the 2-cell isomorphism
  \[
  \al_{f}: Gf * \al_{a}  \cong \al_{b} * Ff
  \]
  is the identity 2-cell.
\end{defn}

\begin{rmk}
  Note that when $F,G$ in the definition above are strict 2-functors rather
  than pseudofunctors, strict naturality is what is usually
  called 2-naturality in the 2-categorical literature
  \cite{KS74Review}.  It will become clear below why we have
  introduced this additional layer of terminology.
\end{rmk}

\begin{lem}\label{stricttranslem}
  Let $\cA$, $\cB$, and $\cC$ be 2-categories and let $F, G \cn \cA
  \times \cB \to \cC$ be a pair of pseudofunctors.
  \begin{enumerate}
  \item Assume that $F, G$ are cubical, and
    $\overline{F}, \overline{G} \cn \cA \otimes \cB \to \cC$ are the
    associated 2-functors.  Then there is an isomorphism between the
    set of pseudonatural transformations $\al \cn F \impl G$ and the
    set of pseudonatural transformations
    $\overline{\al} \cn \overline{F} \impl \overline{G}$.
    Furthermore, the following conditions are equivalent:
    \begin{itemize}
    \item $\overline{\al}$ is a 2-natural transformation,
    \item $\al$ is a strict transformation, and
    \item the components $\al_{f,1}, \al_{1,g}$ are the identity for
      all 1-cells $f$ in $\cA$ and $g$ in $\cB$.
    \end{itemize}
  \item Assume that $F, G$ are opcubical, and
    $\overline{F^{*}}, \overline{G^{*}} \cn \cA \otimes \cB \to \cC$
    are the associated 2-functors.  Then there is an isomorphism
    between the set of pseudonatural transformations
    $\al \cn F \impl G$ and the set of pseudonatural transformations
    $\overline{\al} \cn \overline{F^{*}} \impl \overline{G^{*}}$.
    Furthermore, the following conditions are equivalent:
    \begin{itemize}
    \item $\overline{\al}$ is a 2-natural transformation,
    \item $\al$ is a strict transformation, and
    \item the components $\al_{f,1}, \al_{1,g}$ are the identity for
      all 1-cells $f$ in $\cA$ and $g$ in $\cB$.
    \end{itemize}
  \item Assume that $F$ is cubical and $G$ is opcubical, and
    $\overline{F}, \overline{G^{*}} \cn \cA \otimes \cB \to \cC$ are
    the associated 2-functors.  Then there is an isomorphism between
    the set of pseudonatural transformation $\al \cn F \impl G$ and
    the set of pseudonatural transformations
    $\overline{\al} \cn \overline{F} \impl \overline{G^{*}}$.
    Furthermore, the following conditions are equivalent:
    \begin{itemize}
    \item $\overline{\al}$ is a 2-natural transformation and
    \item the components $\al_{f,1}, \al_{1,g}$ are the identity for
      all 1-cells $f$ in $\cA$ and $g$ in $\cB$.
    \end{itemize}
  \end{enumerate}
\end{lem}
\begin{proof}\proofof{stricttranslem}
  For the first two variants above, define
  $\overline{\al}_{a \otimes b} = \al_{a,b}$ for the 0-cell components
  and
  $\overline{\al}_{f \otimes \id} = \al_{f,\id}, \overline{\al}_{\id
    \otimes g} = \al_{\id,g}$
  for the 1-cell components on generators.  The 1-cell components for
  a general 1-cell are forced by the axioms, and it is easy to check
  the pseudonaturality axioms as well as the equivalent conditions
  stated above.  The same definitions work for the third variant, only
  note that $\al_{f,g}$ is not forced to be the identity even if both
  $\al_{f,1}, \al_{1,g}$ are since it must satisfy the axiom shown
  below where the single unlabeled isomorphism is from the
  pseudofunctoriality of $G$ and is not necessarily the identity as
  $G$ is only opcubical and not cubical.
  \[
  \xy
  (0,0)*+{F(a,b)}="00";
  (30,0)*+{G(a,b)}="10";
  (15,-14)*+{F(a',b)}="01";
  (45,-14)*+{G(a',b)}="11";
  (0,-28)*+{F(a',b')}="02";
  (30,-28)*+{G(a',b')}="12";
  (70,0)*+{F(a,b)}="60";
  (100,0)*+{G(a,b)}="70";
  (115,-14)*+{G(a',b)}="71";
  (70,-28)*+{F(a',b')}="62";
  (100,-28)*+{G(a',b')}="72";
  {\ar^{\al_{a,b}} "00"; "10"};
  {\ar^{G(f,\id)} "10"; "11"};
  {\ar^{G(\id,g)} "11"; "12"};
  {\ar_{F(f,g)} "00"; "02"};
  {\ar_{\al_{a',b'}} "02"; "12"};
  {\ar^{F(f,\id)} "00"; "01"};
  {\ar^{F(\id,g)} "01"; "02"};
  {\ar^{\al_{a',b}} "01"; "11"};
  {\ar^{\al_{a,b}} "60"; "70"};
  {\ar^{G(f,\id)} "70"; "71"};
  {\ar^{G(\id,g)} "71"; "72"};
  {\ar_{F(f,g)} "60"; "62"};
  {\ar_{\al_{a',b'}} "62"; "72"};
  {\ar_{G(f,g)} "70"; "72"};
  (20,-7)*+{=}; (20,-21)*+{=}; (5,-14)*+{=}; (58,-14)*+{=};
  (83,-14)*{\Downarrow \al_{f,g}}; (105,-14)*{\cong}
  \endxy
  \]
\end{proof}

\begin{prop}\label{qs3}
  Axioms (QS1) and (QS2) imply (QS3) in the presence of the others.
\end{prop}
\begin{proof}\proofof{qs3}
  We must show that $\Sigma_{f, \beta}$ and $\Sigma_{\beta,g}$ are
  identities.  We do this for $\Sigma_{\beta,g}$ below; the other case
  is similar.  By the modification axiom for $R$, we have the
  indicated equality of pasting diagrams.
  \[
  \def\objectstyle{\scriptstyle}
  \def\labelstyle{\scriptstyle}
  \xy
  (0,0)*+{a \oplus b \oplus c}="00";
  (30,0)*+{a \oplus b \oplus c'}="10";
  (0,-15)*+{b \oplus a \oplus c}="01";
  (0,-30)*+{b \oplus c \oplus a}="02";
  (30,-30)*+{b \oplus c' \oplus a}="12";
  (30,-15)*+{b \oplus a \oplus c'}="11";
  (70,0)*+{a \oplus b \oplus c}="60";
  (100,0)*+{a \oplus b \oplus c'}="70";
  (70,-15)*+{b \oplus a \oplus c}="61";
  (70,-30)*+{b \oplus c \oplus a}="62";
  (100,-30)*+{b \oplus c' \oplus a}="72";
  {\ar^{\id \oplus g} "00"; "10"};
  {\ar_{\be \oplus \id} "10"; "11"};
  {\ar_{\id \oplus \be} "11"; "12"};
  {\ar_{\be \oplus \id} "00"; "01"};
  {\ar_{\id \oplus \be} "01"; "02"};
  {\ar_{\id \oplus g \oplus \id} "02"; "12"};
  {\ar^{\id \oplus g} "01"; "11"};
  {\ar@/^2.5pc/^{\be} "10"; "12"};
  {\ar^{\id \oplus g} "60"; "70"};
  {\ar_{\be \oplus \id} "60"; "61"};
  {\ar_{\id \oplus \be} "61"; "62"};
  {\ar_{\id \oplus g \oplus \id} "62"; "72"};
  {\ar@/^2.5pc/^{\be} "60"; "62"};
  {\ar@/^2.5pc/^{\be} "70"; "72"};
  (15,-7.5)*{\cong}; 
  (15,-22.5)*{=}; 
  (37,-15)*{=};
  (54,-15)*+{=};
  (77,-15)*+{=}; (93,-15)*+{=}
  \endxy
  \]
  The two triangular regions marked with equal signs are identities by
  (QS1), and the two squares marked with equal signs are identities by
  (QS2).  Thus the invertible 2-cell in the remaining square is an
  identity after whiskering by $\id \oplus \be$.  But since $\id
  \oplus \be$ is itself an isomorphism 1-cell, this statement is true
  before whiskering, thus verifying (QS3).
\end{proof}

\begin{thm}\label{p2isoqs2}
There is an isomorphism of categories $\PGM \cong \qsSMIICat$.
\end{thm}
\begin{proof}\proofof{p2isoqs2}
  Each of the categories above consists of 2-categories equipped with
  additional structure, together with 2-functors preserving all of
  that structure.  Thus to construct the desired isomorphism, we
  merely have to show that to give the data for a permutative
  Gray-monoid structure is the same as to give the data for a
  quasi-strict structure, and similarly that such data satisfies the
  axioms for one structure if and only if it satisfies the axioms for
  the other.  This will produce a bijection on objects, and moreover
  immediately imply a bijection on morphisms.
	
  Recall from \cref{defn:quasi-strict} that a quasi-strict symmetric
  monoidal 2-category consists of
  \begin{itemize}
  \item a Gray-monoid $\cC$, with \emph{cubical} multiplication $+\cn\cC \times \cC \to \cC$ and unit object $e$, with
  \item a strict braided structure in the sense of Crans
    \cite{Cra98Generalized}, the braiding of which we shall call $b$,
    such that
  \item the modifications $R_{-|--}, R_{--|-}, v$ are identities and
  \item the naturality cells $b_{f,\id}, b_{\id, g}$ are identities. 
  \end{itemize}
  By \cref{qs3}, we have omitted axiom (QS3) as it is redundant, and we have given the multiplication as a cubical functor as that is how it appears in the source and target of $b$.  A
  permutative Gray-monoid, on the other hand, consists of
  \begin{itemize}
  \item a Gray-monoid $\cC$, with multiplication given by a 2-functor $\oplus\cn\cC \otimes \cC \to \cC$, with
  \item a 2-natural transformation $\be$ such that
  \item $\be^{2}=1$,
  \item $\be$ is the identity when either object is the unit object,
    and
  \item $\be_{x,yz}$ is given as a composite of $\be_{x,y}$ and $\be_{x,z}$.
  \end{itemize}

  In the presence of the rest of the quasi-strict structure, the
  strict braided structure reduces to the single axiom that
  $b_{e,x}$  is the identity for any object $x$.  Now the braiding
  $b$ is a pseudonatural transformation from an opcubical functor (by
  \cref{cubtoopcub}) to a cubical one, but it satisfies the
  conditions in the third variant listed in \cref{stricttranslem} so it
  induces a 2-natural transformation 
    \[
  \overline{b}: \overline{(+ \circ \tau_{\times})^{*}} \rtarr \overline{+}
  \]
  where we have written the cubical functor giving the multiplication
  as $+$ and the associated 2-functor as $\overline{+}$.  Since
  $\tau_{\times}$ is strict, it is easy to check that
  \[
  (+ \circ \tau_{\times})^{*} = +^{*} \circ \tau_{\times},
  \]
  and by \cref{cubtoopcub2} the associated 2-functor is then
  $\overline{+} \circ \tau_{\otimes}$.  Thus we see that the
  pseudonatural transformation $b$ in the definition of a quasi-strict
  symmetric monoidal 2-category, subject to the conditions in (QS2),
  corresponds to the 2-natural transformation $\be$ in the definition
  of a permutative Gray-monoid.  The only remaining difference between
  the two definitions is that the equation $R_{--|-} = 1$ is not
  explicitly required in the definition of a permutative Gray-monoid,
  but requiring both $v = 1$ and $R_{-|--}=1$ (the third and fifth
  axioms) shows that $\be_{xy,z}$ is the appropriate composite of
  $\be_{x,z}$ and $\be_{y,z}$ by noting that the composite
  \[
  (\be_{x,z}\otimes z) \circ (x \otimes \be_{y,z}) \circ \be_{z,xy}
  \]
  is the identity and using the invertibility of $\be$.
\end{proof}

\begin{rmk}
  The note \cite{Bar14Quasistrict} gives a slightly different
  repackaging of quasi-strict symmetric monoidal 2-categories,
  together with a graphical calculus for 2-cells.  Bartlett's focus is
  on constructing symmetric monoidal bicategories using generators and
  relations, so such a calculus is crucial for his application, while
  our focus is more theoretical.
\end{rmk}

\subsection{Permutative 2-categories}
\label{sec:perm-2-cat}

We now come to the second strict notion of symmetric monoidal
bicategory that we will introduce.  This notion is not equivalent to
those studied thus far in the categorical sense, but it does have many
nice properties such as being described by an operad (see
\cref{prop:perm-2-cats-P-algs}).

\begin{defn}\label{defn:perm-2-cat}
  A \textit{permutative 2-category} $\cC$ consists of a monoid
  $(\cC, \oplus, e)$ in $(\IICat, \times)$, together with a 2-natural
  isomorphism,
  \[
  \xy
  (0,0)*+{\cC \times \cC}="00";
  (25,0)*+{\cC \times \cC}="10";
  (12.5,-10)*+{\cC}="01";
  {\ar^{\tau} "00"; "10"};
  {\ar^{\oplus} "10"; "01"};
  {\ar_{\oplus} "00"; "01"};
  (12.5,-4)*{\Anglearrow{40} \beta}
  \endxy
  \]
  where $\tau \cn \cC \times \cC \to \cC \times \cC$ is the symmetry
  isomorphism in $\IICat$ for the cartesian product, such that the
  same axioms hold as for permutative Gray-monoids once all the
  instances of $\otimes$ are replaced with $\times$.
\end{defn}

From a purely categorical point of view, permutative 2-categories are
very special creatures.  It is not true that every symmetric monoidal
bicategory is symmetric monoidal biequivalent to a permutative
2-category (see \cite[Example 2.30]{Sch2011Classification}). On the other
hand, we will see that this is a natural structure to consider
homotopically: every symmetric monoidal bicategory is weakly
equivalent (i.e., homotopy equivalent after passing to nerves) to a
permutative 2-category.  We prove this in \cref{prop:epzo-weak-equiv}
by proving that every permutative Gray-monoid is weakly equivalent to
a permutative 2-category.  In \cref{thm:main-css-2} we
prove that this establishes an equivalence of homotopy theories.

\begin{prop}
  \label{prop:perm-2-cat-criterion-for-perm-gray-mon}
  Let $(\cC,\oplus,\beta)$ be a permutative Gray-monoid.  Then the
  composite
  \[
  \cC \times \cC \fto{c} \cC \otimes \cC \fto{\oplus} \cC
  \]
  of the universal cubical functor $c$ (see \cref{cubicalmulticat})
  with $\oplus$, together with $\beta * 1_c$, give $\cC$ the structure
  of a permutative 2-category if and only if $\oplus \circ c$ is a
  2-functor.
\end{prop}
\begin{proof}\proofof{prop:perm-2-cat-criterion-for-perm-gray-mon}
  Since the axioms are the same, all that remains is to show that
  $\be*1_{c}$ is 2-natural.  By \cref{stricttranslem}, $\be$ is
  2-natural so $\be*1_{c}$ is a strict transformation from $\oplus c$
  to $\oplus \tau_{\otimes} c$.  Let $+ \cn \cC \times \cC \to \cC$ be
  the cubical functor associated to $\oplus$, so that $\oplus c = +$.
  We know that the cubical functor associated to
  $\oplus \tau_{\otimes}$ is $+^{*} \circ \tau_{\times}$, so
  \[
  \be *1_{c} \cn + \rtarr +^{*}\circ \tau_{\times}
  \]
  is a strict transformation.  But by assumption, $+ = \oplus c$ is a
  2-functor, so $+^{*} = +$ is a 2-functor and strict naturality of
  $\be * 1_{c}$ is then 2-naturality.
\end{proof}

We now turn to the 2-monadic aspects of the theory of permutative
2-categories.  Let $E\Sigma_n$ be the translation category of $\Sigma_n$, viewed as a
discrete 2-category.  The $E\Sigma_n$ give a symmetric operad in
$(\IICat, \times)$ and thus a monad $S$ on $\IICat$.  It is
straightforward to check that this is actually a 2-monad on the
2-category $\IICat_2$.  This is merely the $\Cat$-enrichment of the
operadic approach to permutative categories, and we therefore leave
the proof of the next proposition to the reader.  Note the contrast
with \cref{rmk:no2monadforpgm}.

\begin{prop}
  \label{prop:perm-2-cats-P-algs}
  Permutative 2-categories are precisely the $S$-algebras in $\IICat_2$.
\end{prop}

Using the 2-monad structure on $S$, we can make the following
definitions. We refer the reader to \cite{BKP1989Two} or \cite{Lac02Codescent} for the general definitions. 
\begin{defn}\label{defn:str-fun}
  A \emph{strict functor} between permutative 2-categories is a strict
  $S$-algebra morphism.  A \emph{pseudo, lax, oplax, or normal
    functor} is, respectively, a pseudo, lax, oplax, or normal
  $S$-algebra morphism.
\end{defn}
\begin{thm}\label{thm:p2cat-subcat-pgm}
  Permutative 2-categories, with any choice of morphism above, form a
  full subcategory of permutative Gray-monoids with the corresponding
  morphism type.
\end{thm}
\begin{proof}\proofof{thm:p2cat-subcat-pgm}
  We describe the case of lax morphisms; the other cases are similar.
  A lax $S$-algebra morphism $h\cn \cC \to \cD$ consists of \cite{BKP1989Two} 
  \begin{itemize}
  \item a 2-functor $h:\cC \to \cD$ between the underlying 2-categories and
  \item a 2-natural transformation
    \[
    \xy
    (0,0)*+{S\cC}="00";
    (30,0)*+{S\cD}="10";
    (0,-15)*+{\cC}="01";
    (30,-15)*+{\cD}="11";
    {\ar^{Sh} "00"; "10"};
    {\ar "10"; "11"};
    {\ar "00"; "01"};
    {\ar_{h} "01"; "11"};
    (15,-6)*{\Downarrow \nu}
    \endxy
    \]
    where the two vertical arrows are the $S$-algebra structures on $\cC, \cD$,
  \end{itemize}
  satisfying two axioms.  Using the structure of the operad $S$ and
  following the calculations in \cite{CG14Operads}, we see that this
  amounts to the following data and axioms:
  \begin{itemize}
  \item a 2-functor $h:\cC \to \cD$ between the underlying 2-categories,
  \item a 1-cell $\nu_{0}:e_{\cD} \to h(e_{\cC})$, and
  \item a 2-natural transformation
    \[
    \xy
    (0,0)*+{\cC \times \cC}="00";
    (30,0)*+{\cD \times \cD}="10";
    (0,-15)*+{\cC}="01";
    (30,-15)*+{\cD}="11";
    {\ar^{h \times h} "00"; "10"};
    {\ar^{\oplus_{\cD}} "10"; "11"};
    {\ar_{\oplus_{\cC}} "00"; "01"};
    {\ar_{h} "01"; "11"};
    (15,-6)*{\Downarrow \nu}
    \endxy
    \]
  \end{itemize}
  subject to the requirement that the following diagrams commute for
  all objects $x,y,z \in \cC$.
  \[
  \xy
  (0,0)*+{e \oplus h(x)}="00";
  (30,0)*+{h(e) \oplus h(x)}="10";
  (30,-10)*+{h(e \oplus x)}="11";
  (30,-20)*+{h(x)}="12";
  {\ar^{\nu_{0} \oplus 1} "00"; "10"};
  {\ar^{\nu} "10"; "11"};
  {\ar@{=} "11"; "12"};
  {\ar@{=} "00"; "12"};
  (50,0)*+{h(x) \oplus e}="20";
  (80,0)*+{h(x) \oplus h(e)}="30";
  (80,-10)*+{h(x \oplus e)}="31";
  (80,-20)*+{h(x)}="32";
  {\ar^{1 \oplus \nu_{0}} "20"; "30"};
  {\ar^{\nu} "30"; "31"};
  {\ar@{=} "31"; "32"};
  {\ar@{=} "20"; "32"};
  (20,-30)*+{h(x) \oplus h(y) \oplus h(z)}="55";
  (60,-30)*+{h(x \oplus y) \oplus h(z)}="65";
  (20,-45)*+{h(x) \oplus h(y \oplus z)}="56";
  (60,-45)*+{h(x \oplus y \oplus z)}="66";
  {\ar^{\nu \oplus 1} "55"; "65"};
  {\ar^{\nu} "65"; "66"};
  {\ar_{1 \oplus \nu} "55"; "56"};
  {\ar_{\nu} "56"; "66"};
  (20,-60)*+{h(x) \oplus h(y) }="85";
  (60,-60)*+{h(x \oplus y) }="95";
  (20,-75)*+{h(y) \oplus h(x)}="86";
  (60,-75)*+{h(y \oplus x )}="96";
  {\ar^{\nu } "85"; "95"};
  {\ar^{h(\beta)} "95"; "96"};
  {\ar_{\be} "85"; "86"};
  {\ar_{\nu} "86"; "96"};
  \endxy
  \]
  Thus we see that there is a bijection between lax $S$-algebra
  morphisms $\cC \to \cD$ and lax functors $\cC \to \cD$ in the sense
  of the cartesian analogue of \cref{laxfundef} in which all instances
  of $\otimes$ are replaced with instances of $\times$.

  To complete the proof, we must show that if $(h,\theta,\theta_0)$ is
  a lax functor between permutative 2-categories viewed as permutative
  Gray-monoids via \cref{prop:perm-2-cat-criterion-for-perm-gray-mon},
  then $(h,\theta * 1_c,\theta_0)$ is a lax functor using this
  cartesian analogue of \cref{laxfundef}.  Just as in
  \cref{prop:perm-2-cat-criterion-for-perm-gray-mon}, the axioms are
  the same so this reduces to showing that $\theta * 1_c$ is
  2-natural.  This fact is easily verified using the same methods as
  those which show that $\beta * 1_c$ is 2-natural in the proof of
  \cref{prop:perm-2-cat-criterion-for-perm-gray-mon}.
\end{proof}

\begin{defn}\label{defn:PIICat}
  We let $\PIICat$ denote the category of permutative 2-categories and
  strict functors.
\end{defn}
We have a relative category of permutative
2-categories and weak equivalences defined as follows.
\begin{defn}
  \label{defn:weak-equiv-piicats}
  Let $\cC, \cD$ be a pair of permutative 2-categories.  A strict
  functor of permutative 2-categories $F \cn \cC \to \cD$ is a
  \emph{weak equivalence} if the underlying 2-functor is a weak
  equivalence of 2-categories.  We let $(\PIICat, \cW)$ denote the
  relative category of permutative 2-categories with weak
  equivalences.
\end{defn}


\section{Diagrams of 2-categories}
\label{sec:diagrams-of-2-cats}

In this section we give the basic theory we will need for diagrams of
2-categories.  In addition to diagrams indexed on $\sF$, we will use
diagrams in $\IICat$ indexed on the category $\sA$ of
\cref{defn:diagram-A}.  Therefore we present the theory for a general
diagram category $\sD$.

\begin{defn}\label{defn:D-2cat}
  Given a category $\sD$, let $\sD\mh\IICat$ denote the category of
  functors and natural transformations from $\sD$ to $\IICat$.
\end{defn}

Recall that $\Ga$-objects in $\IICat$ are defined (via
\cref{defn:gammaobjs}) to be those functors $X\cn \sF \to \IICat$ with
$X(\ul{0}_+) = *$.

We will need to consider a category whose objects are
$\sD$-2-categories but whose morphisms are lax in a particular
sense---we describe this in \cref{sec:D-lax-maps} and describe
transformations between them in \cref{sec:D-trans}.  In
\cref{sec:Grothendieck-constr} we describe a Grothendieck construction
for $\sD$-2-categories and lax maps between them.  In
\cref{sec:SM-diagrams} we extend this theory to symmetric monoidal
diagrams for later use in \cref{sec:definition-of-P}.  Finally, \cref{sec:E-construction} gives a general construction which
allows us to replace lax maps with spans of strict ones.

\subsection{Lax maps of diagrams}
\label{sec:D-lax-maps}

Recall that (small) 2-categories, 2-functors, and 2-natural
transformations naturally organize themselves into a 2-category we
denote $\IICat_{2}$.

\begin{defn}
  \label{defn:D-lax-map}
  Let $X,Y$ be $\sD$-2-categories, viewed as 2-functors $X,Y\cn
  \sD \to \IICat_{2}$ with $\sD$ seen as a locally discrete
  2-category, i.e., one with only identity 2-cells.  Then a
  \emph{$\sD$-lax map} is a lax transformation $h:X \to Y$.
\end{defn}

\begin{note}
  When clear from context, we write $\phi_* = X(\phi)$ for a map
  $\phi \in \sD$.
\end{note}

We spell out the details of this definition.  A lax transformation
$h:X \to Y$ consists of
\begin{itemize}
\item for each object $m \in \sD$, a 2-functor $h_{m}:
  X({m}) \to Y({m})$ and
\item for each $\phi:{m} \to {n}$ in $\sD$, a 2-natural
  transformation $h_\phi$ as below.
  \[
  \xy
  (0,0)*+{X({m})}="0";
  (30,0)*+{Y({m})}="1";
  (0,-12)*+{X({n})}="2";
  (30,-12)*+{Y({n})}="3";
  {\ar^{h_{{m}}} "0"; "1" };
  {\ar^{\phi_*} "1"; "3" };
  {\ar_{\phi_*} "0"; "2" };
  {\ar_{h_{{n}}} "2"; "3" };
  {\ar@{=>}^-{h_{\phi}} (17,-4)="a"; "a"+(-5,-3)};
  \endxy
  \]
\end{itemize}
subject to the conditions that $h_{\id_{{m}}} = 1_{h_{{m}}}$
and that $h_{\psi \phi}$ is equal to
the pasting below.  
\[
\xy
  (0,0)*+{X({m})}="0";
  (30,0)*+{Y({m})}="1";
  (0,-12)*+{X({n})}="2";
  (30,-12)*+{Y({n})}="3";
  (0,-24)*+{X({p})}="4";
  (30,-24)*+{Y({p})}="5";
  {\ar^{h_{{m}}} "0"; "1" };
  {\ar^{\phi_*} "1"; "3" };
  {\ar_{\phi_*} "0"; "2" };
  {\ar_{h_{{n}}} "2"; "3" };
  {\ar^{\psi_*} "3"; "5" };
  {\ar_{\psi_*} "2"; "4" };
  {\ar_{h_{{p}}} "4"; "5" };
  {\ar@{=>}^-{h_{\phi}} (17,-4)="a"; "a"+(-5,-3)};
  {\ar@{=>}^-{h_{\psi}} (17,-17)="a"; "a"+(-5,-3)};
\endxy
\]

One can compose lax transformations in the obvious way: given
$h \cn X \to Y, j \cn Y \to Z$, we define $jh$ by
$(jh)_{m} = j_{m} \circ h_{m}$ and $(jh)_{\phi}$ as the pasting below.
\[
\xy
  (0,0)*+{X({m})}="0";
  (30,0)*+{Y({m})}="1";
  (0,-12)*+{X({n})}="2";
  (30,-12)*+{Y({n})}="3";
	(60,0)*+{Z(m)}="4";
	(60,-12)*+{Z(n)}="5";
  {\ar^{h_{{m}}} "0"; "1" };
  {\ar^{\phi_*} "1"; "3" };
  {\ar_{\phi_*} "0"; "2" };
  {\ar_{h_{{n}}} "2"; "3" };
  {\ar@{=>}^-{h_{\phi}} (17,-4)="a"; "a"+(-5,-3)};
	{\ar^{j_{{m}}} "1"; "4" };
  {\ar^{\phi_*} "4"; "5" };
  {\ar_{j_{{n}}} "3"; "5" };
  {\ar@{=>}^-{j_{\phi}} (47,-4)="a"; "a"+(-5,-3)};
\endxy
\]
It is easy to see that this composition is associative and unital.

\begin{defn}
  \label{defn:D2Cat-lax-category}
  The category $\sD\mh\IICat_{l}$ is defined to be the category of
  $\sD$-2-categories and $\sD$-lax maps between them.
\end{defn}

\begin{rmk}
  \label{rmk:D2Cat-lax-is-a-category}
  $\sD\mh\IICat_{l}$ is actually the underlying category of the
  2-category $\mathbf{Lax}(\sD, \IICat_{2})$, the
  2-category of 2-functors, lax transformations, and modifications (see \cref{defn:D-trans})
   from $\sD$ to $\IICat_{2}$.  As every 2-natural
  transformation is lax, there is a canonical inclusion 
  \[
  \sD\mh\IICat \hookrightarrow \sD\mh\IICat_{l}.
  \]
\end{rmk}

\subsection{Transformations of lax maps}
\label{sec:D-trans}

\begin{defn}
  \label{defn:D-trans}
  Let $X,Y$ be $\sD$-2-categories, and let $h,k:X \to Y$ be $\sD$-lax
  maps between them.  A \emph{$\sD$-transformation}, $\la$,     
  \[
  \begin{xy}
    (0,0)*+{X}="A";
    (30,0)*+{Y}="B";
    {\ar@/^1pc/^{h} "A"; "B"};
    {\ar@/_1pc/_{k} "A"; "B"};
    {\ar@{=>}^{\la} (15,2)="a"; "a"+(0,-4) };
  \end{xy}
  \]
  is a modification between the lax transformations $h$ and $k$. More
  precisely, $\la$ consists of a 2-natural transformation
  $\la_{m}:h_{{m}} \Rightarrow k_{{m}}$ for each object ${m} \in \sD$,
  subject to the condition that for each $\phi:m \to n$ in $\sD$ and
  $x \in X(m)$ the square below commutes.
  \[
  \xy
  (0,0)*+{\phi_{*}(h_{m}(x))}="0";
  (40,0)*+{\phi_{*}(k_{m}(x))}="1";
  (0,-15)*+{h_{n}(\phi_{*}(x))}="2";
  (40,-15)*+{k_{n}(\phi_{*}(x))}="3";
  {\ar^{\phi_{*}(\la_m)} "0"; "1"};
  {\ar^{k_{\phi}} "1"; "3"};
  {\ar_{h_{\phi}} "0"; "2"};
  {\ar_{\la_n} "2"; "3"};
  \endxy
  \]
\end{defn}

As indicated by \cref{rmk:D2Cat-lax-is-a-category},
$\sD$-transformations are actually the 2-cells of a 2-category.

\begin{notn}\label{notn:D2Cat-lax-is-a-2-category}
  We will let $(\sD\mh\IICat)_{2}$ denote the 2-category of
  $\sD$-2-categories, maps between them, and $\sD$-transformations
  between those.  We will let $(\sD\mh\IICat_{l})_{2}$ denote the
  2-category of $\sD$-2-categories, $\sD$-lax maps between them, and
  $\sD$-transformations between those.
\end{notn}

Extending \cref{rmk:D2Cat-lax-is-a-category}, we then have an obvious
inclusion of 2-categories
\[
\textrm{inc} \cn (\sD\mh\IICat)_{2} \hookrightarrow (\sD\mh\IICat_{l})_{2}.
\]
We give a characterization of $\sD$-transformations in
\cref{lem:D-trans-characterization} using a notion of path objects
adjoint to the functor $(- \times \Delta[1])$.

In the case $\sD = \sF$, we use the terms $\Ga$-lax map and
$\Ga$-transformation, respectively, for $\sF$-lax maps of reduced
diagrams and $\sF$-transformations of such. We denote by
$\Ga\mh\IICat_l$ the full subcategory of $\sF\mh\IICat_l$ whose
objects are reduced diagrams.

\begin{defn}
  \label{defn:path-2-category}
  Let $\cA$ be a 2-category.  Then $\cA^{\Delta [1]}$ is defined to be the
  2-category where
  \begin{itemize}
  \item objects are the arrows of $\cA$,
  \item a 1-cell $f \to g$ is a pair $(r,s)$ of arrows in $\cA$ such that $gr = sf$, and
  \item a 2-cell $(r,s) \impl (r',s')$ is a pair $(\al, \be)$ of 2-cells in $\cA$ with $\al: r \impl r', \be: s \impl s'$ such that $g * \al = \be * f$.
  \end{itemize}
  Composition and units are given componentwise in $\cA$.
\end{defn}

\begin{prop}\label{delta1}
  For any 2-category $\cA$, there is a weak equivalence
  $i:\cA \to \cA^{\Delta [1]}$, natural in 2-functors, and a pair of
  2-functors $e_{j}: \cA^{\Delta [1]} \rightarrow \cA, j=0,1$, making the composite
  \[
  \cA \stackrel{i}{\longrightarrow} \cA^{\Delta [1]} \fto{e_{0} \times e_{1}} \cA \times \cA
  \]
  equal to the diagonal 2-functor $\cA \to \cA \times \cA$.  For
  2-functors $F,G:\cA \to \cB$ and a 2-natural transformation
  $\al:F \impl G$, there is a 2-functor
  $\tilde{\al}:\cA \rtarr \cB^{\Delta [1]}$ such that 
    \[
  \begin{xy}
    (0,0)*+{\cA}="0";
    (30,0)*+{\cB^{\Delta [1]}}="1";
    (50,12)*+{\cB}="2";
    (50,-12)*+{\cB}="3";
    {\ar^{\tilde{\al}} "0"; "1"};
    {\ar^{e_{0}} "1"; "2"};
    {\ar_{e_{1}} "1"; "3"};
    {\ar@/^1pc/^{F} "0"; "2"};
    {\ar@/_1pc/_{G} "0"; "3"};
  \end{xy}
  \]
  commutes.  This gives a bijection, natural in $\cA, \cB$ between 2-functors $\tilde{\al}\cn\cA\to\cB^{\Delta [1]}$ and triples $(F,G,\al)$ where $F,G\cn\cA\to\cB$ and $\al\cn F \impl G$ is a 2-natural transformation. 
\end{prop}
\begin{proof}\proofof{delta1}
  The 2-functor $i$ sends an object $x$ to $\id_{x}$, a 1-cell $f$ to
  $(f,f)$, and a 2-cell $\al$ to $(\al, \al)$.  The 2-functors $e_0$
  and $e_1$ send an object in $\cA^{\Delta[1]}$ to the source and
  target of the arrow respectively, and project correspondingly for 1-
  and 2-cells.  Both $e_{0} i$ and $e_{1} i$ are the identity, and there is a
  2-natural transformation $ie_{0} \rtarr \id_{\cA^{\Delta [1]}}$, so
  the 2-functors $i, e_{0}, e_{1}$ are all weak equivalences of
  2-categories.

  For the second claim, $\tilde{\al}$ sends $x$ to $\al_{x}$, $f$ to
  $(Ff, Gf)$, and $\ga$ to $(F\ga, G\ga)$.  The 2-naturality of $\al$
  ensures that these are valid cells of $ \cB^{\Delta [1]}$, and the
  commutativity of the above diagram is then clear from the definition
  of the $e_{j}$.
\end{proof}

\begin{rmk}
  The category $\IICat$ is closed symmetric monoidal with respect to
  the cartesian product and the hom-2-category $\cA^{\cB}$ consisting
  of 2-functors, 2-natural transformations, and modifications from
  $\cB$ to $\cA$.  The path 2-category $\cA^{\Delta [1]}$ above is
  simply a more explicit description for the special case
  $\cB = \Delta [1]$, the category $\bullet \to \bullet$ treated as a discrete 2-category.  2-functors $\cC \to \cA^{\Delta [1]}$ are
  in bijection with 2-functors $\cC \times \Delta [1] \to \cA$, and
  similarly for 2-natural transformations and modifications.
\end{rmk}

\begin{defn}
  \label{defn:interval-for-diagram}
  Let $X$ be a $\sD$-2-category.  We define the \emph{path
    $\sD$-2-category} $X^{\Delta [1]}$ by
  \[
  X^{\Delta [1]}(i) = \big( X(i) \big)^{\Delta [1]}.
  \]
  We have $e_0, e_1\cn X^{\Delta[1]} \to X$ given by applying the
  corresponding 2-functors of \cref{delta1} levelwise.
\end{defn}
\begin{note}
  If $\sD = \sF$ and $X$ is a reduced diagram, then $X^{\Delta[1]}$ is
  also a reduced diagram.
\end{note}

\begin{lem}
  \label{lem:D-trans-characterization}
  Let $X,Y$ be $\sD$-2-categories.  There are bijections of sets,
  natural in both variables, between
  \begin{enumerate}
  \item the set whose elements are triples $(h,k,\la)$, where $h,k\cn X
    \to Y$ are $\sD$-lax maps and $\la:h \rtarr k$ is an
    $\sD$-transformation between them; and
  \item the set of $\sD$-lax maps $X \to Y^{\Delta [1]}$.
\end{enumerate}
\end{lem}
\begin{proof}\proofof{lem:D-trans-characterization}
  To establish the bijection, note that given a $\sD$-lax map
  $\widetilde{\la}:X \to Y^{\Delta [1]}$, we get $\sD$-lax maps
  $h = e_{0}\widetilde{\la}$ and $k = e_{1}\widetilde{\la}$.
  There is also a 2-natural transformation $\la_m:h_{m} \Rightarrow k_{m}$
  for each object $m \in \sD$ given by $\widetilde{\la}_m$.  The single axiom
  required for these $\la_m$ to give a $\sD$-transformation is a
  consequence of the 2-naturality of the $\widetilde{\la}_{\phi}$ and the definition
  of 1-cells in $Y^{\Delta [1]}(n)$ being commutative squares.  It is
  easy to check that this function from $\sD$-lax maps
  $\widetilde{\la}:X \to Y^{\Delta [1]}$ to triples $(h,k,\la)$ is a bijection.
\end{proof}
\begin{note}
  If $\sD = \sF$, the statement of \cref{lem:D-trans-characterization}
  also holds for $\Ga$-transformations and $\Ga$-lax maps into the
  path object.
\end{note}

\begin{defn}\label{defn:hoD2Catlax}
  The relative category $(\sD\mh \IICat_{l}, \cW)$ is the category of
  $\sD$-2-categories with $\sD$-lax maps and weak equivalences the
  levelwise weak equivalences of 2-categories.  The category
  $\ho \sD\mh \IICat_{l}$ is the category obtained by formally
  inverting these weak equivalences.
\end{defn} 

\begin{defn}
  \label{defn:hoGa2Catlax}
  We let $(\Ga\mh\IICat_l,\cW)$ denote the full subcategory of reduced
  diagrams in $(\sF\mh\IICat_l,\cW)$ with the same weak equivalences.
  The category $\ho \Ga\mh\IICat_{l}$ is the category obtained by formally
  inverting these weak equivalences.
\end{defn}

\begin{rmk}
  It is important to note that for the moment we do not know that $\ho
  \sD\mh \IICat_{l}$ is locally small as the localization process
  could have produced a proper class of maps between some pair of
  objects.  The isomorphism of homotopy categories in
  \cref{thm:hos-iso-holax} shows that $\ho \sD\mh \IICat_{l}$ is
  indeed locally small.
\end{rmk}

We point out that the laxity in a $\sD$-lax map occurs at the
level of functoriality with respect to the maps in $\sD$, not at the
level of the maps between the individual 2-categories: those are still
strict 2-functors.  Thus the $\sD$-lax structure does not play a role
in deciding whether or not a map is a levelwise weak equivalence.  

\begin{cor}
  \label{cor:gtrans}
  If $\la\cn h \rtarr k$ is a $\sD$-transformation, then $[h] = [k]$ in
  $\ho \sD\mh \IICat_{l}$.
\end{cor}
\begin{proof}\proofof{cor:gtrans}
  This is a consequence of \cref{delta1} and a standard path object
  argument (\cite[5.11]{Hir03Model}).  Since $i:X \to X^{\Delta[1]}$ is a
  levelwise weak equivalence and $e_{j} i = \id$ for $j=0,1$, then
  $[i]$ is an isomorphism in $\ho \sD\mh \IICat_{l}$ so $[e_{0} i] =
  [e_{1} i]$ implies $[e_{0}] = [e_{1}]$.  Therefore
  \[
    [h] = [e_{0} \tilde{\al}] = [e_{0} ][\tilde{\al}] = [e_{1}][\tilde{\al}] = [e_{1}\tilde{\al}] = [k].
  \]
\end{proof}
\begin{note}
  If $\sD = \sF$, then \cref{cor:gtrans} shows that a
  $\Ga$-transformation induces an equality in $\ho\Ga\mh\IICat_l$.
\end{note}

We now return to the subject of primary interest, namely
$\Ga$-2-categories and their homotopy theory.

\begin{defn}\label{defn:Galax-stable-eq}
A $\Ga$-lax map $h\cn X \to Y$ is a \emph{stable equivalence} if the function
\[
  h^{*}:\ho \Ga\mh \IICat_{l}(Y,Z) \to \ho \Ga\mh \IICat_{l}(X,Z)
  \]
  is an isomorphism for every very special $\Ga$-2-category $Z$.
\end{defn}

\begin{defn}\label{defn:HoGa2Catlax}
  The relative category $(\Ga\mh \IICat_{l}, \cS)$ is the category of
  $\Ga$-2-categories with $\Ga$-lax maps and weak equivalences the
  stable equivalences of $\Ga$-2-categories.  The category
  $\Ho \Ga\mh \IICat_{l}$ is the category obtained by formally
  inverting these weak equivalences.
\end{defn}

\subsection{Grothendieck constructions}
\label{sec:Grothendieck-constr}

In this section we describe the Grothendieck construction for diagrams
of 2-categories
\[
X:\sD \to \IICat.
\]
This can be regarded as an enrichment of the standard construction for
diagrams of categories, or as a specialization of the construction
given in \cite{CCG11Classifying} for lax diagrams of bicategories.  We
use the notation of \cite{CCG11Classifying}.

\begin{defn}\label{defn:Grothendieck-constr}
  The (covariant) 2-categorical Grothendieck construction on $X$ is a
  2-category
  \[
  \sD \wre X
  \]
  which has objects, 1-cells, and 2-cells given by pairs as below,
  where we use square brackets for tuples in the Grothendieck
  construction. 
  \[
  \begin{xy}
    (0,0)*+{[m, x]}="A";
    (40,0)*+{[n, y]}="B";
    {\ar@/^1pc/^{[\phi, f]} "A"; "B"};
  {\ar@/_1pc/_{[\phi, g]} "A"; "B"};
  {\ar@{=>}^{[\phi, \alpha]} (17,3)*+{}; (17,-3)*+{} };
  \end{xy}
  \]
  Here $\phi\cn m \to n$
  is a morphism in $\sD$, $x \in X(m)$, $y \in X(n)$, and
  $f, g, \alpha$ are 1- and 2-cells in $X(n)$ as below.
  \[
  \begin{xy}
    (0,0)*+{\phi_*x}="A";
    (40,0)*+{y}="B";
    {\ar@/^1pc/^{f} "A"; "B"};
    {\ar@/_1pc/_{g} "A"; "B"};
    {\ar@{=>}^{\alpha} (17,3)*+{}; (17,-3)*+{} };
  \end{xy}  
  \]
  Composition of 1-cells in $\sD \wre X$
  \[
  [m,x] \fto{[\phi, f]} [n,y] \fto{[\psi, g]} [p,z]
  \]
  is given by $[\psi\phi, g\circ \psi_*f]$, which corresponds to the
  following composite in $X(p)$
  \[
  \psi_*\phi_*x \fto{\psi_*f} \psi_*y \fto{g} z.
  \]
  Horizontal composition of 2-cells is given similarly. Vertical composition of
  2-cells is given by composing vertically the second component in
  $X(n)$. Both compositions are described explicitly in
  \cite{CCG11Classifying}.
\end{defn}

\begin{rmk}
  \label{rmk:laxness-for-us-vs-CCG}
  The laxness direction for what we have called a $\sD$-lax map $X$ to
  $Y$ is what would be called in \cite{CCG11Classifying} a lax map
  from $X^\op$ to $Y^\op$ -- the diagram categories where the
  direction of 1-cells is reversed in each $X(m)$ and $Y(m)$.  Our
  definitions of the Grothendieck construction and rectification
  therefore also differ from those of \cite{CCG11Classifying} in the
  direction of 1-cells.  These are identified by interchanging $X$
  with $X^\op$.
\end{rmk}

To see that the Grothendieck construction defines a functor, one can
specialize the work in \cite{CCG11Classifying} or take a $\Cat$-enrichment of
\cite{Str72Two} to obtain the following.
\begin{prop}
  \label{prop:Grothendieck-constr-defines-functor}
  The Grothendieck construction defines a functor
  \[
  (\sD \wre -) \cn \sD\mh\IICat_l \to \IICat.
  \]
\end{prop}
\begin{rmk}\label{rmk:Grothendieck-description}
It will be useful for us to have a description of $\sD \wre h$ for a
$\sD$-lax map $h \cn X \to Y$.  The functor
$\sD \wre h\cn \sD \wre X \to \sD \wre Y$ is
given explicitly as follows:\\
On 0-cells
\[
(\sD \wre h)\bsb{{m}}{{x}} = \bsb{{m}}{h_{{m}}({x})}.
\]
On 1-cells
\[
(\sD \wre h)\bsb{{\phi}}{{f}} = \bsb{{\phi}}{h_{{n}}({f}) \circ h_{{\phi}}}.
\]
On 2-cells
\[
(\sD \wre h)\bsb{{\phi}}{{\al}} = \bsb{{\phi}}{h_{{n}}({\al}) * 1_{h_{{\phi}}}}.
\]
These are given by the following 0-, 1-, and 2-cells in $Y(n)$.
\[
\begin{xy}
  (-30,0)*+{\phi_*h_m(x)}="Z";
  (0,0)*+{h_n(\phi_* x)}="A";
  (40,0)*+{h_n(y)}="B";
  {\ar^-{h_\phi} "Z"; "A"};
  {\ar@/^1pc/^{h_n (f)} "A"; "B"};
  {\ar@/_1pc/_{h_n (g)} "A"; "B"};
  {\ar@{=>}^{h_n (\alpha)} (17,3)*+{}; (17,-3)*+{} };
\end{xy}  
\]
\end{rmk}

\subsection{Symmetric monoidal diagrams}
\label{sec:SM-diagrams}

In this section we give a basic theory of symmetric monoidal diagrams
and monoidal lax maps in $\IICat$.  The main result here is that
Grothendieck constructions of symmetric monoidal diagrams and monoidal
lax maps are, respectively, permutative 2-categories and strict maps
of such.  In \cref{sec:definition-of-P} we apply this theory to produce
permutative 2-categories from $\Ga$-2-categories.

\begin{defn}
  \label{defn:sm-D-2-cat}
  Let $(\sD,\oplus, e, \be)$ be a permutative category. A
  \emph{symmetric monoidal $\sD$-2-category} $(X,\nu)$ is a symmetric
  monoidal functor
  \[
  (X,\nu)\cn (\sD,\oplus) \to (\IICat, \times).
  \]
\end{defn}

In particular, this means that there is a collection of 2-functors 
\[\nu_{m,p} \cn X(m)\times X(p) \to X(m\oplus p)\]
that are natural with respect to maps in $\sD\times \sD$. There is
also a 2-functor $\nu_e \cn \ast \to X(e)$, where $\ast$ denotes the
terminal 2-category.

\begin{defn}
  \label{defn:sm-D-lax-map}
  Let $(X,\nu)$ and $(Y,\mu)$ be symmetric monoidal
  $\sD$-2-categories. A $\sD$-lax map $h \cn X \to Y$ is a
  \emph{monoidal $\sD$-lax map} if the following conditions hold.
  \begin{enumerate}
    \item The diagram below commutes.
  \begin{equation}\label{monoidal-lax0}
    \begin{xy}
      (0,-8)*+{\ast}="A";
      (30,0)*+{X(e)}="B";
      (30,-16)*+{Y(e)}="D";
      {\ar^-{\nu_{e}} "A"; "B"};
      {\ar_-{\mu_e} "A"; "D"};
      {\ar^-{h_{e}} "B"; "D"};
    \end{xy}
  \end{equation}
  \item For all $m, p \in \sD$ the square below commutes.
  \begin{equation}\label{monoidal-lax1}
    \begin{xy}
    (0,0)*+{X(m) \times X(p)}="A";
    (30,0)*+{X(m \oplus p)}="B";
    (0,-15)*+{Y(m) \times Y(p)}="C";
    (30,-15)*+{Y(m \oplus p)}="D";
    {\ar^-{\nu_{m,p}} "A"; "B"};
    {\ar_-{\mu_{m,p}} "C"; "D"};
    {\ar_-{h_m \times h_p} "A"; "C"};
    {\ar^-{h_{m \oplus p}} "B"; "D"};
    \end{xy}
  \end{equation}
  \item For all $\phi \cn m \to n$ and $\psi\cn p \to q$ in
  $\sD$ we have the following equality of pasting diagrams.
  \begin{equation}\label{monoidal-lax2}
  \begin{array}{c}
    \begin{xy}
      (0,0)*+{Y(m) \times Y(p)}="yy1";
      (20,20)*+{X(m) \times X(p)}="xx1";
      (20,-20)*+{Y(n) \times Y(q)}="yy2";
      (50,20)*+{X(m \oplus p)}="x1";
      (50,-20)*+{Y(n \oplus q)}="y2";
      (70,0)*+{X(n \oplus q)}="x2";
      (30,0)*+{Y(m \oplus p)}="y1";  
      {\ar^-{\nu_{m,p}} "xx1"; "x1"};
      {\ar^-{\mu_{m,p}} "yy1"; "y1"};
      {\ar_-{\mu_{n,q}} "yy2"; "y2"};
      {\ar_-{h_m \times h_p} "xx1"; "yy1"};
      {\ar_-{h_{m \oplus p}} "x1"; "y1"};
      {\ar^-{h_{n \oplus q}} "x2"; "y2"};
      {\ar_-{\phi_* \times \psi_*} "yy1"; "yy2"};
      {\ar^-{(\phi \oplus \psi)_*} "x1"; "x2"};
      {\ar_-{(\phi \oplus \psi)_*} "y1"; "y2"};
      {\ar@{=>}^-{h_{\phi \oplus \psi}} (47.5,0)="a"; "a"+(5,0) };
      {\ar@{=} "xx1"+(0,-12)="b"; "b"+(2.2,0) };
      {\ar@{=} "yy2"+(0,12)="c"; "c"+(2.2,0) };
    \end{xy}
    \\
    \begin{xy}
      {\ar@{=}^{} (0,2.2)*+{}; (0,-2.2)*+{};};
    \end{xy}
    \\
    \begin{xy}
      (0,0)*+{Y(m) \times Y(p)}="yy1";
      (20,20)*+{X(m) \times X(p)}="xx1";
      (20,-20)*+{Y(n) \times Y(q)}="yy2";
      (50,20)*+{X(m \oplus p)}="x1";
      (50,-20)*+{Y(n \oplus q)}="y2";
      (70,0)*+{X(n \oplus q)}="x2";
      (40,0)*+{X(n) \times X(q)}="xx2";  
      {\ar^-{\nu_{m,p}} "xx1"; "x1"};
      {\ar^-{\nu_{n,q}} "xx2"; "x2"};
      {\ar_-{\mu_{n,q}} "yy2"; "y2"};
      {\ar_-{h_m \times h_p} "xx1"; "yy1"};
      {\ar^-{h_n \times h_q} "xx2"; "yy2"};
      {\ar^-{h_{n \oplus q}} "x2"; "y2"};
      {\ar^-{\phi_* \times \psi_*} "xx1"; "xx2"};
      {\ar_-{\phi_* \times \psi_*} "yy1"; "yy2"};
      {\ar^-{(\phi \oplus \psi)_*} "x1"; "x2"};
      {\ar@{=>}^-{h_{\phi} \times h_{\psi}} (17.5,0)="a"; "a"+(5,0) };
      {\ar@{=} "x1"+(0,-12)="b"; "b"+(-2.2,0) };
      {\ar@{=} "y2"+(0,12)="c"; "c"+(-2.2,0) };
    \end{xy}
  \end{array}
  \end{equation}
  \end{enumerate}
\end{defn}

\begin{rmk}
  \label{rmk:symm-mon-diagrams-and-mon-D-lax-maps-form-subcategory}
  Because the cartesian product of 2-categories is strictly
  functorial, the composite of monoidal $\sD$-lax maps is again
  monoidal.  The collection of symmetric monoidal $\sD$-2-categories
  and monoidal $\sD$-lax maps therefore forms a category
  with a faithful forgetful functor to $\sD\mh\IICat_l$.
\end{rmk}

\begin{defn}
  \label{defn:cat-of-sm-diagrams}
  Given a permutative category $\sD = (\sD, \oplus, e, \be)$, let
  $(\sD,\oplus)\mh\IICat_l$ denote the category of symmetric monoidal
  $\sD$-2-categories and monoidal $\sD$-lax maps.
\end{defn}

\begin{prop}
  \label{prop:Grothendieck-constr-sm-functor}
  Let $(\sD,\oplus,e, \be)$ be a permutative category and let
  $(X,\nu)$ be a symmetric monoidal $\sD$-2-category.  Then the
  Grothendieck construction $\sD \wre X$ is a permutative 2-category.
\end{prop}
\begin{proof}\proofof{prop:Grothendieck-constr-sm-functor}
The monoidal product on objects  of $\sD \wre X$ is defined as
\[
([m,x],[p,y])\mapsto [m\oplus p, \nu_{m,p}(x,y)]. 
\]
The product of a pair of 1-morphisms $([\phi,f],[\psi,g])$ is given by
the 1-morphism whose first coordinate is $\phi \oplus \psi$ and whose
second coordinate is the 1-morphism
\[
(\phi \oplus \psi)_* \nu_{m,p} (x,y) = 
\nu_{n,q}(\phi_* x,\psi_* y) \fto{\nu_{n,q}(f,g)} \nu_{n,q}(x',y').
\]
Pairs of 2-morphisms are treated similarly.

It is routine to check that this product is a 2-functor. Since $\sD$
is a permutative category and $(X,\nu)$ is a symmetric monoidal
functor, it is straightforward to see that this 2-functor defines a
permutative structure on $\sD \wre X$ with monoidal unit given by
$[e,\nu_e(\ast)]$ and symmetry transformation given by
\begin{eqn}
  \label{eqn:symmetry-for-Grothendieck-constr}
[(m\oplus n, \nu_{m,n}(x,y)] \fto{[\be_{m,n},\id]} 
[n\oplus m, \nu_{n,m}(y,x)].\hfill \qedhere
\end{eqn}
\end{proof}

\begin{prop}
  \label{prop:symm-D-lax-to-sm-2functor}
  Let $\sD$ be a permutative category and let $ h \cn X \to Y $ be a
  monoidal $\sD$-lax map of symmetric monoidal $\sD$-2-categories.
  Then
  \[
  \sD \wre h \cn \sD \wre X \to \sD \wre Y
  \]
  is a strict symmetric monoidal 2-functor between permutative
  2-categories. This construction gives the assignment on morphisms of
  a functor
\[\sD\wre - \cn (\sD,\oplus)\mh\IICat_l \to \PIICat.\]
\end{prop}
\begin{proof}\proofof{prop:symm-D-lax-to-sm-2functor}
%
  The fact that $\sD \wre h$ preserves the monoidal structure strictly
  at the level of objects follows from the condition in
  \cref{monoidal-lax1}. Likewise, the condition in
  \cref{monoidal-lax2} shows that $\sD \wre h$ preserves the product
  on 1- and 2-cells strictly.  The fact that $\sD \wre h$ preserves
  the unit strictly follows from \cref{monoidal-lax0}. To see that
  $\sD \wre h$ strictly preserves the symmetry, apply the description
  of $\sD \wre h$ on 1-cells from \cref{rmk:Grothendieck-description} to
  the formula for the symmetry given in
  \cref{eqn:symmetry-for-Grothendieck-constr}.
\end{proof}

\subsection{From lax maps to spans}
\label{sec:E-construction}

In this section we describe a construction, inspired by that of
\cite{Man10Inverse}, which allows us to replace a $\Ga$-lax map with a
weakly-equivalent span of strict $\Ga$-maps.  This is a key ingredient
in our proof of \cref{thm:main-equiv-ho-thy}, but we also use it in
this section to prove that the homotopy categories of $\sD\mh\IICat$
and $\sD\mh\IICat_l$ are isomorphic.  Specializing to reduced diagrams
on $\sF$, we prove the following isomorphisms of categories (see
\cref{cor:hogas-iso-hogalax,cor:Hos-iso-Holax}):
\begin{align*}
  \ho\GaIICat_l \iso \ho\GaIICat,\\
  \Ho\GaIICat_l \iso \Ho\GaIICat.
\end{align*}

\begin{defn}\label{defn:E}
  Let $k\cn X \to Z$ be a $\sD$-lax map.  For each $m \in \sD$, define
  a 2-category $Ek(m)$ as the following pullback in $\IICat$.
  \[
  \begin{xy}
    (0,0)*+{Ek(m)}="0";
    (20,0)*+{Z^{\Delta[1]}(m)}="1";
    (0,-20)*+{X(m)}="2";
    (20,-20)*+{Z(m)}="3";
    {\ar^{\ol{\nu}} "0"; "1"};
    {\ar_-{\om} "0"; "2"};
    {\ar^-{e_1} "1"; "3"};
    {\ar_-{k} "2"; "3"};
  \end{xy}
  \]
  Thus $Ek(m)$ has 0-cells given by triples $(x, f, a)$ where $x$ is an object in $X(m)$, $a$ is an object in $Z(m)$, and
  \[
  a \fto{f} k(x)
  \]
  is a 1-cell in $Z(m)$.

  A 1-cell from $(x,f,a)$ to $(y,g,b)$ is given by a pair $(s,r)$ where $s\colon x \to y$ and $r\colon a \to b$ are 1-cells in $X(m)$ and $Z(m)$, respectively, such that the diagram
    \[
  \begin{xy}
    (0,0)*+{a}="0";
    (30,0)*+{b}="1";
    (0,-17)*+{k(x)}="2";
    (30,-17)*+{k(y)}="3";
    {\ar^-{r} "0"; "1"};
    {\ar_-{k(s)} "2"; "3"};
    {\ar_-{f} "0"; "2"};
    {\ar^-{g} "1"; "3"};
  \end{xy}
  \]
commutes. Similarly, a  2-cell from $(s,r)$ to $(s',r')$ is  a pair $(\beta, \alpha)$ of 2-cells $\be \cn s \Rightarrow s'$ and $\al \cn r \Rightarrow r'$ in $X(m)$ and $Z(m)$, respectively, such that
  \[
  \begin{xy}
    (0,0)*+{a}="0";
    (30,0)*+{b}="1";
    (0,-17)*+{k(x)}="2";
    (30,-17)*+{k(y)}="3";
    {\ar@/^9pt/^-{r} "0"; "1"};
    {\ar@/_9pt/_-{r'} "0"; "1"};
    {\ar@/_9pt/_-{k(s')} "2"; "3"}="s'";
    {\ar_-{f} "0"; "2"};
    {\ar^-{g} "1"; "3"};
    {\ar@{=>}^-{\al} (15,2.8)*+{}; (15,-2.8)*+{}};
            (40,-8)*+{=};
        (50,0)*+{a}="00";
    (80,0)*+{b}="11";
    (50,-17)*+{k(x)}="22";
    (80,-17)*+{k(y)}="33";
    {\ar@/^9pt/^-{r} "00"; "11"};
    {\ar@/^9pt/^-{k(s)} "22"; "33"};
    {\ar@/_9pt/_-{k(s')} "22"; "33"}="s'";
    {\ar_-{f} "00"; "22"};
    {\ar^-{g} "11"; "33"};
    {\ar@{=>}^-{k(\be)} (65,-14.2)*+{}; (65,-19.8)*+{}};
  \end{xy}
  \]
  These cells are depicted as follows
  \[
  \begin{xy}
    (0,0)*+{a}="0";
    (30,0)*+{b}="1";
    (0,-17)*+{k(x)}="2";
    (30,-17)*+{k(y)}="3";
    {\ar@/^9pt/^-{r} "0"; "1"};
    {\ar@/_9pt/_-{r'} "0"; "1"};
    {\ar@/^9pt/^-{k(s)} "2"; "3"};
    {\ar@/_9pt/_-{k(s')} "2"; "3"}="s'";
    {\ar_-{f} "0"; "2"};
    {\ar^-{g} "1"; "3"};
    {\ar@{=>}^-{\al} (15,2.8)*+{}; (15,-2.8)*+{}};
    {\ar@{=>}^-{k(\be)} (15,-14.2)*+{}; (15,-19.8)*+{}};
  \end{xy}
  \]
  Vertical and horizontal composition are given by the corresponding
  compositions in $X(m)$ and $Z(m)$. The 2-functors $\om$ and $\ol{\nu}$ are given by projection onto the first and second components, respectively.
     
  For a map $\phi\cn m \to n$ in $\sD$, there is a 2-functor
  \[
  \phi_* \cn Ek(m) \to Ek(n)
  \]
  that sends the data above to the data represented by the
  following diagram. 
  \begin{eqn}\label{eqn:Ek-phi}
  \begin{xy}
    (0,0)*+{\phi_* a}="0";
    (35,0)*+{\phi_* b}="1";
    (0,-17)*+{\phi_* k(x)}="2";
    (35,-17)*+{\phi_* k(y)}="3";
    (0,-34)*+{k\phi_* (x)}="4";
    (35,-34)*{k\phi_* (y)}="5";
    {\ar@/^9pt/^-{\phi_* r} "0"; "1"};
    {\ar@/_9pt/_-{\phi_* r'} "0"; "1"};
    {\ar@/^9pt/^-{k\phi_* (s)} "4"; "5"};
    {\ar@/_9pt/_-{k\phi_* (s')} "4"; "5"};
    {\ar_-{\phi_* (f)} "0"; "2"};
    {\ar^-{\phi_* (g)} "1"; "3"};
    {\ar_-{k_\phi} "2"; "4"};
    {\ar^-{k_\phi} "3"; "5"};
    {\ar@{=>}^-{\phi_* \al} (15,2.9)*+{}="a"; "a"+(0,-5)*+{}};
    {\ar@{=>}^-{k\phi_* (\be)} (14,-31.2)*+{}="b"; "b"+(0,-5)*+{}};
  \end{xy}
  \end{eqn}
\end{defn}

One can verify that $(\psi\phi)_* = \psi_* \phi_*$ directly from the description in
\cref{eqn:Ek-phi} using the equality of $k_{\psi\phi}$ with the
pasting of $k_\psi$ and $k_\phi$ (see \cref{defn:D-lax-map}). This proves the following proposition.
\begin{prop}
  The 2-categories $Ek(m)$ assemble to make $Ek$ a $\sD$-2-category.
\end{prop}

\begin{prop}\label{prop:E_square}
  The $\sD$-2-category $Ek$ fits in a commuting diagram in
  $\sD\mh\IICat_l$ shown below.  The map $\om$ and the horizontal
  composite $\nu$ are both strict $\sD$-maps.
  \[
  \begin{xy}
    (0,0)*+{Ek}="0";
    (20,0)*+{Z^{\Delta[1]}}="1";
    (40,0)*+{Z}="1'";
    (0,-20)*+{X}="2";
    (20,-20)*+{Z}="3";
    {\ar^-{\ol{\nu}} "0"; "1"};
    {\ar^-{e_0} "1"; "1'"};
    {\ar_-{\om} "0"; "2"};
    {\ar^-{e_1} "1"; "3"};
    {\ar_-{k} "2"; "3"};
    {\ar@/^2pc/^{\nu} "0"; "1'"};
  \end{xy}
  \]
\end{prop}
\begin{proof}\proofof{prop:E_square}
For a map $\phi\cn m \to n$ in $\sD$ it is immediate that
  $\om$ strictly commutes with $\phi_*$.  The laxity of $k$ provides
  1-cells $\phi_* \ol{\nu} (x, f, a) \to \ol{\nu} \phi_* (x, f, a)$ in $Z^{\Delta[1]}(n)$ as below:
  \[
  \phi_* \ol{\nu} (x, f, a) = \phi_*(f)
  \fto{(\id_{\phi_*(a)}, k_\phi)} 
  k_{\phi} \circ \phi_* (f) = \ol{\nu} \phi_* (x, f, a).
  \]
  These are 2-natural because $k_\phi$ is 2-natural, and the
  conditions for identity 1-cells and pasting on composites are
  immediate from those of the $k_\phi$.

  The map $e_0$ is given by projecting to the source of each cell in
  $Z^{\De[1]}$, and, as the laxity of $\ol{\nu}$ is concentrated in
  the target components, the composite $\nu = e_0 \ol{\nu}$ is a
  strict $\sD$-map.
\end{proof}

\begin{prop}\label{prop:om-levelwise-adj}
  The map $\omega$ is a levelwise left 2-adjoint and hence is a levelwise
  weak equivalence.
\end{prop}
\begin{proof}\proofof{prop:om-levelwise-adj}
Consider the 2-functor $i\cn X(m) \to Ek(m)$ given by
\begin{align*}
 x & \mapsto (x, \id_{k(x)}, k(x))\\
 s & \mapsto (s, k(s))\\
 \be & \mapsto (\be, k(\be)).
\end{align*}
It is easy to check that this assignment is indeed a 2-functor that lands in $Ek(m)$ and that $\om i$ is the identity on $X(m)$. Note that for an object $(x,f,a)$ in $Ek(m)$, the pair $(\id_x,f)$ is a 1-cell from $(x,f,a)$ to $(x,\id_{k(x)},k(x))$. We leave to the reader to check that these 1-cells give a 2-natural transformation from the identity on $Ek(m)$ to $i\om$.

The triangle identities are immediate:  First note that $\om(\id_x, f) = \id_x$
and therefore applying $\om$ to the counit is the identity.  Second,
the component of the counit on $i(x)$ is $(\id_x, \id_{k(x)})$.
\end{proof}

Recall that $\ho \sD\mh \IICat$, $\ho \sD\mh \IICat_{l}$ denote the
categories obtained by formally inverting the levelwise weak
equivalences of $\sD$-2-categories in the categories $\sD\mh \IICat$,
$\sD\mh \IICat_{l}$, respectively.  Since the laxity of a map plays no
role in determining whether it is or is not a levelwise weak
equivalence, the inclusion
\[
\sD\mh \IICat\hookrightarrow \sD\mh \IICat_{l}
\]
both preserves and reflects levelwise equivalences, and therefore
induces a functor on homotopy categories
\[
\ho \sD\mh \IICat\to \ho \sD\mh \IICat_{l}.
\]
\begin{thm}\label{thm:hos-iso-holax}
The functor $\ho \sD\mh \IICat\to \ho \sD\mh \IICat_{l}$ is an
isomorphism of categories. 
\end{thm}
\begin{proof}\proofof{thm:hos-iso-holax}
Let $k\cn X \to Y$ be a $\sD$-lax map, and consider the commutative
diagram of \cref{prop:E_square}.  Then we have the following
calculation in $\ho \sD\mh \IICat_{l}$ using that $[e_{0}] = [e_{1}]$
as in \cref{cor:gtrans}.
\[
[k \om] = [e_{1} \overline{\nu}] = [e_{0} \overline{\nu}] = [\nu]
\]
Thus $[k] = [\nu][\om]^{-1}$ in $\ho \sD\mh \IICat_{l}$ since, by
\cref{prop:om-levelwise-adj}, $\om$ is a levelwise weak equivalence.  Now levelwise weak
equivalences satisfy the 2 out of 3 property, so if $k$ is a
levelwise equivalence then $\nu$ is also.  Therefore $[k]^{-1} =
[\om][\nu]^{-1}$ in $\ho \sD\mh \IICat_{l}$.  Since every morphism in
$\ho \sD\mh \IICat_{l}$ is a composite of the form
\[
[k_1]^{-1}[k_2][k_3]^{-1} \cdots [k_{2n}]
\]
where the $k_{2i+1}$'s are all levelwise equivalences, the function on hom-sets
\[
\ho \sD\mh \IICat(X,Y) \to \ho \sD\mh \IICat_{l}(X,Y)
\]
is seen to be surjective by replacing the $[k_{2i}]$ with
$[\nu_{2i}][\om_{2i}]^{-1}$ and the $[k_{2i+1}]^{-1}$ with
$[\om_{2i+1}][\nu_{2i+1}]^{-1}$.  Injectivity then follows from the
fact that performing this procedure on a morphism in the image of $\ho
\sD\mh \IICat(X,Y) \to \ho \sD\mh \IICat_{l}(X,Y)$ is clearly the
identity.  Since these categories have the same objects, the map is
therefore an isomorphism. 
\end{proof}

\begin{prop} \label{prop:E-on-composable-pairs}
  Given a composable pair of $\sD$-lax maps
  \[
  X \fto{h} Y \fto{j} Z
  \]
  there are $\sD$-lax maps
  \[
  Eh \fto{j_E} E(jh) \fto{h_E} E(j)
  \]
  induced by maps on levelwise pullbacks.  The map $j_E$,
  resp. $h_E$, is a strict $\sD$-map if $j$, resp. $h$, is a
  strict $\sD$-map.
\end{prop}
 
\begin{prop} \label{prop:E-is-weird-functor}
  Given a parallel pair of $\sD$-lax maps and a $\sD$-transformation
  \[
  \begin{xy}
    (0,0)*+{X}="A";
    (30,0)*+{Y}="B";
    {\ar@/^1pc/^{h} "A"; "B"};
    {\ar@/_1pc/_{k} "A"; "B"};
    {\ar@{=>}^{\la} (15,2)="a"; "a"+(0,-4) };
  \end{xy}
  \]
  there is a strict $\sD$-map 
  \[
  E\la \cn Eh \to Ek.
  \]
  This defines a functor
  \[
  E\cn \mathbf{Lax}(\sD, \IICat_{2})(X,Y) \to \sD\mh\IICat
  \]
  from the category of $\sD$-lax maps and $\sD$-transformations
  between $X$ and $Y$ to the category of $\sD$-2-categories and strict
  $\sD$-maps.
\end{prop}
\begin{proof}\proofof{prop:E-is-weird-functor}
  The map $E\la$ is given on 0-cells of $Eh(m)$ by $E\la(x, f, a) =
  (x, \la_m(x) \circ f, a)$ and is the identity on 1- and 2-cells.
  The 2-naturality of $\la_m$ ensures that this is a well-defined map
  to $Ek(m)$, and the condition $\la_n(x) \circ h_\phi(x) = k_\phi(x)
  \circ \phi_*(\la_m(x))$ of \cref{defn:D-trans} ensures that this is
  a strict $\sD$-map.
  
  It is immediate that this construction is functorial with respect to
  $\sD$-lax maps and produces the identity map if $\la$ is the
  identity.
\end{proof}
 
For a square of $\sD$-lax maps with a $\sD$-transformation between the
two composites
\[
\begin{xy}
  (0,0)*+{X}="0";
  (20,0)*+{Z}="1";
  (0,-15)*+{Y}="2";
  (20,-15)*+{W}="3";
  {\ar^-{i} "0"; "1"};
  {\ar_-{h} "0"; "2"};
  {\ar^-{k} "1"; "3"};
  {\ar_-{j} "2"; "3"};
  {\ar@{=>}^-{\la} (7.5,-10)="a"; "a"+(5,5)}; 
\end{xy}
\]
we combine \cref{prop:E-on-composable-pairs,prop:E-is-weird-functor}
to obtain the following maps:
\begin{eqn}\label{eqn:E-maps}
\begin{xy}
  (0,0)*+{Eh}="A";
  (20,0)*+{E(jh)}="B";
  (40,0)*+{Ej}="C";
  (0,-15)*+{Ei}="D";
  (20,-15)*+{E(ki)}="E";
  (40,-15)*+{Ek}="F";
  {\ar^-{j_E} "A"; "B"};
  {\ar^-{h_E} "B"; "C"};
  {\ar^-{k_E} "D"; "E"};
  {\ar^-{i_E} "E"; "F"};
  {\ar^-{E\la} "B"; "E"};
\end{xy}
\end{eqn}
When the maps $h$ and $k$ are strict and $\la$ is the identity, we can
say more.  

\begin{defn}\label{D2Cat-lax-arrow}
Let
\[
\sD\mh\IICat^{\bullet \to \bullet}_{\mathrm{lax},\mathrm{str}}
\]
denote the arrow category whose objects are $\sD$-lax maps $X \fto{i}
Z$ and whose morphisms are strict maps making the corresponding
squares commute:
\begin{eqn}\label{eqn:lax-str-square}
\begin{xy}
  (0,0)*+{X}="0";
  (20,0)*+{Z}="1";
  (0,-15)*+{Y}="2";
  (20,-15)*+{W}="3";
  {\ar^-{i} "0"; "1"};
  {\ar_-{h} "0"; "2"};
  {\ar^-{k} "1"; "3"};
  {\ar_-{j} "2"; "3"};
\end{xy}
\end{eqn}
Composition is given by stacking and pasting vertically, along the
strict $\sD$-maps. 
\end{defn} 

\begin{defn}\label{defn:span}
Let $\cE$ be a category.
\begin{enumerate}
\item A \emph{span} in $\cE$ is a diagram of the form $X \leftarrow A \to Y$.
\item A \emph{map of spans} from $X \leftarrow A \to Y$ to $X' \leftarrow A' \to Y'$ consists of maps $f\cn X \to X', g \cn A \to A', h \cn Y \to Y'$ making the two squares commute.
\[
\begin{xy}
  0;<18mm,0mm>:<0mm,12mm>::
  (0,0)*+{X}="A";
  (1,0)*+{A}="B";
  (2,0)*+{Y}="C";
  (0,-1)*+{X'}="D";
  (1,-1)*+{A'}="E";
  (2,-1)*+{Y'}="F";
  {\ar^-{} "B"; "A"};
  {\ar^-{} "B"; "C"};
  {\ar^-{} "E"; "D"};
  {\ar^-{} "E"; "F"};
  {\ar_-{f} "A"; "D"};
  {\ar^-{g} "B"; "E"};
  {\ar^-{h} "C"; "F"};
\end{xy}
\]

\item The category $\mathpzc{Span}\,(\cE)$ has objects spans in $\cE$, and maps of spans between them.
\end{enumerate}
\end{defn}

\begin{rmk}
It is clear that $\mathpzc{Span}\,(\cE)$ is merely the functor category $[ \bullet \ot \bullet \to
  \bullet, \cE]$.
\end{rmk}

We thus have the following refinement of
\cref{eqn:E-maps}.

\begin{cor}\label{cor:E-on-arrow-category}
  The construction $E$ defines a functor
  \[
  \sD\mh\IICat^{\bullet \to \bullet}_{\mathrm{lax},\mathrm{str}} \; \longrightarrow \;
  \mathpzc{Span}\,(\sD\mh\IICat)
  \]
  given on objects by sending a lax map $i\cn X \to Z$ to
  \[
  X \fot{\om} Ei \fto{\nu} Z.
  \]
\end{cor}
\begin{proof}\proofof{cor:E-on-arrow-category}
  Given a commuting square as in \cref{eqn:lax-str-square}, we have
  the composite $\sD$-map
  \[
  Ei \fto{k_E} E(ki) = E(jh) \fto{h_E} E(j).
  \]
  This is a strict $\sD$-map because $h$ and $k$ are strict.  One
  verifies immediately that this map commutes with the structure maps
  $\om$ and $\nu$.

  A lengthy but straightforward check verifies that this construction
  is functorial with respect to the morphisms of
  $\mathpzc{Span}\,(\sD\mh\IICat)$.
\end{proof}

\begin{rmk}\label{rmk:double_cat}
  The category $\sD\mh\IICat^{\bullet \to
    \bullet}_{\mathrm{lax},\mathrm{str}}$ arises naturally in the
  following way.  There is a double category with objects
  $\sD$-2-categories, vertical arrows strict $\sD$-maps, horizontal
  arrows lax-$\sD$-maps, and cells given by commutative squares.  We
  can view this double category as a category internal to categories
  by taking the category of objects to be $\sD\mh\IICat_{l}$ and then
  $\sD\mh\IICat^{\bullet \to \bullet}_{\mathrm{lax},\mathrm{str}}$
  appears as the category of arrows.
\end{rmk}

We now concentrate on the case of $\Ga$-2-categories. Recall that
$\Ga\mh\IICat$ and $\Ga\mh\IICat_l$ denote the full subcategories of
$\sF\mh\IICat$ and $\sF\mh\IICat_l$, respectively, given by reduced
diagrams.  Likewise, we let $\Ga\mh\IICat^{\bullet \to
  \bullet}_{\mathrm{lax},\mathrm{str}}$ denote the full subcategory of
$\sF\mh\IICat^{\bullet \to \bullet}_{\mathrm{lax},\mathrm{str}}$
whose objects are lax maps between $\Ga$-2-categories.

Note that if $k \colon X \to Z$ is a $\Ga$-lax map of
$\Ga$-2-categories, then the $\sF$-2-categories $Z^{\De[1]}$ and $Ek$
are reduced.  This leads to the following refinements of our results
above.

\begin{cor}\label{cor:hogas-iso-hogalax}
The functor $\ho \Ga\mh \IICat\to \ho \Ga\mh \IICat_{l}$ is an
isomorphism of categories. 
\end{cor}
\begin{proof}\proofof{cor:hogas-iso-hogalax}
  When considering only $\Ga$-2-categories, the diagram of
  \cref{prop:E_square} is a diagram of $\Ga$-2-categories.  The result
  then follows by repeating the rest of the proof of
  \cref{thm:hos-iso-holax}.
\end{proof}

\begin{cor}\label{cor:E-on-Ga-arrow-category}
  The construction $E$ defines a functor
  \[
  \Ga\mh\IICat^{\bullet \to \bullet}_{\mathrm{lax},\mathrm{str}} \; \longrightarrow \;
  \mathpzc{Span}\,(\Ga\mh\IICat).
  \]
\end{cor}
\begin{proof}\proofof{cor:E-on-Ga-arrow-category}
  This is immediate from \cref{cor:E-on-arrow-category}.
\end{proof}

\begin{cor}\label{cor:Hos-iso-Holax}
A strict map $h \cn X \to Y$ of $\Ga$-2-categories is a stable
equivalence in the sense of \cref{defn:stable-equiv} if and only if,
when considered as a lax map, it is a stable equivalence in the sense
of \cref{defn:Galax-stable-eq}.  In particular, the inclusion $\Ga\mh
\IICat\hookrightarrow \Ga\mh \IICat_{l}$ preserves and reflects stable
equivalences, so induces a functor of stable homotopy categories
\[
\Ho \Ga\mh \IICat \to \Ho \Ga\mh \IICat_{l}.
\]
This functor is an isomorphism.
\end{cor}
\begin{proof}\proofof{cor:Hos-iso-Holax}
The first statement is obvious by \cref{cor:hogas-iso-hogalax}, and the
second is immediate from the first.  One then uses the same argument as in
\cref{thm:hos-iso-holax} to show that the functors on stable homotopy
categories are isomorphisms.
\end{proof}

\begin{rmk}
 In subsequent work \cite{GJO2017Extending}, the authors develop more general techniques and improve \cref{cor:hogas-iso-hogalax,cor:Hos-iso-Holax} to equivalences of homotopy theories.
\end{rmk}

\begin{cor}\label{cor:nu-stable-eq}
  Let $h\cn X \to Z$ be a $\Ga$-lax map of $\Ga$-2-categories and
  consider the construction
  \[
  X \fot{\om} Eh \fto{\nu} Z
  \]
  of \cref{prop:E_square}.  Then $h$ is a levelwise, resp. stable,
  equivalence if and only if $\nu$ is a levelwise, resp. stable,
  equivalence.
\end{cor}
\begin{proof}\proofof{cor:nu-stable-eq}
  We have 
  \[
  h \om = e_{1} \overline{\nu}
  \]
  and 
  \[
  \nu = e_0 \overline{\nu}
  \]
  by construction.  We know that $\om$, $e_{0}$, and
  $e_{1}$ are all levelwise equivalences and that both $\om$ and $\nu$
  are strict maps.

  The claim for levelwise equivalences is immediate by 2 out of 3 for
  levelwise equivalences.  Using 2
  out of 3 for stable equivalences shows that $h$ is a stable
  equivalence if and only if $\nu$ is a stable equivalence in
  $\Ga\mh\IICat_l$.  The result then follows by
  \cref{cor:Hos-iso-Holax}.
\end{proof}


\section{Permutative 2-categories from \texorpdfstring{$\Ga$}{Gamma}-2-categories}
\label{sec:permutative-2-cats-from-Gamma-2-categories}

In this section we use the Grothendieck construction of
\cref{sec:Grothendieck-constr} to build permutative 2-categories from
$\Ga$-2-categories.  We first show in \cref{sec:gamma-and-A} that the
category of $\Ga$-2-categories admits a full and faithful embedding to
a category of symmetric monoidal diagrams on a permutative category
$\sA$.  We show that this extends to a functor on $\Ga$-lax
transformations in \cref{sec:extending-A-to-Ga-lax-maps} and use this
to give a functor
\[
P \cn \GIICat_l \to \PIICat
\]
in \cref{sec:definition-of-P}.

\begin{notn}\label{notn:SOplax-SOplax-nor}
For symmetric monoidal categories $\cA$ and $\cB$, let
$\mb{SOplax}(\cA, \cB)$ denote the category of symmetric oplax
monoidal functors from $\cA$ to $\cB$ and monoidal transformations
between them.  Let $\mb{SOplax}_{\mathrm{nor}}(\cA, \cB)$ denote the
full subcategory of \emph{normal} symmetric oplax monoidal functors,
i.e., those which strictly preserve the unit.
\end{notn}

\subsection{\texorpdfstring{$\Ga$}{Gamma} and the category
  \texorpdfstring{$\sA$}{A}}
\label{sec:gamma-and-A}

There is a monad on the category of sets which adds a basepoint.
Restricting this monad to a skeleton of finite sets, $\sN$, the
category $\sF$ is isomorphic to both its category of algebras and its
Kleisli category.  Using the general theory of Kleisli categories, one
can check the following universal property of $\Ga$ (see Leinster
\cite{Lei00Homotopy} for a full proof).

\begin{thm}\label{gammaunivprop}
  Let $\cC$ be any category with finite products and terminal object
  $1$.  Then there is an isomorphism of categories
  \[
  [\sF,\cC] \iso 
  \mb{SOplax}\big( (\sN,+,\emptyset), (\cC, \times, \ast) \big).
  \]
  This restricts to an isomorphism between the full subcategories of
  reduced diagrams and normal functors:
  \[
  \Ga\mh\cC = [\sF,\cC]_{\mathrm{red}} \iso 
  \mb{SOplax}_{\mathrm{nor}}\big( (\sN,+,\emptyset), (\cC, \times, \ast) \big).
  \]
\end{thm}

We can even go a bit further using the standard techniques of
2-dimensional algebra \cite{Lac02Codescent}.  Recall that for a
2-monad $T$ on a 2-category $K$, we have a variety of 2-categories of
algebras (by which we always mean strict algebras).
\begin{itemize}
\item There is the 2-category $K^{T}$ of $T$-algebras, strict
  $T$-algebra morphisms, and algebra 2-cells.  This 2-category is the
  enriched category of algebras \cite{Lac02Codescent} and is often
  denoted $\TAlgs$.
\item There is the 2-category $\TAlg$ of $T$-algebras,
  pseudo-$T$-algebra morphisms, and algebra 2-cells.
\item There is the 2-category $\TAlgl$ of $T$-algebras, lax
  $T$-algebra morphisms, and algebra 2-cells.  Reversing the direction
  of the structure cells for the morphisms will produce oplax
  $T$-algebra morphisms and the 2-category $\TAlgop$.
\end{itemize}
These 2-categories come with canonical inclusions as displayed below.
\[
\xy
(0,0)*+{\TAlgs}="0";
(30,0)*+{\TAlg}="1";
(50,10)*+{\TAlgl}="2";
(50,-10)*+{\TAlgop}="3";
{\ar "0"; "1" };
{\ar "1"; "2" };
{\ar "1"; "3" };
\endxy
\]
Under certain conditions on $K$ and $T$ (see \cite{Lac02Codescent}),
the inclusions $\TAlgs \hookrightarrow \TAlgm$ will have left
2-adjoints; here $\TAlgm$ indicates any of the 2-categories $\TAlg$,
$\TAlgl$, $\TAlgop$, although the particular hypotheses vary depending
on which variant is used.  We will generically denote this left
adjoint by $Q$.

When $K=\Cat$ and $T$ is the 2-monad for symmetric monoidal
categories, we write $\SMC$, with no subscript, for $\TAlg$ (the
pseudo-morphism variant).  We have a left 2-adjoint $Q: \SMCop
\rightarrow \SMCs$, and in particular an isomorphism of hom-categories
\[
 \SMCop(\cA,\cB) \cong \SMCs(Q\cA,\cB).
\]
Now $\SMCop(\cA,\cB)$ is exactly the same category that we denoted as
$\mb{SOplax}(\cA,\cB)$ above, so combining these results with
\cref{gammaunivprop} we obtain isomorphisms (with the monoidal
structure on $\cC$ being the cartesian one, when required)
\begin{eqn}
  \label{eqn:ff-to-qq-iso}
[\sF,\cC] \cong \mb{SOplax}(\sN, \cC) \cong \SMCs(Q(\sN),\cC).
\end{eqn}

When $K=\Cat_*$, the 2-category of pointed categories, there is another 2-monad $T$ whose 2-category of strict algebras and strict algebra maps is again $\SMCs$; this 2-monad can be seen as a quotient of the one on $\Cat$ which identifies the basepoint as the unit for the monoidal structure.  By construction, any algebra morphism for $T$ will necessarily preserve the basepoint, and hence the unit, so is a normal functor.  In particular, $\TAlgop$ is then the 2-category of symmetric monoidal categories, normal oplax symmetric monoidal functors, and monoidal transformations.  This $T$ also preserves filtered colimits,
and hence we have a left-adjoint $Q'\cn
\SMC_{nop} \to \SMC_{s}$.  Combining this with
\cref{gammaunivprop} we have
\begin{eqn}
  \label{eqn:ff-red-to-qq-nor-iso}
  \Ga\mh\cC = [\sF,\cC]_\mathrm{red} \iso
  \mb{SOplax}_\mathrm{nor}(\sN, \cC) \cong \SMC_{s}(Q'(\sN),\cC).
\end{eqn}

Finally, we note that every symmetric monoidal category is equivalent
as such to a permutative category \cite{Isb1969coherent}, i.e., a symmetric monoidal category
in which the underlying monoidal structure is strict.  This equivalence is an internal equivalence in
the 2-category of symmetric monoidal categories, symmetric monoidal
pseudofunctors, and monoidal transformations, in other words in the
2-category $\SMC$.  Thus we can find a permutative category
$\sA \simeq Q'(\sN)$, and therefore obtain an equivalence of categories
\[
\SMC(Q'(\sN),\cC) \simeq \SMC(\sA,\cC). 
\]
In particular, any symmetric monoidal functor $Q'(\sN) \to \cC$ (not
just the strict ones, as we obtained in \cref{eqn:ff-red-to-qq-nor-iso}) is
isomorphic to one arising from a symmetric monoidal functor $\sA \to
\cC$, giving full and faithful embeddings
\begin{align}
  \label{eqn:ff-embedding-Ga-2-cat-SM-A-2-cat}
  [\sF,\cC] & \hookrightarrow \SMC(\sA,\cC),\\
  \Ga\mh\cC = [\sF,\cC]_{\mathrm{red}} & \hookrightarrow \SMC(\sA,\cC).
\end{align}
We can always choose the symmetric monoidal equivalence $\sA \to Q'(\sN)$ to be normal, in which case the final embedding actually becomes a full and faithful embedding 
\[
\Ga\mh\cC \hookrightarrow \SMC_{\mathrm{nor}}(\sA,\cC).
\]

Mandell \cite{Man10Inverse} gives an explicit description of such a
permutative category $\sA$.
\begin{defn}
  \label{defn:diagram-A}
  The objects of $\sA$ are (possibly empty) tuples $\vec{m} = (m_1, \ldots, m_r)$,
  where each $m_i$ is a positive integer.  The morphism set
  \[
  \sA(\vec{m},\vec{n}) \subset \Set(\amalg_i \, \ul{m_i}, \amalg_j \, \ul{n_j})
  \]
  consists of those maps for which the inverse image of each
  $\ul{n_j}$ is either empty or contained in a single $\ul{m_i}$.  The
  monoidal product is given by concatenation, and the rest of the
  permutative structure is straightforward to spell out (see
  \cite{Man10Inverse}).
\end{defn}

We will need the following notation to describe maps in $\sA$.
\begin{notn}
  \label{notn:decomposing-phibar-partition-map}
  Suppose $\ol{\phi}\cn \vec{m} \to \vec{n}$ is a map in $\sA$. For each $i$, let
  \begin{eqn}
    \label{eqn:phibar-partition-indexing}
    \mathfrak{p}(\ol{\phi},i) = \{j \mid 
    \emptyset \not = \ol{\phi}^{\;-1}(\ul{n_j}) \subset \ul{m_i}\}
  \end{eqn}
  and let
  \begin{eqn}
    {\phi}_i = \ol{\phi}\big|_{\ul{m_i}} \cn \ul{m_i} \to \coprod_{j \in \mathfrak{p}(\ol{\phi}, i)} \ul{n_j}.
  \end{eqn}
  For each $j \in \mathfrak{p}(\ol{\phi},i)$, let
  \[
  {\phi}_{i,j} \cn \ul{m_i}_+ \to \ul{n_j}_+
  \]
 be the pointed map that restricts to $\phi_i$ on the elements in the preimage of $\ul{n_j}$ and sends everything else to the basepoint.
\end{notn}

\begin{rmk}
  \label{rmk:decomposing-phibar-partition-map-reorder}
  Note that every morphism $\ol{\phi}$ in $\sA$ can be decomposed as
  \[
  \coprod_i \ul{m_i} \to
  \coprod_i \coprod_{\mathfrak{p}(\ol{\phi},i)} \ul{n_j} \to
  \coprod_j \ul{n_j}
  \]
  where the first map is given by $\coprod_i{{\phi}_i}$ and the second
  is given by reordering and inclusion of indexing
  sets.  Each ${\phi}_i$ can be decomposed in $\sA$ as a partition of
  $\ul{m_i}$ followed by the disjoint union of maps given by
  ${\phi}_i\big|_{\ol{\phi}^{\;-1}(\ul{n_j})}$ and a reindexing:
    \[
  \ul{m_i} \hspace{1pc}\longrightarrow
  \coprod_{j\in \mathfrak{p}(\ol{\phi},i)}
  \ul{|{\phi}_i^{-1}(\ul{n_j})|} \quad \iso 
  \coprod_{j\in \mathfrak{p}(\ol{\phi},i)} 
  {\phi}_i^{-1}(\ul{n_j}) \hspace{1pc}\longrightarrow
  \coprod_{j\in \mathfrak{p}(\ol{\phi},i)} \ul{n_j}.
  \]
\end{rmk}


\begin{defn}
  \label{defn:functor-A}
  For $X \in \Ga\mh\cC$, the 
  symmetric strict monoidal functor $AX\cn\sA \to \cC$ is defined as follows.
  \begin{itemize}
  \item For $\vec{m} = (m_1 \ldots, m_r) \in \sA$,
    \[
    AX(\vec{m}) = \prod_{i}X(\ul{m_i}_+)
    \]
and $AX(\,)=X(\ul{0}_+)=\ast$,
  \item For $\ol{\phi} \in \sA(\vec{m}, \vec{n})$,
    \[
    \ol{\phi}_*: X(\vec{m}) \to X(\vec{n})
    \]
    is defined by the composite
    \[
    AX(\vec{m}) \to \prod_i\prod_{j \in \mathfrak{p}(\ol{\phi},i)} X(\ulp{n_j}) \fto{\tau} AX(\vec{n}).
    \]
    The first map is given, for each $i$, by components
    $({\phi}_{i,j})_*$.  The second map is given by permuting
    factors and inserting the unique maps $X(\ulp{0}) = * \to
    X(\ulp{n_j})$ for those $\ul{n_j}$ such that
    $\ol{\phi}^{\;-1}(\ul{n_j}) = \emptyset$.


  \end{itemize}

  Note that by construction $AX$ is a symmetric monoidal diagram with
  respect to concatenation in $\sA$ and the cartesian product in
  $\cC$.
\end{defn}

\subsection{Extending \texorpdfstring{$A$}{A} to \texorpdfstring{$\Ga$}{Gamma}-lax maps}
\label{sec:extending-A-to-Ga-lax-maps}

The assignment $X \mapsto AX$ is functorial in strict maps of
diagrams, but we will need more. Let $\oplus$ denote the monoidal product on $\sA$ given by concatenation.

\begin{prop}
  \label{prop:A-for-lax-maps}
  The functor $A\cn\GIICat \to \SMC(\sA,\IICat)$ extends to a functor
  \[
  A\cn\GIICat_l \to (\sA,\oplus) \mh\IICat_l.
  \]
\end{prop}
Before beginning the proof, we define the relevant extension of $A$.
\begin{defn}
  \label{defn:A-for-lax-maps}
  Let $X$ and $Y$ be $\Ga$-2-categories, and let $h\cn X \to Y$ be a
  $\Ga$-lax map. We define an $\sA$-lax map $Ah\cn AX \to AY$.
  Unpacking \cref{defn:D-lax-map} we see that such a map consists of,
  for each object $\vec{m}$ of $\sA$, a 2-functor
  \[ 
  AX(\vec{m}) \to AY(\vec{m})
  \]
  and for each morphism $\ol{\phi} \cn \vec{m}\to \vec{n}$ in $\sA$, a
  2-natural transformation
  \[
  h_{\ol{\phi}} \cn \ol{\phi}_* h_{\vec{m}} \Rightarrow 
  h_{\vec{n}} \ol{\phi}_*.
  \]

  The 2-functor $(Ah)_{\vec{m}}$ is given by
  \begin{eqn*}
    AX(\vec{m}) = 
    \prod_i X(\ulp{m_i}) \fto{\ \prod h_{\ulp{m_i}}\ } 
    \prod_i Y(\ulp{m_i})=AY(\vec{m}).
  \end{eqn*}
We define the 2-natural transformation $(Ah)_{\ol{\phi}}$ as
  follows. For each $j \in \mathfrak{p}(\ol{\phi},i)$, ${\phi}_{i,j}$
  is a map in $\sF$ and therefore since $h$ is $\Ga$-lax we have
  2-natural transformations
  \[
  h_{{\phi}_{i,j}} \cn 
  {\phi}_{i,j} h_{\ulp{m_i}} \Rightarrow
  h_{\ulp{n_j}} {\phi}_{i,j}.
  \]
  Taking the product over $j \in \mathfrak{p}
  = \mathfrak{p}(\ol{\phi},i)$ defines 2-natural transformations $h_{{\phi_i}}$:
  \begin{equation}
    \label{eq:defn-h-barphi}
    h_{{\phi}_i} \cn
    {\phi}_i h_{\ulp{m_i}} \Rightarrow 
    \left(\prod_{j \in \mathfrak{p}} h_{\ulp{n_j}}\right) {\phi_i}.
  \end{equation}
  Taking the product over $i$ gives the top square of the diagram below.
  \begin{eqn}
    \label{eqn:Ah-phibar-2-steps}
    \begin{xy}
      (-7,25)*+{AX(\vec{m})}="A";
      (45,25)*+{AY(\vec{m})}="A1";
      (-7,0)*+{\prod_i \prod_{j \in J} X(\ul{n_j}_+)}="B";
      (45,0)*+{\prod_i \prod_{j \in J} Y(\ul{n_j}_+)}="B1";
      (-7,-25)*+{AX(\vec{n})}="C";
      (45,-25)*+{AY(\vec{n})}="C1";
      {\ar^{h_{\vec{m}}} "A"; "A1"};
      {\ar^{\prod_i \prod_{j \in \mathfrak{p}} h_{\ulp{n_j}}} "B"; "B1"};
      {\ar^{h_{\vec{n}}} "C"; "C1"};
      {\ar_{\prod_i {\phi}_i} "A"; "B"};
      {\ar_{\prod_i {\phi}_i} "A1"; "B1"};
      {\ar_{\tau} "B"; "C"};
      {\ar_{\tau} "B1"; "C1"};
      {\ar@{=>}_{\prod_i h_{{\phi}_i}} (22,16)*+{}="tc0"; "tc0"+(-6,-5) };
      (20,-11.5)*{=};
    \end{xy}
  \end{eqn}  
  In the bottom square the maps $\tau$ are given by permuting factors
  in the cartesian product and hence the bottom square commutes.
  We define $Ah_{\ol{\phi}}$ as the pasting of these two cells.  The
  2-naturality of $Ah_{\ol{\phi}}$ is immediate by the 2-naturality of
  each $h_{{\phi}_{i,j}}$ and $\tau$.
\end{defn}

\begin{proof}[Proof of \cref{prop:A-for-lax-maps}]\proofof{prop:A-for-lax-maps}
  We verify that the extension of $A$ given in
  \cref{defn:A-for-lax-maps} defines a functor as claimed.
  This consists of verifying the following equality of pasting diagrams for
  $\Ga$-lax maps
  \[
  X \fto{h} Y \fto{k} Z.
  \]
  \begin{eqn*}
    \begin{xy}
      (-7,25)*+{AX(\vec{m})}="A";
      (45,25)*+{AY(\vec{m})}="A1";
      (97,25)*+{AZ(\vec{m})}="A2";
      (-7,0)*+{\prod_i \prod_{j \in J} X(\ul{n_j}_+)}="B";
      (45,0)*+{\prod_i \prod_{j \in J} Y(\ul{n_j}_+)}="B1";
      (97,0)*+{\prod_i \prod_{j \in J} Z(\ul{n_j}_+)}="B2";
      {\ar^{h_{\vec{m}}} "A"; "A1"};
      {\ar^{k_{\vec{m}}} "A1"; "A2"};
      {\ar_{\prod_i \prod_{j \in \mathfrak{p}} h_{\ulp{n_j}}} "B"; "B1"};
      {\ar_{\prod_i \prod_{j \in \mathfrak{p}} k_{\ulp{n_j}}} "B1"; "B2"};
      {\ar_{\prod_i {\phi}_i} "A"; "B"};
      {\ar_{\prod_i {\phi}_i} "A1"; "B1"};
      {\ar^{\prod_i {\phi}_i} "A2"; "B2"};
      {\ar@{=>}_{\prod_i h_{{\phi}_i}} (22,16)*+{}="tc0"; "tc0"+(-5,-5) };
      {\ar@{=>}_{\prod_i k_{{\phi}_i}} (74,16)*+{}="tc0"; "tc0"+(-5,-5) };
      %
      {\ar@{=} (45,-5)="a"; "a"+(0,-3)};
      (-7,-15)*+{AX(\vec{m})}="A";
      (97,-15)*+{AZ(\vec{m})}="A2";
      (-7,-40)*+{\prod_i \prod_{j \in J} X(\ul{n_j}_+)}="B";
      (97,-40)*+{\prod_i \prod_{j \in J} Z(\ul{n_j}_+)}="B2";
      {\ar^{(kh)_{\vec{m}}} "A"; "A2"};
      {\ar_{\prod_i \prod_{j \in \mathfrak{p}} (kh)_{\ulp{n_j}}} "B"; "B2"};
      {\ar_{\prod_i {\phi}_i} "A"; "B"};
      {\ar^{\prod_i {\phi}_i} "A2"; "B2"};
      {\ar@{=>}_{\prod_i (kh)_{{\phi}_i}} (50,-24)*+{}="tc0"; "tc0"+(-9,-5) };
    \end{xy}
  \end{eqn*}
  But this equality is immediate because composition is strict in the
  (1-)category $\GaIICat_l$ (see
  \cref{rmk:D2Cat-lax-is-a-category}).
\end{proof}

Recalling the notions of symmetric monoidal diagram and monoidal lax
map from \cref{sec:SM-diagrams}, the following is
immediate from the definition of $A$.
\begin{prop}
  \label{prop:Ah-has-monoidal-structure}
  Let $X$ and $Y$ be $\Ga$-2-categories and $h\colon X\to Y$ be a
  $\Ga$-lax map. 
  Then the $\sA$-lax map $Ah\cn AX \to AY$ is monoidal.
\end{prop}

\subsection{Definition of \texorpdfstring{$P$}{P}}
\label{sec:definition-of-P}

Composing the functor $A\cn \GaIICat_l \to \sA\mh\IICat_l$ of
\cref{prop:A-for-lax-maps} with the Grothendieck construction of
\cref{sec:Grothendieck-constr}, we obtain a functor
\[
P = \sA \wre (A-) \cn \GIICat_l \to \IICat.
\]

Applying
\cref{prop:Ah-has-monoidal-structure,prop:symm-D-lax-to-sm-2functor},
we have the following refinement.
\begin{thm} 
  \label{thm:P-lax-functorial}
  The assignment on objects $PX = \sA \wre AX$ extends to a functor
  \[ 
  P\colon  \GIICat_l \to \PIICat
  \]
  from the category of strict $\Ga$-2-categories and $\Ga$-lax maps to
  the category of permutative 2-categories and strict (2-)functors.
\end{thm}

In the remainder of this section we give an explicit description of
the objects, 1-cells, and 2-cells of $PX$, and then prove some basic homotopical results about $P$.  The objects of $PX$ are
pairs $[\vec{m},\vec{x}]$ for $\vec{m} = (m_i) \in \sA$ and
$\vec{x} = (x_i) \in AX(\vec{m})$.  The 1- and 2-cells in $PX$ are
pairs as below:
\[
\begin{xy}
  (0,0)*+{[\vec{m}, \vec{x}]}="A";
  (40,0)*+{[\vec{n}, \vec{y}]}="B";
  {\ar@/^1pc/^{[\ol{\phi}, \vec{f}]} "A"; "B"};
  {\ar@/_1pc/_{[\ol{\phi}, \vec{g}]} "A"; "B"};
  {\ar@{=>}^{[\ol{\phi}, \vec{\alpha}]} (17,3)*+{}; (17,-3)*+{} };
\end{xy}
\]
where $\ol{\phi}\cn \vec{m} \to \vec{n}$ is a map in $\sA$ and $\vec{f},
\vec{g}, \vec{\al}$ are cells in $AX(\vec{n})$:
\[
\begin{xy}
  (0,0)*+{\ol{\phi}_*\vec{x}}="A";
  (40,0)*+{\vec{y}.}="B";
  {\ar@/^1pc/^{\vec{f}} "A"; "B"};
  {\ar@/_1pc/_{\vec{g}} "A"; "B"};
  {\ar@{=>}^{\vec{\alpha}} (20,3)*+{}; (20,-3)*+{} };
\end{xy}
\]

\begin{rmk}
  \label{rmk:grothendieck-composition-some-ids}
  The definition of composition in $PX$ immediately yields the
  following special cases for $\ol{\phi}\cn \vec{m} \to \vec{n}$ in
  $\sA$ and appropriately composable 1-cells $\vec{f}$ in $AX(\vec{m})$ and
  $\vec{g}$ in $AX(\vec{n})$:
  \begin{align*}
    \bsb{\id}{\vec{g}} \bsb{\ol{\phi}}{\id} & = 
    \bsb{\ol{\phi}}{\vec{g}},\\
    \bsb{\ol{\phi}}{\vec{g}} \bsb{\id}{\vec{f}} & =
    \bsb{\ol{\phi}}{\vec{g} \circ \ol{\phi}_* \vec{f}}.
  \end{align*}
\end{rmk}

\begin{rmk}
  \label{rmk:description-Ph}
For a $\Ga$-lax map $h\cn X \to Y$, the functor $Ph\cn PX \to PY$ is
given explicitly as:\\
on 0-cells
\[
Ph\bsb{\vec{m}}{\vec{x}} = \bsb{\vec{m}}{Ah_{\vec{m}}(\vec{x})},
\]
on 1-cells
\[
Ph\bsb{\ol{\phi}}{\vec{f}} = \bsb{\ol{\phi}}{Ah_{\vec{n}}(\vec{f}) \circ Ah_{\ol{\phi}}},
\]
on 2-cells
\[
Ph\bsb{\ol{\phi}}{\vec{\al}} = \bsb{\ol{\phi}}{Ah_{\vec{n}}(\vec{\al}) * 1_{Ah_{\ol{\phi}}}}.
\]
These are given by the following 0-, 1-, and 2-cells in $AY$.
\[
\begin{xy}
  (-30,0)*+{\ol{\phi}_*Ah_{\vec{m}}(\vec{x})}="Z";
  (0,0)*+{Ah_{\vec{n}}(\ol{\phi}_* x)}="A";
  (40,0)*+{Ah_{\vec{n}}(\vec{y})}="B";
  {\ar^-{Ah_{\ol{\phi}}} "Z"; "A"};
  {\ar@/^1pc/^{Ah_{\vec{n}} (\vec{f})} "A"; "B"};
  {\ar@/_1pc/_{Ah_{\vec{n}} (\vec{g})} "A"; "B"};
  {\ar@{=>}^{Ah_{\vec{n}} (\vec{\alpha})} (17,3)*+{}; (17,-3)*+{} };
\end{xy}  
\]
\end{rmk}

The following two results make use of the equivalence
\begin{eqn}
\label{eq:Thomason-result}
  \hocolim_{\sA} N(AX) \fto{\hty}  N(\sA \wre AX)
\end{eqn}
due to Thomason for diagrams of categories and generalized to diagrams
of bicategories by Carrasco-Cegarra-Garz\'on \cite[Theorem
  7.3]{CCG10Nerves}.

\begin{lem}
  \label{lem:X1-to-PX-weak-equiv}
  Let $X$ be a special $\Ga$-2-category.  Then the inclusion
  \[
  X(\ulp{1}) \to PX
  \]
  is a weak equivalence of 2-categories.
\end{lem}
\begin{proof}\proofof{lem:X1-to-PX-weak-equiv}
  We have the following chain of weak equivalences
  \[
  X(\ulp{1}) \fto{\hty} \hocolim_{\sN} N(X) \fto{\hty} \hocolim_{\sA}
  N(AX) \fto{\hty}  N(\sA \wre AX) = N(PX).
  \]
  The last equivalence is \cref{eq:Thomason-result} and the others
  follow from changing indexing categories as in the argument of
  \cite[Theorem 5.3]{Man10Inverse}.
\end{proof}

\begin{prop}
  \label{prop:P-preserves-weak-equivalences}
  The functor $P\cn \GaIICat_{l} \to \PIICat$ 
  sends levelwise weak equivalences of $\Ga\mh2$-categories to weak
  equivalences of permutative 2-categories and therefore is a relative functor
  \[
  (\GaIICat_{l}, \cW) \to (\PIICat, \cW).
  \]
\end{prop}
\begin{proof}\proofof{prop:P-preserves-weak-equivalences}
  This is immediate from \cref{eq:Thomason-result}.
\end{proof}


\section{\texorpdfstring{$K$}{K}-theory constructions}
\label{sec:K-theory-constructions}

In this section we give two constructions of $K$-theory for
2-categories.  The first, in \cref{sec:K-epz-fluffy}, is a
construction for permutative Gray-monoids and has a counit described
in \cref{sec:counit}. The second, in \cref{sec:K-epz-strict}, is a
construction for permutative 2-categories, taking advantage of their
additional strictness.  Using both of these, we prove in
\cref{thm:main-css-2} that the homotopy theory of permutative
Gray-monoids is equivalent to that of permutative 2-categories.  The
stricter construction is also essential for the triangle identities of
\cref{sec:K-epz-eta,sec:K-epz-K-P-eta}.

\subsection{\texorpdfstring{$K$}{K}-theory for permutative Gray-monoids}
\label{sec:K-epz-fluffy}

Recall that $\PGMnop$ denotes the category of permutative Gray-monoids
and normal, oplax functors between them. We define a functor
\[
\ko \colon \PGMnop \to \GIICat
\]
using a construction very similar to that of \cite[\S
5.2]{Oso10Spectra}.

Let $(\cC, \oplus, e, \be)$ be a permutative Gray-monoid. The
$\Gamma$-2-category $\ko{\cC}$ is given in \cref{defn:Ko,defn:Ko-phi}.
In \cref{prop:functoriality-of-K} we show that $\Ko$ is functorial
with respect to normal oplax functors of permutative Gray-monoids, and
in \cref{prop:ko-is-special} we show that $\Ko\cC$ is a special
$\Ga$-2-category.

\begin{construction}
  \label{defn:Ko}
  Let $\ulp{n} \in \sF$ be a finite pointed set.  We define a
  2-category $\ko{\cC}(\ulp{n})$ as follows.

  \begin{enumerate}
  \item Objects are given by collections of the form $\{x_s,
    c_{s,t}\}$
  where $x_s$ is an object of $\cC$ for each $s \subset {n}$ and
  $c_{s,t} : x_{s \cup t}\rightarrow x_s \opl x_t$ is a 1-cell of
  $\cC$ for each such pair $(s,t)$ with $s \cap t = \emptyset$. These
  collections are required to satisfy the following axioms:

  \begin{enumerate}
  \item \label{axiomobj1} $x_{\emptyset}=e$;
  \item $c_{\emptyset,s}=\id_{x_s}=c_{s,\emptyset}$;
  \item for every triple $(s,t,u)$ of pairwise disjoint subsets the
    diagram below commutes;
    \begin{equation}\label{astu}
      \xymatrixcolsep{5pc}\xymatrixrowsep{3.5pc}
      \def\labelstyle{\displaystyle}
      \xymatrix{
        x_{s\cup t\cup u} \ar[d]_{c_{s\cup t,u}} \ar[r]^-{c _{s,t\cup
            u}}
        & x_s\opl x_{t\cup u} \ar[d]^{\id\opl c_{t,u}}\\
        x_{s\cup t}\opl x_u \ar[r]_-{c _{s,t}\opl \id} &
        x _s \opl x_t \opl x_u}
    \end{equation}

  \item \label{axiomobj4} for every pair of disjoint subsets $s,t$, the
    diagram below commutes.
    \begin{equation}\label{ats}
      \xymatrixcolsep{4pc}
      \def\labelstyle{\displaystyle}
      \xymatrix{
        x_{s\cup t} \ar@{=}[d] \ar[r]^-{c_{s,t}}
        & x_s\opl x_t \ar[d]^{\beta_{x_s,x_t}}\\
        x_{t\cup s} \ar[r]_-{c _{t,s}}
        & x_t\opl x_s
      }
    \end{equation}
  \end{enumerate}

\item A 1-morphism from $\obj {x}{c}$ to $\obj {x'}{c'}$ is a
  collection $\obj {f}{\ga} $, where $s,t$ are as above; $f_s: x_s
  \rightarrow x' _s$ is a 1-morphism in $\cC$ and $\ga _{s,t}$ is an
  invertible 2-cell:
  \[
  \def\labelstyle{\displaystyle}
  \xy
  (0,0)*+{\xst}="00";
  (30,0)*+{\xs\opl\xt}="20";
  (0,-30)*+{\xst'}="02";
  (30,-30)*+{\xs'\opl\xt'}="22";
  (30,-15)*+{\xs'\opl\xt}="21";
  {\ar^{\cst} "00";"20"};
  {\ar_{\fst} "00";"02"};
  {\ar^{\fs\opl\id} "20";"21"};
  {\ar^{\id\opl\ft} "21";"22"};
  {\ar_{\cst'} "02";"22"};
  {\ar@{=>}_{\gast} (18,-12)*{};(12,-18)*{}};
  \endxy
  \]
  These are required to satisfy the following axioms:
  \begin{enumerate}
  \item $f_{\emptyset}=\id_e$;

  \item $\ga_{\emptyset,s}$ is the identity 2-morphism
    \[
    (f_s \opl \id_e) \circ \id_{x_s}=f_s \Longrightarrow f_s=\id_{x'_s}\circ f_s
    \]
    and similarly for $\ga_{s,\emptyset}$;

  \item for every triple of pairwise disjoint subsets $s,t,u$ we have
    the following equality of pasting diagrams:
    \begin{equation}\label{phistu}
      \begin{gathered}
      \def\labelstyle{\displaystyle}
      \xy
      (0,0)*+{\xstu}="00";
      (40,0)*+{\xs\opl\xtu}="10";
      (90,0)*+{\xs\opl\xt\opl\xu}="20";
      (0,-54)*+{\xstu'}="03";
      (40,-18)*+{\xs'\opl\xtu}="11";
      (40,-54)*+{\xs'\opl\xtu'}="13";
      (90,-18)*+{\xs'\opl\xt\opl\xu}="21";
      (90,-36)*+{\xs'\opl\xt'\opl\xu}="22";
      (90,-54)*+{\xs'\opl\xt'\opl\xu'}="23";
      {\ar^{\csctu} "00";"10"};
      {\ar^{\id \opl\ctu} "10";"20"};
      {\ar_{\id \opl\ctu} "11";"21"};
      {\ar_{\csctu'} "03";"13"};
      {\ar_{\id \opl \ctu'} "13";"23"};
      {\ar_{\fstu} "00";"03"};
      {\ar|{\fs\opl\id} "10";"11"};
      {\ar|{\id\opl\ftu} "11";"13"};
      {\ar^{\fs\opl\id} "20";"21"};
      {\ar^{\id\opl\ft\opl\id} "21";"22"};
      {\ar^{\id\opl\fu} "22";"23"};
      {\ar@{=>}_{\gasctu} (24,-24)*{};(16,-30)*{}};
      {\ar@{=>}_{\Sigma} (67,-5)*{};(59,-11)*{}};
      {\ar@{=>}_{\id\opl\gatu} (67,-32)*{};(59,-38)*{}};
      (40,-62.5)*{=};
      (0,-70)*+{\xstu}="04";
      (40,-70)*+{\xst\opl\xu}="14";
      (90,-70)*+{\xs\opl\xt\opl\xu}="24";
      (0,-124)*+{\xstu'}="07";
      (40,-106)*+{\xst'\opl\xu}="16";
      (40,-124)*+{\xst'\opl\xu'}="17";
      (90,-88)*+{\xs'\opl\xt\opl\xu}="25";
      (90,-106)*+{\xs'\opl\xt'\opl\xu}="26";
      (90,-124)*+{\xs'\opl\xt'\opl\xu'}="27";
      {\ar^{\cstcu} "04";"14"};
      {\ar^{\cst \opl\id} "14";"24"};
      {\ar^{\cst' \opl\id} "16";"26"};
      {\ar_{\cstcu'} "07";"17"};
      {\ar_{\cst' \opl \id} "17";"27"};
      {\ar_{\fstu} "04";"07"};
      {\ar|{\fst\opl\id} "14";"16"};
      {\ar|{\id\opl\fu} "16";"17"};
      {\ar^{\fs\opl\id} "24";"25"};
      {\ar^{\id\opl\ft\opl\id} "25";"26"};
      {\ar^{\id\opl\fu} "26";"27"};
      {\ar@{=>}_{\gastcu} (24,-94)*{};(16,-100)*{}};
      {\ar@{=>}_{\Sigma^{-1}} (67,-111)*{};(59,-117)*{}};
      {\ar@{=>}_{\gast\opl\id} (67,-84)*{};(59,-90)*{}};
      \endxy
    \end{gathered}
  \end{equation}

  \item for every pair of disjoint subsets $s,t$ we have the following
    equality of pasting diagrams:
    \begin{equation}\label{phits}
      \def\labelstyle{\displaystyle}
      \xy
      (0,0)*+{\xst}="00";
      (30,0)*+{\xs\opl\xt}="20";
      (75,0)*+{\xt\opl\xs}="40";
      (0,-30)*+{\xst'}="02";
      (30,-30)*+{\xs'\opl\xt'}="22";
      (30,-15)*+{\xs'\opl\xt}="21";
      (60,-15)*+{\xt\opl\xs'}="31";
      (90,-15)*+{\xt'\opl\xs}="41";
      (75,-30)*+{\xt'\opl\xs'}="42";
      {\ar^{\cst} "00";"20"};
      {\ar^{\be} "20";"40"};
      {\ar_{\fst} "00";"02"};
      {\ar^{\fs\opl\id} "20";"21"};
      {\ar^{\id\opl\ft} "21";"22"};
      {\ar_{\cst'} "02";"22"};
      {\ar_{\be} "22";"42"};
      {\ar_{\id\opl\fs} "40";"31"};
      {\ar_{\ft\opl\id} "31";"42"};
      {\ar^{\ft\opl\id} "40";"41"};
      {\ar^{\id\opl\fs} "41";"42"};
      {\ar_{\be} "21";"31"};
      {\ar@{=>}_{\gast} (18,-12)*{};(12,-18)*{}};
      {\ar@{=>}_{\Sigma} (77,-15)*{};(73,-15)*{}};
      (48,-7.5)*{=};
      (48,-22.5)*{=};
      (45,-37.5)*{=};
      (30,-45)*+{\xts}="05";
      (60,-45)*+{\xt\opl\xs}="25";
      (30,-75)*+{\xts'}="07";
      (60,-75)*+{\xt'\opl\xs'}="27";
      (60,-60)*+{\xt'\opl\xs}="26";
      {\ar^{\cts} "05";"25"};
      {\ar_{\fts} "05";"07"};
      {\ar^{\ft\opl\id} "25";"26"};
      {\ar^{\id\opl\fs} "26";"27"};
      {\ar_{\cts'} "07";"27"};
      {\ar@{=>}_{\gats} (48,-57)*{};(42,-63)*{}};
      \endxy
    \end{equation}
  \end{enumerate}

\item For 1-morphisms $\{ f_s, \ga _{s,t}\} ,\{ g_s, \de _{s,t}\} : \{
  x_s, c_{s,t}\} \rightarrow \{ x'_s, c' _{s,t}\}$, a 2-morphism
  between them is a collection $\{ \al _s\}$ of 2-morphisms in $\cC$, $\al
  _s :f_s \Rightarrow g_s$, such that for all $s,t$ as above we have
  the following equality of pasting diagrams.
  \begin{equation}\label{bst}
    \def\labelstyle{\displaystyle}
    \xy
    (0,0)*+{\xst}="00";
    (40,0)*+{\xs\opl\xt}="20";
    (0,-44)*+{\xst'}="02";
    (40,-44)*+{\xs'\opl\xt'}="22";
    (40,-22)*+{\xs'\opl\xt}="21";
    {\ar^{\cst} "00";"20"};
    {\ar@/^1.4pc/^{\fst} "00";"02"};
    {\ar@/_1.4pc/_{\gst} "00";"02"};
    {\ar^{\fs\opl\id} "20";"21"};
    {\ar^{\id\opl\ft} "21";"22"};
    {\ar_{\cst'} "02";"22"};
    {\ar@{=>}^{\gast} (23,-19)*{};(19,-25)*{}};
    {\ar@{=>}^{\alst} (2,-22)*{};(-2,-22)*{}};
    (55,-22)*{=};
    (70,0)*+{\xst}="00";
    (110,0)*+{\xs\opl\xt}="20";
    (70,-44)*+{\xst'}="02";
    (110,-44)*+{\xs'\opl\xt'}="22";
    (110,-22)*+{\xs'\opl\xt}="21";
    {\ar^{\cst} "00";"20"};
    {\ar_{\gst} "00";"02"};
    {\ar@/^1.6pc/^{\fs\opl\id} "20";"21"};
    {\ar@/^1.6pc/^{\id\opl\ft} "21";"22"};
    {\ar@/_1.6pc/_{\gs\opl\id} "20";"21"};
    {\ar@/_1.6pc/_{\id\opl\gt} "21";"22"};
    {\ar_{\cst'} "02";"22"};
    {\ar@{=>}_{\dest} (89,-19)*{};(85,-25)*{}};
    {\ar@{=>}_{\als\opl 1} (112,-11)*{};(108,-11)*{}};
    {\ar@{=>}^{1\opl \alt} (112,-31)*{};(108,-31)*{}};
    \endxy
  \end{equation}

\end{enumerate}

Vertical composition of 2-morphisms is defined componentwise.
Horizontal composition of 1-cells and  2-cells  given by:
\begin{align*}
\obj{g}{\de} \circ \obj{f}{\ga} & = \{ g_s \circ f_s, (\de \diamond \ga)_{s,t} \}\\
\{ \al '_s\} * \{ \al _s\} & = \{ \al '_s * \al _s\},
\end{align*}
where the 2-morphism $(\de \diamond \ga)_{s,t}$ is defined by the
pasting diagram below.
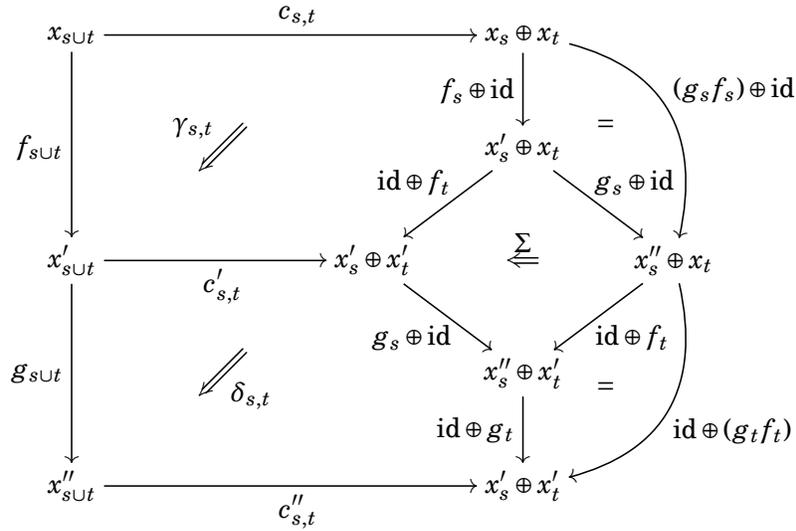
\[
\def\labelstyle{\displaystyle}
\xy
(0,0)*+{\xst}="00";
(60,0)*+{\xs\opl\xt}="20";
(0,-30)*+{\xst'}="02";
(40,-30)*+{\xs'\opl\xt'}="12";
(80,-30)*+{\xs''\opl\xt}="32";
(0,-60)*+{\xst''}="04";
(60,-60)*+{\xs'\opl\xt'}="24";
(60,-15)*+{\xs'\opl\xt}="21";
(60,-45)*+{\xs''\opl\xt'}="23";
{\ar^{\cst} "00";"20"};
{\ar_{\fst} "00";"02"};
{\ar_{\gst} "02";"04"};
{\ar_{\fs\opl\id} "20";"21"};
{\ar_{\id\opl\ft} "21";"12"};
{\ar_{\cst'} "02";"12"};
{\ar_{\cst''} "04";"24"};
{\ar_{\gs\opl\id} "12";"23"};
{\ar_{\id\opl\gt} "23";"24"};
{\ar^{\gs\opl\id} "21";"32"};
{\ar^{\id\opl\ft} "32";"23"};
{\ar@/^2.5pc/^{(\gs\fs)\opl\id} "20";"32"};
{\ar@/^2.5pc/^{\id\opl(\gt\ft)} "32";"24"};
{\ar@{=>}_{\gast} (23,-12)*{};(17,-18)*{}};
{\ar@{=>}^{\dest} (23,-42)*{};(17,-48)*{}};
{\ar@{=>}_{\Sigma} (62,-30)*{};(58,-30)*{}};
(71,-12)*{=};
(71,-47)*{=};
\endxy
\]
\end{construction}

Note that the axioms imply that $\Ko\cC(\ulp{0})$ has a unique object given by $e$, and only identity 1- and 2-cells. 

\begin{rmk}
  \label{rmk:objs-in-K-thy-as-functions}
  In \cref{sec:eta-and-triangle-ids} we will use the fact that the
  0-cells of $\Ko\cC(\ulp{n})$ can be thought of as functions on
  subsets and disjoint pairs of subsets of $\ul{n}$:
  \[
  \iicb{1.2}{(s \subset \ul{n})}{x_s}{\stdisj{\ul{n}}}{c_{s,t}}.
  \]
  The 1- and 2-cells can be thought of similarly.
\end{rmk}

\begin{construction}
  \label{defn:Ko-phi}
Let $\phi \colon \ulp{m} \to \ulp{n}$ be a map in $\sF$. We define a
strict 2-functor
\[
\phi_{\ast} \colon \ko{\cC}(\ulp{m}) \to \ko{\cC}(\ulp{n})
\]
as follows
\begin{align*}
  \phi_* \{x_s, c_{s,t}\}
  & =  \{x^{\phi}_{u}, c^{\phi}_ {u,v}\} = \{x_{\phi^{-1}(u)}, c_ {\phi^{-1}(u), \phi^{-1}(v)}\}\\
  \phi_* \{f_s, \ga_{s,t}\}
  & =  \{f_u^{\phi}, \ga^{\phi}_ {u, v}\} = \{f_{\phi^{-1}(u)}, \ga_ {\phi^{-1}(u), \phi^{-1}(v)}\}\\
  \phi_* \{\al_s\}
  & = \{\al^{\phi}_{u}\} = \{\al_{\phi^{-1}(u)}\},\\
\end{align*}
where $s,t$ range over disjoint subsets of $\ul{m}$ and $u,v$ range
over disjoint subsets of $\ul{n}$.  Since $\phi$ is basepoint
preserving, $\phi^{-1} (u)$ does not contain the basepoint and it is
an allowed indexing subset of $\ul{m}$. Since $u$ and $v$ are
disjoint, their preimages under $\phi$ are also disjoint. Note that if
$\psi \colon \ulp{n} \to \ulp{p}$ is another map in $\sF$ then we have
$(\psi \phi) _{\ast}= \psi _{\ast} \phi _{\ast}$, so $\ko{\cC}$ is a $\Ga$-2-category.
\end{construction}

\begin{prop}
\label{prop:functoriality-of-K}
The construction above gives the object assignment for a functor
\[
\ko \cn \PGMnop \to \GIICat.
\]
\end{prop}

\begin{proof}\proofof{prop:functoriality-of-K}
  Let $(F,\tha)$ be a normal oplax functor between the permutative
  Gray-monoids $\cC$ and $\cD$. We first define a 2-functor
  $\ko{F}_{\ulp{n}}\cn \ko{\cC}(\ulp{n}) \to \ko{\cD}(\ulp{n})$.  This
  is given for objects, 1-cells, and 2-cells as:
  \begin{align*}
    \obj{x}{c} &\longmapsto \{F(x_s),\tha \circ F(c_{s,t})\}\\
    \obj{f}{\ga} &\longmapsto  \{F(f_s), 1_{\tha}\ast F(\ga_{s,t})\}\\
    \{ \al _s\}  &\longmapsto  \{F(\al_s)\},
  \end{align*}
  where $1_{\tha}  \ast F(\ga_{s,t})$ is the 2-cell defined by the
  pasting diagram below.
  \[
  \def\labelstyle{\displaystyle}
  \xy
  (0,0)*+{F(\xst)}="00";
  (30,0)*+{F(\xs\opl\xt)}="20";
  (0,-50)*+{F(\xst')}="02";
  (30,-50)*+{F(\xs'\opl\xt')}="22";
  (30,-25)*+{F(\xs'\opl\xt)}="21";
  (60,0)*+{F(\xs)\opl F(\xt)}="30";
  (60,-50)*+{F(\xs')\opl F(\xt').}="32";
  (60,-25)*+{F(\xs')\opl F(\xt)}="31";
  {\ar^{F(\cst)} "00";"20"};
  {\ar_{F(\fst)} "00";"02"};
  {\ar|{F(\fs\opl\id)} "20";"21"};
  {\ar|{F(\id\opl\ft)} "21";"22"};
  {\ar^{F(\fs)\opl\id} "30";"31"};
  {\ar^{\id\opl F(\ft)} "31";"32"};
  {\ar^{\tha} "20";"30"};
  {\ar^{\tha} "21";"31"};
  {\ar_{\tha} "22";"32"};
  {\ar_{F(\cst')} "02";"22"};
  {\ar@{=>}_{F(\gast)} (18,-20)*{};(12,-30)*{}};
  (45,-12.5)*{=};
  (45,-37.5)*{=};
  \endxy
  \]
  Now one must verify that these assignments send $k$-cells to
  $k$-cells for $k=0,1,2$, and then check 2-functoriality. This is
  largely routine using the permutative Gray-monoid axioms (including
  axioms for Gray tensor product and its interaction with $\be$) and
  normal oplax functor axioms. As an example, we include the
  verification of the axiom in \cref{astu} as part of checking that
  $\{F(x_s),\tha \circ F(c_{s,t})\}$ is an object of
  $\ko{\cD}(\ulp{n})$.  The other verifications are similarly
  straightforward.  We must verify that the following diagram of
  1-morphisms in $\cD$ commutes.
  \[
  \def\labelstyle{\displaystyle}
  \xy
  (0,0)*+{F(\xstu)}="00";
  (50,0)*+{F(\xs\opl\xtu)}="10";
  (100,0)*+{F(\xs)\opl F(\xtu)}="20";
  (0,-20)*+{F(\xst\opl\xu)}="01";
  (50,-20)*+{F(\xs\opl\xt\opl\xu)}="11";
  (100,-20)*+{F(\xs)\opl F(\xt\opl\xu)}="21";
  (0,-40)*+{F(\xst)\opl F(\xu)}="02";
  (50,-40)*+{F(\xs\opl\xt)\opl F(\xu)}="12";
  (100,-40)*+{F(\xs)\opl F(\xt)\opl F(\xu)}="22";
  {\ar^{F(\csctu)} "00";"10"};
  {\ar^{\tha} "10";"20"};
  {\ar^{F(\cst \opl \id)} "01";"11"};
  {\ar^{\tha} "11";"21"};
  {\ar_{F(\cst) \opl \id} "02";"12"};
  {\ar_{\tha\opl\id} "12";"22"};
  {\ar_{F(\csctu)} "00";"01"};
  {\ar_{\tha} "01";"02"};
  {\ar^{F(\id \opl \ctu)} "10";"11"};
  {\ar^{\tha} "11";"12"};
  {\ar^{\id\opl F(\ctu)} "20";"21"};
  {\ar^{\id\opl\tha} "21";"22"};
  \endxy
  \]
  The upper left square commutes because $F$ is a 2-functor, and the
  diagram commutes before applying $F$. The upper right and lower left
  squares commute since $\tha$ is a 2-natural transformation. Finally,
  the commutativity of the lower right square is one of the axioms for
  normal oplax functors.

  Now the collection of 2-functors $\ko{F}_{\ulp{n}}$ is natural with
  respect to $\ulp{n}$, and therefore $\ko{F}$ is a strict map of
  $\Ga$-2-categories. Finally, it is easy to check that $\ko$ is
  functorial.
\end{proof}

\begin{rmk}
  \label{rmk:unique-cells-given-by-partition}
  Let $\cC$ be a permutative Gray-monoid, $\ulp{m} \in \sF$, and
  $\{x,c\} \in \Ko \cC(\ulp{m})$.  Given a subset $s \subset \ul{m}$
  and a partition $s = s_1 \cup \cdots \cup s_a$, there are, \it{a priori},
  many 1-cells
  \[
  x_s \to x_{s_1} \oplus \cdots \oplus x_{s_a}
  \]
  in $\Ko \cC (\ulp{m})$ given by composing instances of $c$ and
  $\beta$.  However conditions \eqref{astu} and \eqref{ats}, together
  with 2-naturality of $\beta$ with respect to maps in the Gray tensor
  product, ensure that these are all equal.

  Likewise, given $\{f,\ga\}\cn \{x,c\} \to \{y,d\}$, all 2-cells
  \[
  \def\labelstyle{\displaystyle}
  \begin{xy}
    (0,0)*+{x_s}="A";
    (30,0)*+{\oplus_i x_{s_i}}="B";
    (0,-20)*+{y_s}="C";
    (30,-20)*+{\oplus_i y_{s_i}}="D";
    {\ar^-{!} "A"; "B"};
    {\ar_-{!} "C"; "D"};
    {\ar_-{f_s} "A"; "C"};
    {\ar^-{\oplus_i f_{s_i}} "B"; "D"};
    {\ar@{=>} "B"+(-8,-8); "C"+(8,6)};
  \end{xy}
  \]
  given by pasting instances of $\ga$ and $\beta$ are equal.  Recall
  that $\oplus_i f_{s_i}$ is defined in
  \cref{notn:sum-of-one-cells-gray-monoid}.
\end{rmk}

\begin{defn}\label{defn:stable-equiv-pgm}
  Let $\cC, \cD$ be a pair of permutative Gray-monoids.  A strict
  functor of permutative Gray-monoids $F \cn \cC \to \cD$ is a
  \emph{stable equivalence} if $\Ko F$ is a stable equivalence of
  $\Ga$-2-categories.  We let $(\PGM, \cS)$ denote the relative
  category of permutative Gray-monoids with stable equivalences.
\end{defn}

The next proposition follows from \cite[\S 5.2]{Oso10Spectra} and has
\cref{prop:Ko-preserves-weak-equivalences} as an immediate
consequence.  In \cref{prop:Ko-preserves-weak-equivalences} we use the
fact that every strict functor is a normal oplax one and hence
implicitly restrict $\Ko$ to the subcategory $\PGM$.
\begin{prop}
  \label{prop:ko-is-special}
  The $\Ga$-2-category $\ko \cC$ is special, with $\ko\cC(\ulp{1})$
  isomorphic to $\cC$.
\end{prop}
\begin{prop}
  \label{prop:Ko-preserves-weak-equivalences}
  The functor $\Ko\cn \PGM \to \GaIICat$ preserves weak equivalences
  and is therefore a relative functor
  $\Ko\cn (\PGM,\cW) \to (\GaIICat, \cW)$.
\end{prop}

\begin{rmk}
 In \cite{GO12infinite}, the authors use their coherence theorem to construct a $K$-theory functor for all symmetric monoidal bicategories, by first constructing a pseudo-diagram of bicategories indexed on $\sF$, and then rectifying it using the methods of \cite{CCG11Classifying}. When the input is a permutative Gray-monoid, one can use the same technique as in \cite[\S 2]{GO12infinite} to prove that the two resulting $\Ga$-bicategories are levelwise weakly equivalent. We have chosen to use this explicit construction because of its functoriality.
\end{rmk}

\subsection{Counit for permutative Gray-monoids}
\label{sec:counit}

Let $\cC$ be a permutative Gray-monoid.  We  now use (weak) functoriality of the Grothendieck construction
\cite[\S 3.2]{CCG11Classifying} to give a symmetric monoidal
pseudofunctor
\[
\epzo\cn P\Ko\cC \to \cC.
\]
Recalling \cref{defn:symm-mon-psfun}, this requires that we construct
pseudofunctors
\[
\epzo_{\vec{m}}\cn A\Ko\cC(\vec{m}) \to \cC
\]
for each object $\vec{m}$ in $\sA$, together with pseudonatural
transformations
\[
\begin{xy}
  (0,0)*+{A\Ko\cC(\vec{m})}="A";
  (20,-10)*+{\cC}="C";
  (0,-20)*+{A\Ko\cC(\vec{n})}="B";
  {\ar^-{\epzo_{\vec{m}}} "A"; "C"};
  {\ar_-{\epzo_{\vec{n}}} "B"; "C"};
  {\ar_-{\ol{\phi}_*} "A"; "B"};
  {\ar@{=>}_-{\epzo_{\ol{\phi}}} (9,-8)*{}; (6,-12)*{}};
\end{xy}
\]
for each morphism $\ol{\phi} \cn \vec{m} \to \vec{n}$ in $\sA$. For the general
situation considered in \cite{CCG11Classifying}, there are further
modifications, but we will verify that these are in fact identities.
Applying the Grothendieck construction therefore produces $\epzo$ as a
pseudofunctor between bicategories, the source and target of which
happen to be 2-categories.  

\begin{lem}
  Let $\cC$ be a permutative Gray-monoid.  Evaluation at a subset
  $s \subset \ul{m}$ is a 2-functor
  \[
  \ev_s\cn \Ko \cC(\ulp{m}) \to \cC.
  \]
\end{lem}
\begin{defn}
  \label{defn:epzo-m}
  Let $\cC$ be a permutative Gray-monoid, and let $\vec{m} \in \sA$.
  Then $\epzo_{\vec{m}}$ is defined to be the composite
  \[
  A\Ko \cC(\vec{m}) = \prod_i \Ko\cC(\ul{m_i}_+) \fto{\prod_i \ev_{\ul{m_i}}} \prod_i \cC \fto{\oplus} \cC.
  \]
  Since $\oplus$ is cubical and therefore only a pseudofunctor and
  each $\ev_{\ul{m_i}}$ is a 2-functor, the composite
  $\epzo_{\vec{m}}$ is a pseudofunctor.  For the empty sequence
  $(\,)$, $\epzo_{(\,)}$ is defined as the 2-functor $*\to \cC$ that
  sends the point to the unit object $e$ of the monoidal structure.
\end{defn}

\begin{notn}
  \label{notn:phibar-partition-indexing}
  For each $i$, recall the following notation of
  \cref{eqn:phibar-partition-indexing}:
  \[
  \mathfrak{p}(\ol{\phi},i) = \{j \big| \emptyset \not = \ol{\phi}^{\;-1}(\ul{n_j}) \subset \ul{m_i}\}.
  \]
  The restriction of $\ol{\phi}$ to $\ul{m_i}$ gives a partition of
  $\ul{m_i}$ into
  \[
  \coprod_{j \in \mathfrak{p}(\ol{\phi},i)}
  \ol{\phi}^{\;-1}(\ul{n_j}),
  \]
  where we order $\mathfrak{p}(\ol{\phi},i)$ as a subset of the
  indexing of the tuple $\vec{n}$.
\end{notn}

For an object $\overrightarrow{\{x,c\}} = \{\vec{x},\vec{c}\}$ of
$\Ko\cC(\vec{m})$, let
\begin{equation}
\label{eq:defn-c^i}
c^i_{\ol{\phi}}\cn x^i_{\ul{m_i}} \to
\bigoplus_{j \in \mathfrak{p}({\ol{\phi},i})} x^i_{\ol{\phi}^{\;-1}(\ul{n_j})}
\end{equation}
be the unique 1-cell determined by $c^i$ and the partition
above (see \cref{rmk:unique-cells-given-by-partition}).
Using the symmetry of $\cC$ to reorder, we
obtain each 1-cell $\epzo_{\ol{\phi}}(\{\vec{x},\vec{c}\})$ as the composite:
\begin{equation}
\label{eq:defn-epzo-barphi}
\begin{xy}
  (0,0)*+{\ds 
    \bigoplus_i x^{i}_{\ul{m_i}}}="A";
  (50,0)*+{\ds \bigoplus_i \bigoplus_{j \in \mathfrak{p}({\ol{\phi},i})} x^{i}_{\ol{\phi}^{\;-1}(\ul{n_j})}}="B";
  (100,0)*+{\ds \bigoplus_j
    \left(\epzo_{\vec{n}}(\ol{\phi}_*\{\vec{x},\vec{c}\}) \right)^j
  }="C";
  {\ar^-{\bigoplus_i c^i_{\ol{\phi}}} "A"; "B"};
  {\ar^-{\beta} "B"; "C"};
\end{xy}
\end{equation}
where
\[
\left(\epzo_{\vec{n}}(\ol{\phi}_*\{\vec{x},\vec{c}\}) \right)^j  =
x^{i_j}_{\ol{\phi}^{\;-1}(\ul{n_j})}
\]
for $\ol{\phi}^{\;-1}(\ul{n_j}) \subset \ul{m_{i_j}}$.  Note that this is
well-defined because $i_j$ is uniquely determined if
$\ol{\phi}^{\;-1}(\ul{n_j}) \not = \emptyset$ and $x^i_{\emptyset} = e$
for any $i$.

For a morphism in $A\Ko\cC(\vec{m})$
\[
\overrightarrow{\{f,\ga\}} = \{\vec{f},\vec{\ga}\}\cn \{\vec{x},\vec{c}\} \to \{\vec{y},\vec{d}\}
\]
we have a 2-cell given by the pasting
\begin{equation}
\label{eq:epzo-psnat}
\begin{gathered}
\def\labelstyle{\displaystyle}
\begin{xy}
  (0,30)*+{\epzo_{\vec{m}}(\{\vec{x},\vec{c}\})}="A1";
  (50,30)*+{\ds \bigoplus_i
    \bigoplus_{j \in \mathfrak{p}({\ol{\phi},i})}
    x^i_{\ol{\phi}^{\;-1}(\ul{n_j})}} ="A2";
  (100,30)*+{\epzo_{\vec{n}}(\ol{\phi}_*\{\vec{x},\vec{c}\})}="A3";
  (0,0)*+{\epzo_{\vec{m}}(\{\vec{y},\vec{d}\})}="B1";
  (50,0)*+{\ds \bigoplus_i
    \bigoplus_{j \in \mathfrak{p}({\ol{\phi},i})}
    y^i_{\ol{\phi}^{\;-1}(\ul{n_j})}} ="B2";
  (100,0)*+{\epzo_{\vec{n}}(\ol{\phi}_*\{\vec{y},\vec{d}\})}="B3";
  {\ar^-{\oplus_ic^i_{\ol{\phi}}} "A1";"A2" };
  {\ar^-{\beta} "A2";"A3" };
  {\ar_-{\epzo_{\vec{m}}(\{\vec{f},\vec{\ga}\})} "A1";"B1" };
  {\ar^-{\oplus_i\oplus_{j}f^{i}} "A2";"B2" };
  {\ar_-{\oplus_id^i_{\ol{\phi}}} "B1";"B2" };
  {\ar_-{\beta} "B2";"B3" };
  {\ar^-{\epzo_{\vec{n}}(\ol{\phi}_*\{\vec{f},\vec{\ga}\})} "A3";"B3" };
  {\ar@{=>} "A3"+(-21,-12);"B2"+(21,12) };
  {\ar@{=>} "A2"+(-21,-12);"B1"+(21,12) };
\end{xy}
\end{gathered}
\end{equation}
where the left-hand 2-cell is given as in
\cref{rmk:unique-cells-given-by-partition} and the right-hand 2-cell
is given by pseudonaturality of $\beta$.

\begin{lem}
  \label{lem:epzo-phibar}
  The above data constitute a pseudonatural transformation
  $\epzo_{\ol{\phi}}$.
\end{lem}

\begin{rmk}
  \label{rmk:epz-phibar-is-id-for-map-of-finite-sets}
  When $\vec{m} = (m)$ and $\vec{n} = (n)$ are tuples of length one
  and $\ol{\phi}$ consists of a single map of finite sets
  \[
  \ul{m} \to \ul{n},
  \]
  then $\epzo_{\ol{\phi}}$ is the identity pseudo natural transformation.  
\end{rmk}

\begin{rmk}
  \label{rmk:epz-phibar-is-permutation-when-phibar-is}
  For general $\vec{m}$, $\vec{n}$, if $\ol{\phi}$ is a block
  permutation of the $\ul{m_i}$ but does not partition or permute the
  elements of any $\ul{m_i}$, then the formulas in
  \cref{eq:defn-epzo-barphi,defn:Ko-phi} show that $\epzo_{\ol{\phi}}$
  is given by the component of $\beta$ for the corresponding
  permutation of summands $x^i_{\ul{m_i}}$.  In particular, for the map $\id\cn \vec{m} \to \vec{m}$, the pseudonatural transformation $\epzo _{\id}$
  is the identity.
\end{rmk}

\begin{prop}
  \label{prop:epzo-pseudofunctor}
  There exists a pseudofunctor $\epzo \cn P\Ko\cC \to \cC$ defined by
  \begin{align*}
    \bsb{\vec{m}}{\{\vec{x},\vec{c}\}} & \quad \mapsto \quad
    \epzo_{\vec{m}}(\{\vec{x},\vec{c}\}),\\
    \bsb{\ol{\phi}}{\{\vec{f},\vec{\ga}\}} & \quad \mapsto \quad
    \epzo_{\vec{n}}(\{\vec{f},\vec{\ga}\}) \circ \epzo_{\ol{\phi}},\\
    \bsb{\ol{\phi}}{\vec{\al}} & \quad \mapsto \quad
    \epzo_{\vec{n}}(\vec{\al}) * 1_{\epzo_{\ol{\phi}}}.
  \end{align*}
\end{prop}
\begin{proof}\proofof{prop:epzo-pseudofunctor}
  It is easy to check that $\epzo_{\id} = \id$.  Moreover, for
  composable morphisms in $\sA$,
  \[
  \vec{m} \fto{\ol{\phi}} \vec{n} \fto{\ol{\psi}} \vec{p}
  \]
  we have the following equality by
  \cref{rmk:unique-cells-given-by-partition}.
  \[
  \begin{xy}
    (0,20)*+{A\Ko\cC(\ulp{m})}="A";
    (0,0)*+{A\Ko\cC(\ulp{n})}="B";
    (0,-20)*+{A\Ko\cC(\ulp{p})}="C";
    (30,0)*+{\cC}="D";
    {\ar_{\ol{\phi}_*} "A"; "B"};
    {\ar_{\ol{\psi}_*} "B"; "C"};
    {\ar@/^1pc/^{\epzo_{\vec{m}}} "A"; "D"};
    {\ar^{\epzo_{\vec{n}}} "B"; "D"};
    {\ar@/_1pc/_{\epzo_{\vec{p}}} "C"; "D"};
    {\ar@{=>}_{\epzo_{\ol{\phi}}} (15,10.5)*+{}="tc0"; "tc0"+(-4,-4) };
    {\ar@{=>}_{\epzo_{\ol{\psi}}} (15,-5.5)*+{}="tc1"; "tc1"+(-4,-4) };
    (37,0)*+{=};
    (50,20)*+{A\Ko\cC(\ulp{m})}="A";
    (50,-20)*+{A\Ko\cC(\ulp{p})}="C";
    (80,0)*+{\cC}="D";
    {\ar_{(\ol{\psi}\ol{\phi})_*} "A"; "C"};
    {\ar@/^1pc/^{\epzo_{\vec{m}}} "A"; "D"};
    {\ar@/_1pc/_{\epzo_{\vec{p}}} "C"; "D"};
    {\ar@{=>}_{\epzo_{\ol{\psi}\ol{\phi}}} (65,2)*+{}="tc0"; "tc0"+(-4,-4) };
  \end{xy}
  \]
  By the bicategorical Grothendieck construction of \cite[\S
  3.2]{CCG11Classifying}, $\epzo$ is a pseudofunctor. Since $\epzo_{\id}$ is the identity pseudo natural transformation and $\epzo_{m}$ preserves identity 1-cells, so does $\epzo$.
\end{proof}

\begin{prop}
  \label{thm:epzo-symm-mon}
  Let $\cC$ be a permutative Gray-monoid.  The pseudofunctor $\epzo$
  is the underlying functor of a symmetric monoidal pseudofunctor
  \[
  P\Ko \cC \to \cC.
  \]
\end{prop}
\begin{proof}\proofof{thm:epzo-symm-mon}
  Following \cref{defn:symm-mon-psfun}, we need two transformations,
  one relating the unit in $P\ko{\cC}$ with the unit in $\cC$ and
  another one relating the monoidal product in $P\ko{\cC}$ with the
  one in $\cC$. The unit in the monoidal structure of $P\ko{\cC}$ is
  the pair $(\,(\,),\ast)$. By definition we have
  $\epzo (\,(\,),\ast) = e$ and by
  \cref{rmk:epz-phibar-is-permutation-when-phibar-is} $\epzo$ strictly
  preserves the identity 1-cell of $e$. Thus the unit is preserved
  strictly and we can take the first transformation to be the
  identity.  We now define the second transformation
  \[
  \begin{xy}
    (0,0)*+{P\Ko\cC \times P\Ko\cC}="A";
    (30,0)*+{\cC \times \cC}="B";
    (0,-20)*+{P\Ko\cC}="C";
    (30,-20)*+{\cC}="D";
    {\ar^-{\epzo \times \epzo} "A"; "B"};
    {\ar_{\Box} "A"; "C"};
    {\ar_{\epzo} "C"; "D"};
    {\ar^{\oplus} "B"; "D"};
    {\ar@{=>}_{\chi} (17,-8)*+{}="tc0"; "tc0"+(-4,-4) };
  \end{xy}
  \]
  which can be taken to have identity components on objects, as both
  composites give the same value when evaluated at a pair of objects.
  Thus $\chi$ will be a pseudonatural transformation with identity
  components, or in other words an invertible icon \cite{Lac10Icons}.
  We must now construct components for 1-cells, and check two axioms,
  one for the units and one for composition. These diagrams appear in,
  e.g., \cite{Lac10Icons}.  The component at a pair of 1-morphisms
  \begin{align*}
  [\ol{\phi},\{\vec{f},\vec{\ga}\}] \cn &
  [\vec{m},\{\vec{x},\vec{c}\}] \to [\vec{n},\{\vec{y},\vec{d}\}]\\
  [\ol{\psi},\{\vec{f}',\vec{\ga}'\}] \cn &
  [\vec{m}',\{\vec{x}',\vec{c}'\}] \to [\vec{n}',\{\vec{y}',\vec{d}'\}]
  \end{align*}
  is given by the pasting below.  Recall that
  $\epzo_{\ol{\phi}}(\{\vec{x},\vec{c}\})$ is defined in
  \cref{eq:defn-epzo-barphi} as a composite 1-cell $\beta \circ
  \oplus_i c^i_{\ol{\phi}}$.  For ease of notation we let
  \[
  c_{\ol{\phi}} = \oplus_i c^i_{\ol{\phi}}
  \]
  and let
  \[
  c_{\ol{\phi}}\epzo_{\vec{m}}(\{\vec{x},\vec{c}\}) =
  \bigoplus_i \bigoplus_{j\in \mathfrak{p}({\ol{\phi},i})}
  x^i_{\ol{\phi}^{\;-1}(\ul{n_j})}.
  \]
{
  \renewcommand{\epzo}{\tilde{\varepsilon}}
  \renewcommand{\ol}{\bar}
  \newcommand{\TMPxi}{\epzo_{\vec{m}}\{\vec{x},\vec{c}\}}
  \newcommand{\TMPxii}{c_{\ol{\phi}}\epzo_{\vec{m}}\{\vec{x},\vec{c}\}}
  \newcommand{\TMPxiii}{\epzo_{\vec{n}}\ol{\phi}_*\{\vec{x},\vec{c}\}}
  \newcommand{\TMPXi}{\epzo_{\vec{m}'}\{\vec{x'},\vec{c'}\}}
  \newcommand{\TMPXii}{c_{\ol{\psi}}\epzo_p\{\vec{x'},\vec{c'}\}}
  \newcommand{\TMPXiii}{\epzo_{\vec{n}'}\ol{\psi}_*\{\vec{x'},\vec{c'}\}}
  \newcommand{\TMPy}{\epzo_{\vec{n}}\{\vec{y},\vec{d}\}}
  \newcommand{\TMPY}{\epzo_{\vec{n}'}\{\vec{y'},\vec{d'}\}}
  \newcommand{\TMPcx}{c_{\ol{\phi}}}
  \newcommand{\TMPbx}{\beta}
  \newcommand{\TMPf}{\epzo_{\vec{n}}(f)}
  \newcommand{\TMPcX}{c_{\ol{\psi}}}
  \newcommand{\TMPbX}{\beta}
  \newcommand{\TMPF}{\epzo_{\vec{n}'}(f')}
  \[
  \renewcommand{\objectstyle}{\scriptstyle}
  \renewcommand{\labelstyle}{\textstyle}
  \begin{xy}
    (0,0)*+{\TMPxi \oplus \TMPXi}="0,0";
    "0,0"+(19,20)*+{\TMPxii \oplus \TMPXi}="1,1";
    "1,1"+(19,20)*+{\TMPxiii \oplus \TMPXi}="2,2"; 
    "2,2"+(19,20)*+{\TMPy \oplus \TMPXi}="3,3";
    "1,1"+(19,-20)*+{\TMPxii \oplus \TMPXii}="2,0"; 
    "2,2"+(19,-20)*+{\TMPxiii \oplus \TMPXii}="3,1";
    "3,3"+(19,-20)*+{\TMPy \oplus \TMPXii}="4,2"; 
    "3,1"+(19,-20)*+{\TMPxiii \oplus \TMPXiii}="4,0"; 
    "4,2"+(19,-20)*+{\TMPy \oplus \TMPXiii}="5,1";
    "5,1"+(19,-20)*+{\TMPy \oplus \TMPY}="6,0";
    {\ar^-{\TMPcx \oplus \id} "0,0"; "1,1"};
    {\ar^-{\TMPbx \oplus \id} "1,1"; "2,2"};
    {\ar^-{\TMPf \oplus \id} "2,2"; "3,3"};
    {\ar^-{\id \oplus \TMPcX} "3,3"; "4,2"};
    {\ar^-{\id \oplus \TMPbX} "4,2"; "5,1"};
    {\ar^-{\id \oplus \TMPF} "5,1"; "6,0"};
    {\ar|-{\id \oplus \TMPcX} "1,1"; "2,0"};
    {\ar|-{\id \oplus \TMPcX} "2,2"; "3,1"};
    {\ar|-{\TMPf \oplus \id} "3,1"; "4,2"};
    {\ar|-{\TMPf \oplus \id} "4,0"; "5,1"};
    {\ar|-{\TMPbx \oplus \id} "2,0"; "3,1"};
    {\ar|-{\id \oplus \TMPbX} "3,1"; "4,0"};
    {\ar@/_2pc/_-{\TMPcx \oplus \TMPcX} "0,0"; "2,0"};
    {\ar@/_2pc/_-{\TMPf \oplus \TMPF} "4,0"; "6,0"};
    {\ar@{=>}^-{\Sigma^{-1}} "3,3"+(0,-15)="a"; "a"+(0,-4)};
  \end{xy}
  \]
}
If either $f$ or $\psi$ is the identity, then the single instance of
$\Si$ above is also the identity, which immediately implies the unit
axiom for $\chi$.  The composition axiom involves a much larger
diagram that we omit, noting that it follows from the definitions of
the cells involved and the Gray tensor product axioms.

There are then three invertible modifications to construct in order to
show that $\epzo$ has a monoidal structure (see, e.g.,
\cite[App. A]{McCru00Balanced}).  Two concern the unit object, and one
easily checks that these can be chosen to be the identity, and the
third concerns using the monoidal product on a triple of objects.
This last one, usually denoted $\omega$,
can also be taken to be the identity using the Gray tensor product
axioms.  Then there are two axioms to check, but each consists of only
identity 2-cells, thus they both trivially commute.

There is one further modification, $U$, needed for a symmetric
monoidal structure on $\epzo$.  One can check that $\epzo$ sends the
1-cell $\beta$ in $P\ko{\cC}$ to the 1-cell $\beta$ in $\cC$, and then
the 2-naturality of $\beta$ (see \cref{defn:pgm}) can be used to show
that we can take $U$ to be the identity as well. 
Below we will check the final three symmetric monoidal pseudofunctor
axioms which all concern the interaction between $U$, the monoidal
structure on $\epzo$, and the modifications $R_{-|--}, R_{--|-}, v$ in
the source and target.  Diagrams for these are drawn in
\cite[App. A,B]{McCru00Balanced}.

For the first axiom, one pasting diagram consists of two instances of
$\omega$ which are both the identity, a naturality 2-cell for $\be$
which is the identity since one of the components is an identity
1-cell, one instance of $U$ which is the identity, and one final
2-cell.  This 2-cell is obtained by pasting together instances of the
pseudofunctoriality isomorphisms for $\epzo$.  Recalling the
expression for the braiding in
\cref{eqn:symmetry-for-Grothendieck-constr}, these 2-cells are all
special cases of the pseudofunctoriality isomorphisms in which the
1-cell has the form $\bsb{\ol{\sigma}}{\{\id,\id\}}$ for a block
permutation $\sigma$.  Combining
\cref{rmk:epz-phibar-is-permutation-when-phibar-is} with the formulas
for $\epzo$ in \cref{prop:epzo-pseudofunctor}, these are all identity
2-cells.  A similar argument, with the additional observation that
$\epzo$ strictly preserves identity 1-cells (see
\cref{eq:epzo-psnat}), shows that the other pasting diagram for
this axiom also consists of only identity 2-cells.  Thus the axiom
reduces to the statement that the identity 2-cell is equal to itself.

The second axiom, relating $R_{--|-}$ to $U$, is analogous.  The third
axiom relates $U$ to the syllepsis $v\cn\beta^{2} \cong 1$, and
follows by the same line of reasoning, only this time using that the
pseudofunctoriality isomorphism
\[
\epzo(\beta) \circ \epzo(\beta) \cong \epzo(\beta \circ \beta)
\]
is the identity.
\end{proof}
Combining the formulas of
\cref{prop:epzo-pseudofunctor,prop:functoriality-of-K,rmk:description-Ph},
one sees immediately that $\epzo$ is natural with respect to strict functors
of permutative Gray-monoids.
\begin{prop}
  \label{prop:epzo-natural}
  The functor $\epzo = \epzo_{\cC} \cn P \Ko \cC \to \cC$ is natural
  in $\cC$ with respect to strict functors of permutative Gray-monoids.
\end{prop}
\begin{rmk}

  For general pseudofunctors $(F,\theta)$, $\epzo$ is not strictly
  natural.  We do expect that $\epzo$ satisfies a weak naturality,
  induced by the pseudofunctoriality transformation $\theta$, but have
  not pursued those details.
\end{rmk}

We now describe the homotopy-theoretic properties of $\epzo$.  We use
these to show in \cref{thm:main-css-2} that the homotopy theory of
permutative Gray-monoids is equivalent to that of permutative
2-categories.

\begin{prop}
  \label{prop:epzo-weak-equiv}
  For a permutative Gray-monoid $\cC$, $\epzo\cn P\Ko\cC \to \cC$ is a
  weak equivalence of 2-categories.
\end{prop}
\begin{proof}\proofof{prop:epzo-weak-equiv}
  By \cref{prop:ko-is-special}, $\Ko\cC$ is a special $\Ga$-2-category
  and $\Ko\cC(\ulp{1}) \iso \cC$.  Hence by
  \cref{lem:X1-to-PX-weak-equiv} the first map in the composite below
  is a weak equivalence
  \[
  \cC \iso \Ko \cC(\ulp{1}) \to P\Ko\cC \fto{\epzo} \cC.
  \]
  Since the composite is 
  identity on $\cC$, the result follows.
\end{proof}

\begin{rmk}
\cref{prop:epzo-weak-equiv} demonstrates a genuinely new phenomenon for symmetric monoidal structures appearing at
the 2-dimensional level.  Namely, we have a strict notion of symmetric
monoidal structure which models all weak homotopy types of a more
general symmetric monoidal structure but does not model all
categorical equivalence types.

Indeed, there exist permutative Gray-monoids which are not equivalent
(in the categorical sense) to any permutative 2-category.  The
fundamental reason for this is that Gray-monoids model (unstable)
connected 3-types and, as is well-known, strict monoidal 2-categories
cannot.  The 2-cells $\Sigma$ are necessary to model the generally
nontrivial Whitehead product
\[
\pi_2 \times \pi_2 \to \pi_3.
\]
Simpson \cite[\S 2.7]{Simpson} shows, for example, that the 3-type of
$S^2$ cannot be modeled by any strict 3-groupoid and therefore
certainly not by any strict monoidal 2-groupoid.

The example of \cite[Example 2.30]{Sch2011Classification} is a permutative
Gray-monoid for which, for the same reason, cannot be equivalent to any
strict monoidal 2-category.  In particular, it cannot be equivalent to
a permutative 2-category.
\end{rmk}


We would like to use \cref{prop:epzo-weak-equiv} to show that every
permutative Gray-monoid is weakly equivalent, in the category of
permutative Gray-monoids and strict maps, to a permutative 2-category.
However $\epzo$ does not achieve this directly since it is not a
strict symmetric monoidal map.  We do achieve a zigzag of weak
equivalences in \cref{prop:epzo-strictification}, however, by applying
our variant of the strictification of \cite{Sch2011Classification} given
in \cref{thm:cohqs2cats2}.

\begin{prop}
  \label{prop:epzo-strictification}
  Applying the construction in \cref{thm:cohqs2cats2} to $\epzo$
  yields a natural zigzag of strict symmetric monoidal weak
  equivalences between $P\Ko\cC$ and $\cC$.  Hence these are naturally
  isomorphic in $\ho\PGM$.
\end{prop}
\begin{proof}\proofof{prop:epzo-strictification}
  The second part of \cref{thm:cohqs2cats2}, when applied to $\epzo$,
  yields the square
  \[
  \xy
  (0,0)*+{P\Ko\cC^{qst}}="00";
  (40,0)*+{P\Ko\cC}="10";
  (0,-12)*+{\cC^{qst}}="01";
  (40,-12)*+{\cC}="11";
  {\ar^{\nu} "00"; "10" };
  {\ar^{\epzo} "10"; "11" };
  {\ar_{\epzo^{qst}} "00"; "01" };
  {\ar_{\nu} "01"; "11" };
  (20,-6)*{\simeq}
  \endxy
  \]
  which only commutes up to a symmetric monoidal equivalence as
  indicated.  All of the objects in this square are permutative
  Gray-monoids, and the only map which is not strict is $\epzo$.  Both
  instances of $\nu$ are symmetric monoidal biequivalences, and the
  symmetric monoidal equivalence 2-cell filling the square yields a
  homotopy upon taking nerves, so the 2-out-of-3 property for weak
  equivalences shows that $\epzo^{qst}$ is a weak equivalence as
  $\epzo$ is by \cref{prop:epzo-weak-equiv}.  Thus our zigzag of
  strict symmetric monoidal weak equivalences is
  \[
  P\Ko\cC \stackrel{\nu}{\longleftarrow} P\Ko\cC^{qst} \stackrel{\epzo^{qst}}{\longrightarrow} \cC^{qst} \stackrel{\nu}{\longrightarrow} \cC.
  \]
  Given a strict map of permutative Gray-monoids $F\cn\cB\to\cC$, the
  zigzag above is natural in $F$ since both $\nu$ and $\epzo$ are.
\end{proof}

\subsection{\texorpdfstring{$K$}{K}-theory for permutative 2-categories}
\label{sec:K-epz-strict}

Recall that $\PIICatnop$ denotes the category of permutative
2-categories and normal oplax functors between them.  Using the
additional rigidity of permutative 2-categories we give a second,
somewhat simpler, $K$-theory functor
\[
\Kt \cn \PIICatnop \to \GaIICat.
\]
It is this version of $K$-theory that we will use to construct our equivalences of homotopy theories in \cref{sec:eta-and-triangle-ids}.

Throughout this section, let $(\cC,\oplus, e,\beta)$ be a permutative
2-category.
\begin{construction}
  \label{defn:Kt}
  For a finite pointed set $\ulp{n} \in \sF$, $\kt\cC(\ulp{n})$ is a
  2-category defined as follows.

  \begin{enumerate}
  \item The objects of $\kt\cC(\ulp{n})$ are
    the same as those of $\ko\cC(\ulp{n})$.

  \item A 1-morphism from $\obj{x}{c}$ to $\obj{x'}{c'}$ is a system
    of 1-morphisms $\{ f_s\}$, with $f_s\colon x_s \to x'_s$, such
    that $f_{\emptyset}=\id _e$ and for all $s$ and $t$, the diagram
    \[
    \def\labelstyle{\displaystyle}
    \xymatrix{
      x_{s\cup t} \ar[r]^-{c _{s,t}} \ar[d]_{f_{s\cup t}}
      & x_s\opl x_t \ar[d]^{f_s\opl f_t} \\
      x'_{s\cup t} \ar[r]_-{c'_{s,t}}
      & x'_s\opl x'_t.
    }
    \]
    commutes.

  \item A 2-morphism
    \[
    \def\labelstyle{\displaystyle}
    \begin{xy}
      (0,0)*+{\{x_s,c_s\}}="A";
      (40,0)*+{\{x'_s,c'_s\}}="B";
      {\ar@/^1pc/^{\{f_s\}} "A"; "B"};
      {\ar@/_1pc/_{\{g_s\}} "A"; "B"};
      {\ar@{=>}^{\{\alpha_s\}} (17,3)*+{}; (17,-3)*+{} };
    \end{xy}
    \]
    is a collection of 2-morphisms $\al_s \colon f_s \to g_s$ such that
    $\al_{\emptyset}=\id_{\id_e}$ and for all $s$ and $t$, the
    equality
    \[
    \def\labelstyle{\displaystyle}
    \xy
    (0,0)*+{\xst}="00";
    (30,0)*+{\xs\opl\xt}="20";
    (0,-24)*+{\xst'}="02";
    (30,-24)*+{\xs'\opl\xt'}="22";
    {\ar^{\cst} "00";"20"};
    {\ar@/^1.6pc/^{\fst} "00";"02"};
    {\ar@/_1.6pc/_{\gst} "00";"02"};
    {\ar^{\fs\opl\ft} "20";"22"};
    {\ar_{\cst'} "02";"22"};
    {\ar@{=>}^{\alst} (2,-12)*{};(-2,-12)*{}};
    (45,-12)*{=};
    (60,0)*+{\xst}="00";
    (90,0)*+{\xs\opl\xt}="20";
    (60,-24)*+{\xst'}="02";
    (90,-24)*+{\xs'\opl\xt'}="22";
    {\ar^{\cst} "00";"20"};
    {\ar_{\gst} "00";"02"};
    {\ar@/^1.7pc/^{\fs\opl\ft} "20";"22"};
    {\ar@/_1.7pc/_{\gs\opl\gt} "20";"22"};
    {\ar_{\cst'} "02";"22"};
    {\ar@{=>}^{\als \opl \alt} (92,-12)*{};(88,-12)*{}};
    \endxy
    \]
    holds.
  \end{enumerate}

  Composition of 1-morphisms and 2-morphisms is done componentwise. We
  emphasize that this definition is possible only because $\oplus
  \colon \cC \times \cC \to \cC$ is a 2-functor. It is easy to check that
  $\kt\cC(\ulp{n})$ is a 2-category. Given $\phi \colon \ulp{m} \to
  \ulp{n}$ in $\sF$, the 2-functor $\phi_{\ast} \colon
  \kt\cC(\ulp{m}) \to \kt\cC(\ulp{n})$ is defined in an analogous way
  to the one for $\ko$ (\cref{defn:Ko-phi}).
\end{construction}

\begin{rmk}
  This construction is a $\Cat$-enrichment of the standard
  $K$-theory for permutative categories initiated by Segal
  \cite{Seg74Categories}.  See, e.g., \cite{Man10Inverse}.
\end{rmk}

\begin{defn}\label{defn:stable-equiv-piicats}
  Let $\cC, \cD$ be a pair of permutative 2-categories.  A strict
  functor of permutative 2-categories $F \cn \cC \to \cD$ is a
  \emph{stable equivalence} if $\kt F$ is a stable equivalence of
  $\Ga$-2-categories.  We let $(\PIICat, \cS)$ denote the relative
  category of permutative 2-categories and stable equivalences.
\end{defn}

We have the analogue of \cref{prop:ko-is-special} for $\Kt$.

\begin{prop}
  \label{prop:kt-is-special}
  The $\Ga$-2-category $\kt \cC$ is special, with $\kt\cC(\ulp{1})$
  isomorphic to $\cC$.
\end{prop}

Since $\sC$ is a permutative 2-category it is also a permutative
Gray-monoid, so we also have the weaker construction $\Ko\cC$ of
\cref{sec:K-epz-fluffy}.  As the construction $\Kt\cC$ is a
specialization of the former construction, we have an inclusion
\[\label{Kt-to-Ko}
\Kt\cC \to \Ko\cC
\]
given by
\label{Kt-to-Ko-formulas}
\begin{align*}
  \{x_s,c_{s,t}\} & \mapsto \{x_{s},c_{s,t}\},\\
  \{f_s\} & \mapsto \{f_s,\id\},\\
  \{\al_s\} & \mapsto \{\al_s,\id\}.
\end{align*}

\begin{prop}
  \label{prop:K-Ko-levelwise-weak-equiv}
  For a permutative 2-category $\cC$, the inclusion
  \[
  \Kt\cC \to \Ko\cC
  \]
  is a strict map of $\Ga$-2-categories and a levelwise weak
  equivalence.  The inclusion is natural with respect to strict
  functors of permutative 2-categories, i.e., strict symmetric
  monoidal 2-functors.
\end{prop}
\begin{proof}\proofof{prop:K-Ko-levelwise-weak-equiv}
  The first statement is obvious, and therefore we get a commutative
  square
  \[
  \begin{xy}
    (0,0)*+{\Kt\cC(\ulp{n})}="A";
    (30,0)*+{\ko\cC(\ulp{n})}="B";
    (0,-12)*+{\cC^{n}}="C";
    (30,-12)*+{\cC^{n}}="D";
    {\ar "A"; "B"};
    {\ar "B"; "D"};
    {\ar "A"; "C"};
    {\ar "C"; "D"};
   \end{xy}
  \]
  using the projection maps, \cref{prop:kt-is-special}, and the fact
  that the map $\Kt\cC(\ulp{1}) \to \Ko\cC(\ulp{1})$ is the identity.
  Both vertical maps and the bottom map are weak equivalences, so the
  top is as well, proving the second statement.  Naturality with
  respect to strict functors of permutative 2-categories is then just
  a straightforward check using the definitions.
\end{proof}

\begin{rmk}
  \label{rmk:K-functorial}
  If $F\cn \cC \to \cD$ is a normal oplax map of permutative
  2-categories, then the map of $\Ga$-2-categories $\Ko F$ constructed in
  \cref{prop:functoriality-of-K} in fact defines a map
  \[
  \Kt F\cn \Kt\cC \to \Kt\cD.
  \]
\end{rmk}

We have the following as an immediate consequence of
\cref{prop:K-Ko-levelwise-weak-equiv,prop:Ko-preserves-weak-equivalences},
together with the fact that stable equivalences satisfy the 2 out of 3
property.
\begin{prop}\label{prop:inc_creates_stable}
The inclusion $\PIICat \hookrightarrow \PGM$ preserves and reflects stable equivalences, so defines a relative functor $(\PIICat, \cS) \to (\PGM, \cS)$.
\end{prop}

We also have a further result about $K$.

\begin{prop}
  \label{prop:K-preserves-weak-equivalences}
  The functor $K \cn \PIICat \to \GaIICat$ sends weak equivalences of
  permutative 2-categories to levelwise weak equivalences of
  $\Ga$-2-categories, and therefore is a relative functor $(\PIICat, \cW) \to (\GaIICat, \cW)$.
\end{prop}

\begin{rmk}
Note that the version of the previous proposition in which $\cW$ is replaced by $\cS$ is in fact merely the definition of stable equivalences in $\PIICat$.
\end{rmk}

The additional strictness of $K\cC$ for a permutative 2-category $\cC$
provides a stricter version of the map $\epzo$ constructed above.
\begin{defn}
\label{defn:epz}
For a permutative 2-category $\cC$, let $\epz: PK\cC \to \cC$ be the
composite
\[
PK\cC \to P\Ko\cC \fto{\epzo} \cC.
\]
For $\ol{\phi}\cn \vec{m} \to \vec{n}$ in $\sA$, we let
$\epz_{\vec{m}}$ and $\epz_{\ol{\phi}}$ denote the corresponding
composites with $AK\cC \to A\Ko\cC$.
\end{defn}


Although $\epzo$ is only a weak symmetric monoidal functor of
permutative Gray-monoids, the composite $\epz$ enjoys a number of
stricter properties.
\begin{prop}
  \label{prop:epz-times-is-strict-SM}
  Let $\cC$ be a permutative 2-category and let $\ol{\phi}\cn \vec{m}
  \to \vec{n}$ in $\sA$. Then
  \begin{enumerate}
  \item $\epz_{\vec{m}}$ is a 2-functor.
  \item $\epz_{\ol{\phi}}$ is 2-natural.
  \item $\epz$ is a strict symmetric monoidal functor.
  \end{enumerate}
\end{prop}
\begin{proof}\proofof{prop:epz-times-is-strict-SM}
  The first claim holds because the composite in \cref{defn:epzo-m} is
  2-functorial when $\oplus$ is a 2-functor.  The second holds because
  the 2-cells in \cref{eq:epzo-psnat} are identities when $\cC$ is a
  permutative 2-category.  The third holds upon noting that $\chi$ in the
  proof of \cref{thm:epzo-symm-mon} is the identity
  2-cell when $\cC$ is a permutative 2-category.
\end{proof}

Taking a $\Cat$-enrichment of \cite[Thm 4.7]{Man10Inverse} or by
specializing \cref{prop:epzo-natural} we have the following statement
about naturality of $\epz$.
\begin{prop}
  \label{prop:epz-is-natural}
  Let $\cC$ be a permutative 2-category.  Then
  \[
  \epz\cn PK\cC \to \cC
  \]
  is natural with respect to strict maps of permutative 2-categories,
  and lax natural with respect to lax maps.
\end{prop}

\begin{prop}
  \label{prop:epz-weak-equiv}
  For a permutative 2-category $\cC$, $\epz\cn PK\cC \to \cC$ is a
  weak equivalence.
\end{prop}
\begin{proof}\proofof{prop:epz-weak-equiv}
  The map $\epz$ is the composite of two maps by definition.  The first is a weak equivalence since it is $P$ applied to the levelwise equivalence of \cref{prop:K-Ko-levelwise-weak-equiv}.  The second is a weak equivalence by \cref{prop:epzo-weak-equiv}.
\end{proof}

We are finally able to prove our first important equivalence of homotopy theories.

\begin{thm}\label{thm:main-css-2}
The inclusion $\PIICat \hookrightarrow \PGM$ induces an equivalence of
homotopy theories
\begin{align*}
(\PIICat, \cW) & \stackrel{\simeq}{\longrightarrow} (\PGM, \cW).
\end{align*}
\end{thm}
\begin{proof}\proofof{thm:main-css-2}
First note that the inclusion obviously preserves weak equivalences. 
In the other direction, we have the functor $P\Ko \cn \PGM \to \PIICat$.  The functor $\Ko$ preserves weak equivalences by \cref{prop:Ko-preserves-weak-equivalences}, and the functor $P$ preserves weak equivalences by \cref{prop:P-preserves-weak-equivalences}, so $P\Ko$ defines a relative functor.  One natural zigzag of weak equivalences is given in \cref{prop:epzo-strictification} and the other is
\[
P\Ko \cC \ot PK\cC \to \cC.
\]
Thus \cref{cor:rel-cats-sufficient} shows that we have an equivalence
of homotopy theories.
\end{proof}

\begin{cor}
  \label{cor:equiv-HoPGM-and-HoP2Cat}
  The functors 
  \begin{align*}
\PIICat & \hookrightarrow \PGM, \\
\PGM & \stackrel{P\Ko}{\longrightarrow} \PIICat
\end{align*}
 establish an equivalence of
  homotopy categories
  \[
  \ho\PGM \hty \ho\PIICat.
  \]
 \end{cor}


\section{Equivalences of homotopy theories}
\label{sec:eta-and-triangle-ids}

This section establishes our main result, that there are equivalences
of homotopy theories between permutative Gray-monoids, permutative
2-categories, and $\Ga$-2-categories when each is equipped with the
relative category structure given by stable equivalences.  We begin by
constructing a lax unit $\eta$ to complement the strict counit $\epz$
of \cref{defn:epz}.  These provide a lax adjunction in
the following sense. In
\cref{sec:eta-is-lax-transformation,sec:eta-Ga-lax} we show that the
unit has components which are $\Ga$-lax maps
\[
X \to KPX
\]
and satisfies an oplax naturality condition with respect to $\Ga$-lax maps.

The key technical result is that $\eta$ furnishes us with a natural
zigzag of stable equivalences between the identity and $KP$.  In
\cref{sec:K-epz-eta,sec:K-epz-K-P-eta} we prove that the composites
$K\epz \circ \eta K$ and $\epz P \circ P\eta$ are identities on the
objects of $\GaIICat$ and $\PIICat$, respectively.  These triangle
identities, together with some simple naturality statements, show that
$\eta$ is a stable equivalence (\cref{thm:eta-stable-equiv}).  The
final step is then to convert this lax stable equivalence into a
zigzag (in this case, a span) of strict stable equivalences using the
methods of \cref{sec:E-construction}.

\subsection{Construction of the unit}
\label{sec:construction-of-eta}

Throughout this section let $X$ be a $\Ga$-2-category.  We define a $\Ga$-lax map
\[
\eta\cn X \to KPX.
\]

\begin{notn}
  \label{notn:pis}
  For $s \subset \ul{m}$, define
  \[
  \pi^s\cn \ul{m}_+ \to \ul{|s|}_+ , \quad s = \{i_1, \ldots, i_{|s|}\}
  \]
  \[
  \pi^s(i) =
  \begin{cases}
    0, & \text{ if } i \not\in s\\
    k, & \text{ if } i = i_k \in s.
  \end{cases}
  \]
\end{notn}
\begin{note}
  The map $\pi^s$ is the unique surjective, order-preserving map
  $\ul{m}_+ \to \ul{|s|}_+$ which maps the complement of $s$ to the
  basepoint.
\end{note}
\begin{notn}
  \label{notn:pist}
  For $s, t \subset \ul{m}$ with $s \cap t = \emptyset$, define a map
  in $\sA$
  \[
  \pi^{s,t}\cn |s \cup t| \to (|s|, |t|)
  \]
  that corresponds to the partition of the ordered set $s \cup t$ into the disjoint union of $s$ and $t$.
\end{notn}
\begin{rmk}
  \label{rmk:pist-pis-pit}
  For $x \in X(\ul{m}_+)$, we have $\pi^s_*x \in AX(|s|) = X(\underline{|s|}_+)$
  and
  \[
  \pi^{s,t}_* \pi^{s\cup t}_* x = 
  \pi^s_* x \times \pi^t_* x
  \]
  in $AX(|s|,|t|) = X(\ul{|s|}_+) \times X(\ul{|t|}_+)$.
\end{rmk}

Now for a finite set $\ulp{m}$ we define $\eta = \eta_{X,\ulp{m}}$
using the notation of \cref{rmk:objs-in-K-thy-as-functions}.
\begin{defn}
  \label{defn:eta}
  For a 0-cell $x \in X(\ulp{m})$,  
  \begin{equation}
    \label{eq:defn-eta}
    \eta (x) = 
    \iicb{1}{s}{\bsb{|s|}{\pi^s_* x}}{s,t}{\bsb{\pi^{s,t}}{\id_{\pi^s_* x \times \pi^t_* x}}}
  \end{equation}
  By \cref{rmk:pist-pis-pit}, this gives a well-defined 0-cell of
  $KPX(\ul{m}_+)$.\\
  For a 1-cell $x \fto{f} y$, 
  \[
  \eta(f) = \left\{s \mapsto \bsb{\id_{|s|}}{\pi^s_* f}\right\}.
  \]
  For a 2-cell $\al\cn f \to g$, 
  \[
  \eta(\al) = \left\{s \mapsto \bsb{\id_{|s|}}{\pi^s_* \al}\right\}.
  \]

  

  One easily verifies that $\eta(f)$ and $\eta(\al)$ satisfy the
  structure diagrams of 1- and 2-cells of $KPX(\ulp{m})$ by checking
  each $s, t \subset \ul{m}$ with $s \cap t = \emptyset$. It is also easy to check that $\eta_{X,\ul{m}_+}$ is a 2-functor.
\end{defn}

\subsection{Laxity of the unit components}
\label{sec:eta-is-lax-transformation}

For $\phi\cn \ul{m}_+ \to \ul{n}_+$ and $s \subset \ul{n}$ we have
\begin{align*}
(\eta (\phi_* x))_s & = [|s|, \pi^s_* \phi_* x] \in PX  \\
(\phi_* \eta (x))_s & = [|\phi^{-1} (s)|, \pi^{\phi^{-1}(s)}_* x] \in PX.
\end{align*}
These are not equal, so $\eta = \eta_X$ is not a strict map of
$\Ga$-2-categories.  Instead, $\eta_X$ naturally has the structure of
a $\Ga$-lax map (\cref{defn:D-lax-map}), as we now describe.

\begin{notn}
  \label{notn:phis}
  For $\phi \cn \ulp{m} \to \ulp{n}$ in $\sF$ and $s \subset \ul{n}$,
  the map $\phi^s$ is defined by reindexing $\phi\big|_{\phi^{-1}(s)}$
  so that the following diagram of pointed sets commutes.
  \[
  \begin{xy}
    (0,20)*+{\ul{m}_+}="A1";
    (20,20)*+{\ul{n}_+}="A2";
    (0,0)*+{\ul{|\phi^{-1}(s)|}_+}="B1";
    (20,0)*+{\ul{|s|}_+}="B2";
    {\ar^-{\phi} "A1"; "A2"};
    {\ar_-{\phi^s_+} "B1"; "B2"};
    {\ar_-{\pi^{\phi^{-1}(s)}} "A1"; "B1"};
    {\ar^-{\pi^s} "A2"; "B2"};
  \end{xy}
  \]
\end{notn}

\begin{defn}
  \label{defn:eta-phi}
  Let $X$ be a $\Ga$-2-category.  For $\phi\cn \ulp{m} \to \ulp{n}$ in
  $\sF$ and $x \in X(\ulp{m})$, let
  \[
  \eta_{\phi}(x) \cn \phi_* \eta (x) \to \eta (\phi_* x)
  \]
  be given by
  \[
  \eta_{\phi}(x) = \left\{s \mapsto \bsb{\phi^s}{\id_{\pi^s \phi_* x}}\right\}.
  \]
  Each $\bsb{\phi^s}{\id_{\pi^s \phi_* x}}$ is a well-defined morphism
  in $PX$ since
  \begin{eqn}
    \label{eqn:pi-phi-eq-phi-pi}
    \phi^s_* \pi^{\phi^{-1}(s)}_* x = \pi^s_* \phi_* x. 
  \end{eqn}
\end{defn}

These 1-cells are natural in $x$ and hence we have the following.
\begin{prop}
  \label{prop:eta-Ga-lax}
  Let $X$ be a $\Ga$-2-category.  The 1-cells $\eta_\phi =
  \eta_{X,\phi}$ give $\eta_X$ the structure of a $\Ga$-lax map $X \to
  KPX$.
\end{prop}

\subsection{Naturality of the unit}
\label{sec:eta-Ga-lax}

In this section we discuss the oplax naturality of the $\Ga$-lax maps
$\eta_X$ for varying $X$.  This is sufficient to imply that $\eta$
induces a natural transformation in the homotopy category of
$\Ga$-2-categories and $\Ga$-lax maps
(\cref{prop:eta-natural-in-hoGa2Cat-lax}) and we use this to show
$\eta$ is a stable equivalence (\cref{thm:eta-stable-equiv}).  We
observe, furthermore, that $\eta$ is strictly natural with respect to
strict $\Ga$-maps (\cref{cor:eta-nat-strict}) and this is essential to
the proof of \cref{thm:main-css-1}.

Let $h\cn X \to Y$ be a $\Gamma$-lax map.  The diagram
\[
\begin{xy}
  (0,20)*+{X}="A1";
  (20,20)*+{KPX}="A2";
  (0,0)*+{Y}="B1";
  (20,0)*+{KPY}="B2";
  {\ar^-{\eta_X} "A1"; "A2"};
  {\ar_-{\eta_Y} "B1"; "B2"};
  {\ar_-{h} "A1"; "B1"};
  {\ar^-{KPh} "A2"; "B2"};
\end{xy}
\]
does not generally commute.  We define a $\Ga$-transformation 
\[
\la \cn \eta h \rightarrow KPh\, \eta
\]
as follows.  For readability, we omit subscripts $X$ and $Y$ on
$\eta$.  For each $\ulp{m} \in \sF$ and each $x \in X(\ulp{m})$, we
combine \cref{defn:eta,rmk:description-Ph,rmk:K-functorial} to compute
\begin{eqn}\label{eqn:eta-h}
  (\eta h)_{\ulp{m}}(x) = 
  \iicb{1}{(s \subset \ul{m})}{\bsb{|s|}{\pi^s_* h(x)}}{\stdisj{\ul{m}}}{\bsb{\pi^{s,t}}{\id}}
\end{eqn}
and
\begin{eqn}\label{eqn:KPh-eta}
  (KPh\, \eta)_{\ulp{m}}(x) = 
  \iicb{1.3}{(s \subset \ul{m})}{\bsb{|s|}{h(\pi^s_* x)}}{\stdisj{\ul{m}}}{\bsb{\pi^{s,t}}{h_{\pi^{s,t}}}}.
\end{eqn}
\begin{defn}\label{defn:lambda-components}
  For each $\ulp{m} \in \sF$ and each $x \in X(\ulp{m})$, we define 
  component 1-cells 
  \[
  \la_{\ulp{m}}(x)\cn (\eta h)_{\ul{m}_+}(x) \to (KPh\, \eta)_{\ul{m}_+}(x)
  \]
  in $KPY(\ulp{m})$.  These are given by
  \[
  \la_{\ulp{m}}(x) = \left\{s \mapsto \bsb{\id_{|s|}}{h_{\pi^s}(x)}\right\}.
  \]
  Checking that this is a legitimate 1-cell reduces to verifying
  $\pi^s_* h_{\pi^{s \cup t}} = h_{\pi^s \pi^{s \cup t}} = h_{\pi^s}$.
  This follows because $h$ is a $\Ga$-lax map.
\end{defn}
In \cref{lem:lambda-2-natural} we prove that these are the components
of a 2-natural transformation
\[
\la_{\ulp{m}}\cn \eta h_{\ul{m}_+} \to KPh\, \eta_{\ul{m}_+}
\]
for each $\ulp{m} \in \sF$. In \cref{lem:lambda-cube-commutes} we
verify the single axiom for a $\Ga$-transformation (see
\cref{defn:D-trans}).  Together these prove the following.
\begin{prop}
  \label{prop:eta-lax-natural}
  For a $\Ga$-lax map $h\cn X \to Y$, $\la$ gives a
  $\Ga$-transformation
  \[
  \la \cn \eta h \rightarrow KPh\, \eta.
  \]  
\end{prop}

\begin{cor}\label{cor:eta-nat-strict}
For a strict $\Ga$-map $h\cn X \to Y$, the $\Ga$-transformation $\la$ is the identity.  Thus $\eta$ is strictly natural on strict maps, i.e., $\eta h = KPh \eta$ if $h$ is a strict $\Ga$-map.
\end{cor}

\begin{rmk}\label{rmk:eta-lax-nat-maybe}
  We do expect that $\eta$ is a lax transformation when $\GaIICat_{l}$
  is regarded as a 2-category (see \cref{rmk:D2Cat-lax-is-a-category})
  and $P$, $K$ are extended to this structure.  However we have not
  needed to pursue those details.
\end{rmk}

\begin{lem}
  \label{lem:lambda-2-natural}
  Each $\la_{\ulp{m}}$ is 2-natural.
\end{lem}
\begin{proof}\proofof{lem:lambda-2-natural}
  It suffices to check 2-naturality for each subset $s \subset
  \ul{m}$.  Given $s \subset \ul{m}$, the 1-cell $[\id_{|s|},
  h_{\pi^s}]$ is a 2-natural transformation from $[|s|, \pi^s h(-)]$ to
  $[|s|, h \pi^s(-)]$ because $h$ is $\Ga$-lax and hence $h_{\pi^s}$ is
  2-natural.
\end{proof}
\begin{lem}
  \label{lem:lambda-cube-commutes}
  For all $\phi\cn \ul{m}_+ \to \ul{n}_+$ the following cube commutes:
\[
\begin{array}{c}
\begin{xy}
  (0,0)*+{X(\ulp{m})}="Xm";
  (20,20)*+{X(\ulp{n})}="Xn";
  (20,-20)*+{Y(\ulp{m})}="Ym";
  (50,20)*+{KPX(\ulp{n})}="KPXn";
  (50,-20)*+{KPY(\ulp{m})}="KPYm";
  (70,0)*+{KPY(\ulp{n})}="KPYn";
  (30,0)*+{KPX(\ulp{m})}="KPXm";  
  {\ar^-{X(\phi)} "Xm"; "Xn"};
  {\ar^-{\eta_{\ulp{n}}} "Xn"; "KPXn"};
  {\ar^-{KPh_{\ulp{n}}} "KPXn"; "KPYn"};
  {\ar_-{h_{\ulp{m}}} "Xm"; "Ym"};
  {\ar_-{\eta_{\ulp{m}}} "Ym"; "KPYm"};
  {\ar_-{KPY(\phi)} "KPYm"; "KPYn"};
  {\ar^-{\eta_{\ulp{m}}} "Xm"; "KPXm"};
  {\ar_-{KPX(\phi)} "KPXm"; "KPXn"};
  {\ar^-{KPh_{\ulp{m}}} "KPXm"; "KPYm"};
  {\ar@{=>}_-{\eta_{\phi}} "KPXm"+(-5,8); "Xn"+(3,-7) };
  {\ar@{=>}^-{\lambda_{\ulp{m}}} "Ym"+(5,8); "KPXm"+(-2,-7) };
  {\ar@{=}^-{} "KPXn"+(0,-18.9); "KPYm"+(0,18.9) };
\end{xy}
\\
\begin{xy}
  {\ar@{=}^{} (0,2.2)*+{}; (0,-2.2)*+{};};
\end{xy}
\\
\begin{xy}
  (0,0)*+{X(\ulp{m})}="Xm";
  (20,20)*+{X(\ulp{n})}="Xn";
  (20,-20)*+{Y(\ulp{m})}="Ym";
  (50,20)*+{KPX(\ulp{n})}="KPXn";
  (50,-20)*+{KPY(\ulp{m})}="KPYm";
  (70,0)*+{KPY(\ulp{n})}="KPYn";
  (40,0)*+{Y(\ulp{n})}="Yn";  
  {\ar^-{X(\phi)} "Xm"; "Xn"};
  {\ar^-{\eta_{\ulp{n}}} "Xn"; "KPXn"};
  {\ar^-{KPh_{\ulp{n}}} "KPXn"; "KPYn"};
  {\ar_-{h_{\ulp{m}}} "Xm"; "Ym"};
  {\ar_-{\eta_{\ulp{m}}} "Ym"; "KPYm"};
  {\ar_-{KPY(\phi)} "KPYm"; "KPYn"};
  {\ar^-{\eta_{\ulp{n}}} "Yn"; "KPYn"};
  {\ar^-{h_{\ulp{n}}} "Xn"; "Yn"};
  {\ar^-{Y(\phi)} "Ym"; "Yn"};
  {\ar@{=>}^-{\eta_{\phi}} "KPYm"+(-5,8); "Yn"+(3,-7) };
  {\ar@{=>}^-{\lambda_{\ulp{n}}} "Yn"+(5,8); "KPXn"+(-2,-7) };
  {\ar@{=>}^-{h_{\phi}} "Ym"+(0,18); "Xn"+(0,-18) };
\end{xy}
\end{array}
\]
\end{lem}
\begin{proof}\proofof{lem:lambda-cube-commutes}
  Commutativity of the cube in the statement of the lemma reduces to
  checking, for each object $x \in X(\ulp{m})$, commutativity of the
  following rectangle in $KPY(\ulp{n})$.
  \[
  \begin{xy}
    (0,0)*+{\phi_* \eta h (x)}="A";
    (50,0)*+{\eta \phi_* h (x)}="B";
    (100,0)*+{\eta h (\phi_* x)}="C";
    (0,-18)*+{\phi_* KPh\, \eta (x)}="D";
    (50,-18)*+{KPh\, \phi_* \eta (x)}="E";
    (100,-18)*+{KPh\, \eta (\phi_* x)}="F";
    {\ar^-{\eta_\phi} "A"; "B"};
    {\ar^-{\eta(h_\phi)} "B"; "C"};
    {\ar_-{\phi_* \la_{\ulp{m}}} "A"; "D"};
    {\ar_-{\id} "D"; "E"};
    {\ar_-{KPh(\eta_\phi)} "E"; "F"};
    {\ar^-{\la_{\ulp{n}}} "C"; "F"};
  \end{xy}
  \]
  %
  Formulas for the objects in this diagram are given as in
  \cref{eqn:eta-h,eqn:KPh-eta}.  Formulas for the 1-cells are given
  similarly, using \cref{defn:eta-phi} in addition.
  Commutativity of the rectangle above then reduces to checking, for each
  $s \subset \ul{n}$, commutativity of the following rectangle in
  $PY$.
  %
  \begin{eqn}
  \label{eqn:heart-rectangle}
  \begin{gathered}
  \begin{xy}
    (0,0)*+{\bsb{|\phi^{-1}(s)|}{\pi^{\phi^{-1}(s)}_* hx}}="A";
    (45,0)*+{\bsb{|s|}{\pi^s_* \phi_* hx}}="B";
    (95,0)*+{\bsb{|s|}{\pi^s_* h \phi_* x}}="C";
    (0,-18)*+{\bsb{|\phi^{-1}(s)|}{h \pi^{\phi^{-1}(s)}_* x}}="D";
    (45,-18)*+{\bsb{|\phi^{-1}(s)|}{h \pi^{\phi^{-1}(s)}_* x}}="E";
    (95,-18)*+{\bsb{|s|}{h\pi^s_* \phi_* x}}="F";
    {\ar^-{\bsb{\phi^s}{\id}} "A"; "B"};
    {\ar^-{\bsb{\id}{\pi^s_* (h_\phi)}} "B"; "C"};
    {\ar_-{\bsb{\id}{h_{\pi^{\phi^{-1}(s)}}}} "A"; "D"};
    {\ar_-{\id} "D"; "E"};
    {\ar_-{\bsb{\phi^s}{h_{\phi^s}}} "E"; "F"};
    {\ar^-{\bsb{\id}{h_{\pi^s}}} "C"; "F"};
  \end{xy}
  \end{gathered}
  \end{eqn}
  
  The top right composite of \cref{eqn:heart-rectangle} is
  $\bsb{\phi^s}{h_{\pi^s} \pi^s_{\ast}(h_\phi)}$ by
  \cref{rmk:grothendieck-composition-some-ids}.  Likewise, the bottom
  left composite is $\bsb{\phi^s}{h_{\phi^s} \circ
    (\phi^s)_*(h_{\pi^{\phi^{-1}(s)}})}$.  To verify that the second
  components are equal, consider \cref{eqn:fire-1} and \cref{eqn:fire-2}
  respectively: both are special cases of the axioms for $\Ga$-lax
  maps (see \cref{defn:D-lax-map}).

  \begin{eqn}
  \label{eqn:fire-1}
  \begin{gathered}
  \begin{xy}
    (0,0)*+{X(\ulp{m})}="A";
    (30,0)*+{X(\ulp{n})}="B";
    (60,0)*+{X(\ulp{|s|})}="C";
    (0,-20)*+{Y(\ulp{m})}="D";
    (30,-20)*+{Y(\ulp{n})}="E";
    (60,-20)*+{Y(\ulp{|s|})}="F";
    {\ar^-{\phi_*} "A"; "B"};
    {\ar^-{\pi^s_*} "B"; "C"};
    {\ar_-{\phi_*} "D"; "E"};
    {\ar_-{\pi^s_*} "E"; "F"};
    {\ar_-{h_{\ulp{m}}} "A"; "D"};
    {\ar^-{h_{\ulp{n}}} "B"; "E"};
    {\ar^-{h_{\ulp{|s|}}} "C"; "F"};
    {\ar@{=>}_-{h_\phi} "D"+(12,8); "B"+(-12,-8)};
    {\ar@{=>}_-{h_{\pi^s}} "E"+(12,8); "C"+(-12,-8)};
    (75,-10)*{=};
    (90,0)*+{X(\ulp{m})}="Xm";
    (120,0)*+{X(\ulp{|s|})}="Xs";
    (90,-20)*+{Y(\ulp{m})}="Ym";
    (120,-20)*+{Y(\ulp{|s|})}="Ys";
    {\ar^-{\pi^s_* \phi_*} "Xm"; "Xs"};
    {\ar_-{\pi^s_* \phi_*} "Ym"; "Ys"};
    {\ar_-{h_{\ulp{m}}} "Xm"; "Ym"};
    {\ar^-{h_{\ulp{|s|}}} "Xs"; "Ys"};
    {\ar@{=>}_-{h_{\pi^s \phi}} "Ym"+(12,8); "Xs"+(-12,-8)};
  \end{xy}
  \end{gathered}
  \end{eqn}

  \begin{eqn}
  \label{eqn:fire-2}
  \begin{gathered}
  \begin{xy}
    (0,0)*+{X(\ulp{m})}="A";
    (30,0)*+{X(\ulp{|\phi^{-1}(s)|})}="B";
    (60,0)*+{X(\ulp{|s|})}="C";
    (0,-20)*+{Y(\ulp{m})}="D";
    (30,-20)*+{Y(\ulp{|\phi^{-1}(s)|})}="E";
    (60,-20)*+{Y(\ulp{|s|})}="F";
    {\ar^-{\pi^{\phi^{-1}(s)}_*} "A"; "B"};
    {\ar^-{\phi^s_*} "B"; "C"};
    {\ar_-{\pi^{\phi^{-1}(s)}_*} "D"; "E"};
    {\ar_-{\phi^s_*} "E"; "F"};
    {\ar_-{h_{\ulp{m}}} "A"; "D"};
    {\ar^-{h_{\ulp{|\phi^{-1}(s)|}}} "B"; "E"};
    {\ar^-{h_{\ulp{|s|}}} "C"; "F"};
    {\ar@{=>}_-{\:\: h_{\pi^{\phi^{-1}(s)}}} "D"+(12,8); "B"+(-12,-8)};
    {\ar@{=>}_-{h_{\phi^s}} "E"+(12,8); "C"+(-12,-8)};
    (75,-10)*{=};
    (90,0)*+{X(\ulp{m})}="Xm";
    (120,0)*+{X(\ulp{|s|})}="Xs";
    (90,-20)*+{Y(\ulp{m})}="Ym";
    (120,-20)*+{Y(\ulp{|s|})}="Ys";
    {\ar^-{\phi^s_* \pi^{\phi^{-1}(s)}_*} "Xm"; "Xs"};
    {\ar_-{\phi^s_* \pi^{\phi^{-1}(s)}_*} "Ym"; "Ys"};
    {\ar_-{h_{\ulp{m}}} "Xm"; "Ym"};
    {\ar^-{h_{\ulp{|s|}}} "Xs"; "Ys"};
    {\ar@{=>}_-{\;\;\; h_{\phi^s \pi^{\phi^{-1}(s)}}} "Ym"+(12,8); "Xs"+(-12,-8)};
  \end{xy}
  \end{gathered}
  \end{eqn}

  By definition of $\phi^s$ we have $\phi^s \pi^{\phi^{-1}(s)} =
  \pi^s \phi$ (see \cref{eqn:pi-phi-eq-phi-pi}). Hence the right
  hand sides of \cref{eqn:fire-1} and \cref{eqn:fire-2} are equal.

\end{proof}


\subsection{Two equivalences}

Two lengthy calculations aside, we are now ready to state and prove
our main theorems.  We show that the unit $\eta$ is a stable
equivalence, and use this to prove our two main equivalences of
homotopy theories.  We begin with the following, which is immediate
from \cref{cor:gtrans,prop:eta-lax-natural}.
\begin{prop}
  \label{prop:eta-natural-in-hoGa2Cat-lax}
  In the homotopy category $\ho\GaIICat_l$, the unit $\eta\cn
  \mathrm{Id} \Rightarrow KP$ is a natural transformation of functors.
\end{prop}

Our next lemma is a key step in the proof that $\eta$ is a stable equivalence, and demonstrates the importance of the triangle identities.  This lemma depends on two careful computations of the
composites $(K\epz_{PX}) \circ (KP\eta_X)$ and $(K\epz_{PX}) \circ
(\eta_{KPX})$.  As the computations are completely self-contained, we
state their key application in \cref{lem:KPeta-and-etaKP} and postpone
the explicit details to \cref{prop:Kepz-eta-is-id,prop:epz-Peta-is-id}
in \cref{sec:K-epz-eta,sec:K-epz-K-P-eta}, respectively.
\begin{lem}
  \label{lem:KPeta-and-etaKP}
  Let $X$ be a $\Ga$-2-category.  In the homotopy category
  $\ho\GaIICat_l$ we have $\eta_{KPX} = KP\eta_X$.
\end{lem}
\begin{proof}\proofof{lem:KPeta-and-etaKP}
  First, note that $\epz_{PX}$ is a weak equivalence by
  \cref{prop:epz-weak-equiv}.  So
  the result follows from the two equalities
  \[
  (K\epz_{PX}) \circ (KP\eta_X) = \id_{KPX} = (K\epz_{PX}) \circ (\eta_{KPX}).
  \]
  The first of these follows from \cref{prop:epz-Peta-is-id} and
  functoriality of $K$ with respect to strict symmetric monoidal
  functors (see \cref{prop:functoriality-of-K}).  The second equality
  is proved in \cref{prop:Kepz-eta-is-id}. 
\end{proof}

\begin{thm}
  \label{thm:eta-stable-equiv}
  Let $X$ be a $\Ga$-2-category.  The unit $\eta = \eta_X$ is a stable
  equivalence.
\end{thm}
\begin{proof}\proofof{thm:eta-stable-equiv}
  We verify directly that for any very special $\Ga$-2-category, $Z$,
  the induced map
  \[
  \ho\GIICat_l(KPX, Z) \fto{\eta^*} \ho\GIICat_l(X,Z)
  \]
  is a bijection of sets.  To do this, we explicitly construct an
  inverse as follows. First, given and object $\ul{m}_+$ in $\sF$,
  Consider the diagram of 2-categories
    \[
  \begin{xy}
    (0,0)*+{Z(\ul{m}_+)}="A";
    (30,0)*+{KPZ(\ul{m}_+)}="B";
    (0,-20)*+{Z(\ul{1}_+)^m}="C";
    (30,-20)*+{KPZ(\ul{1}_+)^m},="D";
    {\ar^-{\eta_{Z,m}} "A"; "B"};
    {\ar_-{\eta_{Z,1}^m} "C"; "D"};
    {\ar "A"; "C"};
    {\ar "B"; "D"};
  \end{xy}
  \] 
  where the vertical maps are the Segal maps. This diagram commutes up
  to a 2-natural transformation. Since $Z$ and $KPZ$ are special, the
  two vertical maps are weak equivalences. Note that $\eta_{Z,1}$ is
  precisely the map of \cref{lem:X1-to-PX-weak-equiv}, thus it is a
  weak equivalence. It follows then that $\eta_Z$ is a levelwise weak
  equivalence.  For $h\cn X \to Z$, let $F(h)$ be given by the zigzag
  below:
  \[
  KPX \fto{KPh} KPZ \fot{\eta_Z} Z.
  \]
  To show that $F$ defines an inverse bijection of sets, we must
  verify $\eta^*(F(h)) = h$ for $h \cn X \to Z$ and $F(\eta^*(k)) = k$
  for $k \cn KPX \to Z$.  These follow by verifying that the
  diagrams below commute in $\ho \GIICat_l$.
  \[
  \begin{xy}
    (0,0)*+{X}="X";
    (25,0)*+{KPX}="KPX";
    (0,-20)*+{Z}="Z";
    (25,-20)*+{KPZ}="KPZ";
    {\ar^-{\eta_X} "X"; "KPX"};
    {\ar_-{\eta_Z} "Z"; "KPZ"};
    {\ar_-{h} "X"; "Z"};
    {\ar^-{KPh} "KPX"; "KPZ"};
    (55,0)*+{KPX}="KPXb";
    (80,0)*+{Z}="Zb";
    (55,-20)*+{KPKPX}="KP2Xb";
    (80,-20)*+{KPZ}="KPZb";
    {\ar^-{k} "KPXb"; "Zb"};
    {\ar_-{KP(k)} "KP2Xb"; "KPZb"};
    {\ar_-{KP(\eta_X)} "KPXb"; "KP2Xb"};
    {\ar^-{\eta_Z} "Zb"; "KPZb"};
  \end{xy}
  \] 
  The first commutes by \cref{prop:eta-natural-in-hoGa2Cat-lax}, and
  the second by
  \cref{lem:KPeta-and-etaKP,prop:eta-natural-in-hoGa2Cat-lax}.
\end{proof}
We can now use that $\eta$ is a stable equivalence to verify that $P$ preserves stable equivalences as well as weak equivalences.

\begin{prop}\label{prop:P-pres-steq}
The functor $P\cn\GaIICat_{l} \to \PIICat$ preserves and reflects stable equivalences, and so is a relative functor $(\GaIICat_{l}, \cS) \to (\PIICat, \cS)$.
\end{prop}
\begin{proof}\proofof{prop:P-pres-steq}
Let $h\cn X \to Y$ be a $\Ga$-lax map between $\Ga$-2-categories.  Then by \cref{prop:eta-natural-in-hoGa2Cat-lax}, we have
\[
[\eta][h] = [KPh][\eta]
\]
in $\ho \GaIICat_{l}$, so this equation holds in $\Ho \GaIICat_{l}$ as
well.  By \cref{thm:eta-stable-equiv}, we know that $\eta$ is a stable
equivalence, and since these satisfy the 2-out-of-3 property, $h$ is a
stable equivalence if and only if $KPh$ is.  But by definition (see
\cref{defn:stable-equiv-pgm}) $KPh$ is a stable equivalence if and
only if $Ph$ is, and therefore $h$ is a stable equivalence if and only
if $Ph$ is.
\end{proof}

\begin{thm}\label{thm:equiv-HoPGM-and-HoP2Cat-part2}
The inclusion $\PIICat \hookrightarrow \PGM$ induces an equivalence of
homotopy theories
\begin{align*}
(\PIICat, \cS) & \stackrel{\simeq}{\longrightarrow} (\PGM, \cS).
\end{align*}
\end{thm}
\begin{proof}\proofof{thm:equiv-HoPGM-and-HoP2Cat-part2}
The inclusion $\PIICat \hookrightarrow \PGM$ preserves stable equivalences by \cref{prop:inc_creates_stable}.  The rest of the proof is now the same as that of \cref{thm:main-css-2}, using \cref{prop:P-pres-steq}.
\end{proof}

We now turn to our final equivalence of homotopy theories.

\begin{lem}\label{lem:eta-zig}
The $\Ga$-lax maps $\eta_{X} \cn X \to KPX$ induce a natural zigzag of stable equivalences from $X$ to $KPX$ in the category $\GaIICat$.
\end{lem}
\begin{proof}\proofof{lem:eta-zig}
Recall the lax arrow category $\GaIICat^{\bullet \to
  \bullet}_{\mathrm{lax},\mathrm{str}}$ of \cref{D2Cat-lax-arrow}
whose objects are lax maps but with morphisms given by strict maps.
Since $\eta$ is natural with respect to strict maps by
\cref{cor:eta-nat-strict}, it defines a functor
\[
\eta \cn \GaIICat \to \GaIICat^{\bullet \to \bullet}_{\mathrm{lax},\mathrm{str}}.
\]
Composing this with the functor $E$ of
\cref{cor:E-on-Ga-arrow-category} yields a functor
\[
E\eta_{(-)} \cn \GaIICat \to \mathpzc{Span}\,(\GaIICat).
\]
This means that given a strict $\Ga$-map $h\cn X \to Y$, one has the
following commuting diagram in $\GaIICat$.
\begin{eqn}\label{eqn:span}
\begin{xy}
  (0,0)*+{X}="0";
  (20,0)*+{E\eta_X}="1";
  (40,0)*+{KPX}="2";
  (0,-15)*+{Y}="3";
  (20,-15)*+{E\eta_Y}="4";
  (40,-15)*+{KPY}="5";
  {\ar_-{\omega} "1"; "0"};
  {\ar^-{\nu} "1"; "2"};
  {\ar^-{\omega} "4"; "3"};
  {\ar_-{\nu} "4"; "5"};
  {\ar_-{h} "0"; "3"};
  {\ar_-{E\eta_h} "1"; "4"};
  {\ar^-{KPh} "2"; "5"};
\end{xy}
\end{eqn}
Therefore the functor $E\eta$ provides, for each $\Ga$-2-category $X$,
a zigzag of strict $\Ga$-maps which is natural with respect to strict
$\Ga$-maps.

The maps $\om$ are levelwise weak equivalences by
\cref{prop:om-levelwise-adj} and hence stable equivalences.  Further, commutativity of the left-hand square shows that $X \mapsto E\eta_{X}$ preserves stable equivalences, and so is a relative functor.
\Cref{cor:nu-stable-eq} implies that the maps $\nu$ are stable
equivalences because, by \cref{thm:eta-stable-equiv}, the components
of $\eta$ are stable equivalences.
\end{proof}

\begin{thm}\label{thm:main-css-1}
The functors 
\begin{align*}
K\cn\PIICat\to\GaIICat, \\
P\cn\GaIICat\to\PIICat
\end{align*}
induce an equivalence of homotopy theories between $(\PIICat, \cS)$ and $(\GaIICat, \cS)$.
\end{thm}
\begin{proof}\proofof{thm:main-css-1}
By \cref{cor:rel-cats-sufficient}, we need to construct a natural
zigzag of stable equivalences between the composites $KP, PK$ and the
respective identity functors.  The first of those is given in
\cref{lem:eta-zig}, while the second is given by the transformation
$\epz$ using \cref{prop:epz-is-natural} and
\cref{prop:epz-weak-equiv}.
\end{proof}

\subsection{The composite \texorpdfstring{$K \epz _{\cC} \circ
    \eta_{K\cC}$}{Kε η}}
\label{sec:K-epz-eta}

In this section we prove the following:
\begin{prop}
  \label{prop:Kepz-eta-is-id}
  Let $\cC$ be a permutative 2-category.  Then the composite $\Ga$-lax
  map of $\Ga$-2-categories is the identity on $K\cC$:
  \[
  K\cC \fto{\eta_{K\cC}} KPK\cC \fto{K\epz _{\cC}} K\cC.
  \] 
\end{prop}
\begin{proof}\proofof{prop:Kepz-eta-is-id}
Throughout this section we work at a fixed permutative 2-category
$\cC$ and will often simply write $\eta = \eta_{K\cC}$ for simplicity.  We first describe
\[
\eta_{K\cC} \cn K \cC \to KPK \cC.
\]

For $\ul{m}_+ \in \sF$ we use \cref{defn:eta} to describe $\eta =
\eta_{\ul{m}_+}$ on the following 0-, 1-, and 2-cells in $K\cC(\ul{m}_+)$:
\[
\begin{xy}
  (0,0)*+{\{x,c\}}="A";
  (24,0)*+{\{y,d\}}="B";
  {\ar@/^1pc/^{\{f\}} "A"; "B"};
  {\ar@/_1pc/_{\{g\}} "A"; "B"};
  {\ar@{=>}^{\{\al\}} (10,3)*+{}; (10,-3)*+{} };
\end{xy}
\]
On 0-cells:
\[
\eta\big(\{x,c\}\big) = \iicb{1.4}{
  (s \subset \ul{m})
}{
  \bsb{|s|}{\pi^s_*\{x,c\}}
}{
  \stdisj{\ul{m}}
}{
  \bsb{\pi^{s,t}}{\id_{\pi^s_*\{x,c\} \times \pi^t_*\{x,c\}}}
}
\]
where, unraveling the definition of $\pi^s$, we have
\[
\pi^s_*\{x,c\} = 
\iicb{1.4}{
  (r \subset \ul{|s|})
}{
  x_{(\pi^s)^{-1}(r)}
}{
  \left(
    \genfrac{}{}{0pt}{}{r,u \subset \ul{|s|}}{r \cap u = \emptyset} 
  \right)
}{
  c_{(\pi^s)^{-1}(r),(\pi^s)^{-1}(u)}
}.
\]
Note that we make use of \cref{rmk:pist-pis-pit} to see that 
$\eta\big(\{x,c\}\big)_{s,t}$ is a well-defined map in $PK\cC$:
\[
\bsb{\pi^{s,t}}{\id_{\pi^s_*\{x,c\} \times \pi^t_*\{x,c\}}}\cn
\bsb{|s \cup t|}{\pi^{s \cup t}_*\{x,c\}}
\to
\bsb{(|s|,|t|)}{\pi^s_*\{x,c\} \times \pi^t_*\{x,c\}}.
\]
On 1-cells:
\[
\eta\big(\{f\}\big) = \left\{
  (s \subset \ul{m}) \quad \mapsto \quad
  \bsb{\id_{|s|}}{\pi^s_*\{f\}}
\right\}.
\]
On 2-cells:
\[
\eta\big(\{\alpha\}\big) = \left\{
  (s \subset \ul{m}) \quad \mapsto \quad
  \bsb{\id_{|s|}}{\pi^s_*\{\alpha\}}
\right\}.
\]

Now we describe the composite $K\epz _{\cC} \circ \eta_{K\cC}$ using the
description of $\eta_{K\cC}$ above, the definition of $\epz$
(described explicitly for $\epzo$ in \cref{prop:epzo-pseudofunctor}) and the
functoriality of $K$ described in \cref{prop:functoriality-of-K}.

On 0-cells:
\[
K\epz_\cC \circ \eta_{K\cC}\big(\{x,c\}\big) = \iicb{1}{
  (s \subset \ul{m})
}{
  \epz_{|s|} \bsb{|s|}{\pi^s_*\{x,c\}} =
  x_{(\pi^s)^{-1}(\ul{|s|})} = x_s
}{
  \stdisj{\ul{m}}
}{
  \epz_{(|s|,|t|)}(\id_{\pi^s_*\{x,c\} \times \pi^t_*\{x,c\}}) \circ \epz_{\pi^{s,t}}
}.
\]
Now $\epz_{\vec{m}}(\id_{\vec{x}}) = \id_{\epz_{\vec{m}}(\vec{x})}$ 
and, unraveling the definition of $\epz_{\ol{\phi}}$ for $\ol{\phi} =
\pi^{s,t}$, one finds
\[
\epz_{\pi^{s,t}} = c_{s,t}\cn x_{s \cup t} \to x_s \oplus x_t.
\]
On 1-cells:
\[
K\epz_\cC \circ \eta_{K\cC} \big(\{f\}\big) = \left\{s \mapsto 
\epz_{|s|}[\id_{|s|}, \pi^s\{f\}]\right\} = \left\{f\right\}.
\]
Similarly, on 2-cells:
\[
K\epz_\cC \circ \eta_{K\cC} \big(\{\alpha\}\big) = \left\{\alpha\right\}.
\]
Therefore $(K\epz_\cC \circ \eta_{K\cC})_{\ulp{m}}$ is the
identity on 0-, 1-, and 2-cells of $K\cC(\ulp{m})$.  

Now $\epz$ is a strict map of 2-categories by
\cref{prop:epz-times-is-strict-SM} and therefore $K\epz$ is a strict map
of $\Ga$-2-categories.  However, $\eta$ is only a $\Ga$-lax map and so
for a map of finite sets $\phi\cn \ul{m}_+ \to \ul{n}_+$, we have the
following diagram.
\[
\begin{xy}
  (0,20)*+{K \cC(\ul{m}_+)}="A1";
  (30,20)*+{KPK\cC(\ul{m}_+)}="A2";
  (60,20)*+{K \cC(\ul{m}_+)}="A3";
  (0,0)*+{K \cC(\ul{n}_+)}="B1";
  (30,0)*+{KPK\cC(\ul{n}_+)}="B2";
  (60,0)*+{K \cC(\ul{n}_+)}="B3";
  {\ar^-{\eta} "A1";"A2" };
  {\ar^-{K\epz} "A2";"A3" };
  {\ar_-{\phi} "A1";"B1" };
  {\ar_-{\phi} "A2";"B2" };
  {\ar_-{\eta} "B1";"B2" };
  {\ar_-{K\epz} "B2";"B3" };
  {\ar_-{\phi} "A3";"B3" };
  (45,10)*+{=};
  {\ar@{=>}^{\eta_{\phi}} "A2"+(-13,-8)="a"; "a"+(-5,-5)};
\end{xy}
\]
To verify that $K\epz \circ \eta$ is the identity as a lax transformation,
we need to verify that this 2-cell composite -- formed by whiskering
$\eta_\phi$ with $K\epz$ -- is the identity.  To do this, we must
consider
\[
K\epz(\eta_{\phi}) \cn 
K\epz\big(\phi_* \eta\big(\{x,c\}\big)\big) \to 
K\epz\big(\eta \big(\phi_* \{x,c\}\big)\big).
\]

First, we describe the source and target of $\eta_\phi$ more explicitly:
\begin{align*}
\phi_* \eta\big(\{x,c\}\big) & = \iicb{2}{
  (s \subset \ul{n})
}{
  \lrsb{|\phi^{-1}(s)|}{\pi^{\phi^{-1}(s)}_*\{x,c\}}
}{
  \stdisj{\ul{n}}
}{
  \lrsb{
    \pi^{\phi^{-1}(s),\phi^{-1}(t)}
  }{
    \id_{\pi^{\phi^{-1}(s)}_*\{x,c\} \times \pi^{\phi^{-1}(t)}_*\{x,c\}}}
}\\
\intertext{and}
\eta\big(\phi_*\{x,c\}\big) & = \iicb{2}{
  (s \subset \ul{n}) 
}{
  \lrsb{|s|}{\pi^s_*\phi_*\{x,c\}}
}{
  \stdisj{\ul{n}}
}{
  \lrsb{
    \pi^{s,t}
  }{
    \id_{\pi^s_*\phi_*\{x,c\} \times \pi^t_*\phi_*\{x,c\}}
  }
}.
\end{align*}
Note that 
\[
\phi_*\{x,c\} = \iicb{1}{s}{x_{\phi^{-1}(s)}}{
  s,t
}{
  c_{\phi^{-1}(s) \, , \, \phi^{-1}(t)}
}.
\]
Using the notation of \cref{sec:eta-is-lax-transformation}, the map
$\eta_\phi$ is given by
\[
\left\{ 
(s \subset \ul{n}) \mapsto 
\bsb{\phi^s}{\id_{\pi^s_{\ast}\phi_*\{x,c\}}} 
\right\}.
\]

Hence we have
\[
K\epz(\eta_{\phi}) = 
\left\{
(s \subset \ul{n}) \mapsto \epz\big(
\bsb{\phi^s}{\id_{\pi^s_*\phi_*\{x,c\}}}
\big) = 
\epz(\id_{\pi^s_* \phi_*\{x,c\}}) \circ \epz_{\phi^s}
\right\}.
\]
Unraveling the definitions, one finds
\[
\epz(\id_{\pi^s \phi\{x,c\}}) = \id_{x_{\phi^{-1}(s)}}
\]
and by \cref{rmk:epz-phibar-is-id-for-map-of-finite-sets}
\[
\epz_{\phi^s} = 
\id_{x_{\phi^{-1}(s)}}.
\]
Therefore $K\epz(\eta_{\phi}) = \{s \mapsto \id_{x_{\phi^{-1}(s)}}\} =
\id_{\{x,c\}}$ and this
completes the proof.  \end{proof}

\subsection{The composite \texorpdfstring{$\epz_{P X} \circ P\eta_X$}{ε Pη}}
\label{sec:K-epz-K-P-eta}

In this section we prove the following:
\begin{prop}
  \label{prop:epz-Peta-is-id}
  Let $X$ be a $\Ga$-2-category.  Then the composite 
  \[
  PX \fto{P \eta_X} PKPX \fto{\epz_{P X}} PX
  \]
  of strict symmetric monoidal functors is the identity on $PX$.
\end{prop}

\begin{proof}\proofof{prop:epz-Peta-is-id}
First, recall that $\epz_{PX}$ is a strict symmetric monoidal
functor by \cref{prop:epz-times-is-strict-SM} and $P \eta_X = \sA
\int A\eta_X$ is a strict symmetric monoidal functor by
\cref{prop:Ah-has-monoidal-structure,prop:symm-D-lax-to-sm-2functor}.

To show that $\epz_{PX} \circ P\eta_X = \id_{PX}$ as maps of permutative 2-categories,
we check this equality on 0-, 1-, and 2-cells.
We use the definition of $\eta$ in \cref{defn:eta}, the description of
$P$ in \cref{sec:definition-of-P}, and the definition of $\epz$ (see
\cref{defn:epz,prop:epzo-pseudofunctor}).  On 0-cells, the definitions of
$\epz$ and $A\eta$ immediately give:
\begin{align*}
\epz_{PX} \circ P\eta_X \left( \bsb{\vec{m}}{\vec{x}} \right) 
& = \epz_{PX}\left( \bsb{\vec{m}}{A\eta_{\vec{m}}(\vec{x})} \right) \\
& = \epz_{\vec{m}}(A\eta_{\vec{m}}(\vec{x})) \\
& = \bsb{\vec{m}}{\vec{x}}.\\
\end{align*}
On 1-cells:
\begin{align*}
\epz_{PX} \circ P\eta_X \left( \bsb{\ol{\phi}}{\vec{f}} \right) 
& = \epz_{PX}\left( \bsb{\ol{\phi}}{A\eta_{\vec{n}}(\vec{f}) \circ A\eta_{\ol{\phi}}} \right)\\
& = \epz_{\vec{n}}(A\eta_{\vec{n}}(\vec{f}) \circ A\eta_{\ol{\phi}}) \circ \epz_{\ol{\phi}}.\\ 
\end{align*}
Because, by \cref{rmk:grothendieck-composition-some-ids}, each 1-cell
is given as a composite
\[
\bsb{\ol{\phi}}{\vec{f}} = \bsb{\id}{\vec{f}} \bsb{\ol{\phi}}{\id},
\]
and because both $P\eta_X$ and $\epz_{PX}$ are 2-functors (see
\cref{prop:Ah-has-monoidal-structure,prop:symm-D-lax-to-sm-2functor,prop:epz-times-is-strict-SM}),
it suffices to consider 1-cells of the form $\bsb{\id}{\vec{f}}$ or
$\bsb{\ol{\phi}}{\id}$.  In the first case, it is immediate that
$\epz_{PX} \circ P\eta_X \left( \bsb{\id}{\vec{f}} \right) = \bsb{\id}{\vec{f}}$
because $A\eta_{\id} = \id$ and $\epz_{\id} = \id$.  In the second
case, we have
\begin{align*}
\epz_{PX} \circ P\eta_X \left( \bsb{\ol{\phi}}{\id} \right) 
& = \epz_{\vec{n}}(A\eta_{\ol{\phi}}) \circ \epz_{\ol{\phi}}.\\ 
\end{align*}

This composite is, by definition, the component at an object $\vec{x} \in X(\vec{m})$ of the transformation formed by the pasting below.  
As in \cref{rmk:description-Ph,prop:epzo-pseudofunctor}, the object $\vec{x}$ has been suppressed from the notation.
\begin{eqn}
\label{eqn:epz-Aeta-phibar}
\begin{xy}
  (-23,30)*+{AX(\vec{m})}="A";
  (30,30)*+{AKPX(\vec{m})}="A1";
  (-23,0)*+{\prod_i \prod_{j \in \mathfrak{p}} X(\ul{n_j}_+)}="B";
  (30,0)*+{\prod_i \prod_{j \in \mathfrak{p}} KPX(\ul{n_j}_+)}="B1";
  (-23,-30)*+{AX(\vec{n})}="C";
  (30,-30)*+{AY(\vec{n})}="C1";
  (75,0)*+{PX}="D";
  {\ar^-{A\eta_{\vec{m}}} "A"; "A1"};
  {\ar^-{\prod_i \prod_{j \in \mathfrak{p}} \eta_{\ulp{n_j}}} "B"; "B1"};
  {\ar_-{A\eta_{\vec{n}}} "C"; "C1"};
  {\ar_-{\prod_i {\phi}_i} "A"; "B"};
  {\ar_-{\prod_i {\phi}_i} "A1"; "B1"};
  {\ar_-{\tau} "B"; "C"};
  {\ar_-{\tau} "B1"; "C1"};
  {\ar@/^1.7pc/^-{\epz_{\vec{m}}} "A1"; "D"};
  {\ar^-{\oplus_i \epz_{({n_j})_{j \in \mathfrak{p}}}} "B1"; "D"};
  {\ar@/^-1.7pc/_-{\epz_{\vec{n}}} "C1"; "D"};
  {\ar@{=>}_{\prod_i A\eta_{{\phi}_i}} (5,19)*+{}="tc0"; "tc0"+(-5,-5) };
  {\ar@{=}_{} (5,-11)*+{}="tc1"; "tc1"+(-4,-4) };
  {\ar@{=>}_{\oplus_i \epz_{{\phi}_i}} (48,16)*+{}="tc2"; "tc2"+(-5,-5) };
  {\ar@{=>}_{\beta} (48,-11)*+{}="tc2"; "tc2"+(-5,-5) };
\end{xy}
\end{eqn}
The lower 2-cell $\beta$ is given, as in the right-hand side of
\cref{eq:defn-epzo-barphi}, by the 1-cell $\beta$ for the relevant
permutation of terms
\[
\oplus_i \epz_{(n_j)_{j \in \mathfrak{p}}}\big(
  ({\phi}_{i})_*\{x^i, c^i\}\big) \fto{\ \beta\ } \epz_{\vec{n}}\big(\ol{\phi}_*\{\vec{x},\vec{c}\}\big).
\]
This is precisely the permutation of components in
\cref{rmk:decomposing-phibar-partition-map-reorder} which makes the
following diagram of finite sets commute.
\[
\begin{xy}
  (0,0)*+{\coprod_i \ul{m_i}}="m";
  (60,0)*+{\coprod_j \ul{n_j}}="n";
  (30,-12)*+{\coprod_i \coprod_{j \in \mathfrak{p}} \ul{n_j}}="o";
  {\ar^-{\ol{\phi}} "m"; "n"};
  {\ar_-{\coprod_i {\phi}_i} "m"; "o"};
  {\ar "o"; "n"};
\end{xy}
\]
Therefore, restricting to the top half of \cref{eqn:epz-Aeta-phibar},
it suffices to show that for each $i$ we have
\begin{eqn}
\label{eqn:epz-eta-at-phibar-i}
\epz_{(n_j)_{j \in \mathfrak{p}}}\big(\eta_{{\phi}_i}\big) 
\circ \epz_{{\phi}_i} 
= [{\phi}_i, \id].
\end{eqn}
This left-hand composite is, for each $x^i \in X(\ul{m_i}_+)$, given
as follows:
\begin{eqn}
\label{eqn:epz-eta-at-phibar-i-long}
\ev_{\ul{m_i}} \eta_{\ulp{m_i}} (x^i) \fto{\epz_{{\phi}_i}}
\left(
  \prod_{j \in \mathfrak{p}} \ev_{\ul{n_j}} 
\right) {\phi}_{i*} \eta_{\ulp{m_i}} (x^i) 
\fto{
  \ \left(\prod_{j \in \mathfrak{p}} \ev_{\ul{n_j}}\right)(A\eta_{{\phi}_i}) \ 
}
\left( 
  \prod_{j \in \mathfrak{p}} \ev_{\ul{n_j}} \eta_{\ulp{n_j}} 
\right) {\phi}_{i*} (x^i).
\end{eqn}

To understand this composite, we make critical use of
\cref{notn:phibar-partition-indexing,notn:decomposing-phibar-partition-map,notn:phis}.
From the definition of $\eta$ in \cref{eq:defn-eta}, we have
\begin{eqn*}
\eta_{\ulp{m_i}}(x^i) = 
\iicb{1.4}%
  {(s \subset \ul{m_i})}{\bsb{|s|}{\pi^s_* x^i}}%
  {\stdisj{\ul{m_i}}}{\bsb{\pi^{s,t}}{\id_{\pi^s_* x^i \times \pi^t_* x^i}}}.
\end{eqn*}
By definition (see \cref{eq:defn-epzo-barphi}), $\epz_{{\phi}_i}$
is the map 
\[
\bsb{|m_i|}{x^i} \to 
\lrsb{\left(|{\phi}_i^{-1}(n_j)|\right)_{j \in \mathfrak{p}}}{\prod_{j \in \mathfrak{p}}\pi^{\left({\phi}_i^{-1}(\ul{n_j})\right)}_* x^i}
\]
induced by the partition of $\ul{m_i}$ into $\coprod_{j \in \mathfrak{p}}
\ul{|{\phi}_i^{-1}(\ul{n_j})|}$ and the identity on
$\prod_{j \in \mathfrak{p}}\pi^{{\phi}_i^{-1}(\ul{n_j})}_* x^i$.
Next, from the definition of $A\eta_{{\phi}_i}$
in \cref{eq:defn-h-barphi} we have
\[
A\eta_{{\phi}_i} = \prod_{j \in \mathfrak{p}} \left\{
(s \subset \ul{n_j}) \quad \mapsto \quad
\bsb{({\phi}_{i,j})^{s}}{\id}
\right\}
\]
and hence
\begin{eqn}
\label{eqn:ev-Aeta}
\left(\prod_{j \in \mathfrak{p}}\ev_{\ul{n_j}}\right)
\big(A\eta_{{\phi}_i}\big) 
= \prod_{j \in \mathfrak{p}} \bsb{({\phi}_{i,j})^{\ul{n_j}}}{\id}.
\end{eqn}
Each $({\phi}_{i,j})^{\ul{n_j}}$ is given by the restriction of
${\phi}_{i,j}$ to $({\phi}_{i,j})^{-1}(\ul{n_j}) \subset \ulp{m_i}$
and a reindexing
$\ul{|({\phi}_{i,j})^{-1}(\ul{n_j})|} \iso
({\phi}_{i,j})^{-1}(\ul{n_j})$.
Hence the product in \cref{eqn:ev-Aeta} is given in the first component
by the disjoint union of the maps
\[
{\phi}_{i,j}\big|_{({\phi}_{i,j})^{-1}(\ul{n_j})} =
{\phi}_i\big|_{({\phi}_{i,j})^{-1}(\ul{n_j})} =
{\phi}_i\big|_{({\phi}_{i})^{-1}(\ul{n_j})}
\]
appropriately reindexed to have source
$\ul{|({\phi}_{i})^{-1}(\ul{n_j})|}$.

The composite \eqref{eqn:epz-eta-at-phibar-i-long} of maps in $PX$ is
given by composing the first components in $\sA$, since both are
identities in their second component.  This is
\[
\ul{m_i} \hspace{1pc}\longrightarrow
\coprod_{j\in \mathfrak{p}(\ol{\phi},i)}
\ul{|{\phi}_i^{-1}(\ul{n_j})|} \quad \iso 
\coprod_{j\in \mathfrak{p}(\ol{\phi},i)} 
{\phi}_i^{-1}(\ul{n_j}) \hspace{1pc}\longrightarrow
\coprod_{j\in \mathfrak{p}(\ol{\phi},i)} \ul{n_j}.
\]
But this is precisely the decomposition of ${\phi}_i$
described in \cref{rmk:decomposing-phibar-partition-map-reorder} and
is equal to ${\phi}_i$.
Therefore the composite \eqref{eqn:epz-eta-at-phibar-i} is equal to
$[{\phi}_i, \id]$ as claimed.

Lastly, on 2-cells we have
\begin{align*}
\epz_{PX} \circ P\eta_X \left( \bsb{\ol{\phi}}{\vec{\al}} \right) 
& = \epz_{PX} \left( \bsb{\ol{\phi}}{A\eta_{\vec{n}}(\vec{\al}) * 1_{A\eta_{\ol{\phi}}}} \right) \\
& = \epz_{\vec{n}}(A\eta_{\vec{n}}(\vec{\al}) * 1_{A\eta_{\ol{\phi}}})
* 1_{\epz_{\ol{\phi}}}\\
& = \bsb{\ol{\phi}}{\vec{\al}}
\end{align*}
where the last equality follows by again considering the cases
$\vec{\al} = \id$ and $\ol{\phi} = \id$ separately. This completes the
proof.
\end{proof}


\clearpage
\appendix
\section{Diagram of key structures}\label{appendix}

In this section we give an overview of the categories, notions of weak
equivalences, functors, and transformations which combine to prove the
main results.  We first give a diagram showing categories and
functors in \cref{fig:main-diagram}.  We recall the definitions of
these structures, and then the transformations which connect them.

\begin{figure}[h]
\[
\begin{xy}
  (0,0)*+{\GaIICat_{l}}="Ga2Cl";
  (-36,0)*+{\GaIICat}="Ga2C";
  (-54,30)*+{\Ga\mh\IICat^{\bullet \to \bullet}_{\mathrm{lax},\mathrm{str}}}="Arr";
  (-72,0)*+{\mathpzc{Span}\,(\Ga\mh\IICat)}="Span";
  (36,0)*+{(\sA,\oplus)\mh\IICat_{l}}="A2Cl";
  (36,20)*+{\PIICat}="P2C";
  (36,40)*+{\PGM}="PGM";
  (36,47)*+{\qsSMIICat}="qsSM2";
  (36,65)*+{\SMBicat}="SMB";
  (36,43.5)*+{\rotatebox{270}{$\iso$}}="iso";
  {\ar_-{\eta} "Ga2C"; "Arr"};
  {\ar_-{E} "Arr"; "Span"};
  {\ar@/^2pc/^-{pr_{\ell}} "Span"; "Ga2C"}; 
  {\ar^-{pr_m} "Span"; "Ga2C"}; 
  {\ar@/_2pc/_-{pr_r} "Span"; "Ga2C"}; 
  {\ar_-{j} "Ga2C"; "Ga2Cl"};
  {\ar_-{A} "Ga2Cl"; "A2Cl"};
  {\ar_-{\sA \wre (-)} "A2Cl"; "P2C"};
  {\ar_-{i} "P2C"; "PGM"};
  {\ar@/_1pc/_-{(-)^{qst}} "SMB"; "qsSM2"}; 
  {\ar@/_1pc/_-{u} "qsSM2"; "SMB"}; 
  {\ar@/_4pc/_-{\Ko} "PGM"; "Ga2C"};
  {\ar@/_2pc/_-{\Kt} "P2C"; "Ga2C"};
\end{xy}
\]
\caption{This diagram depicts categories and functors.  Note that it
  is not entirely commutative.  Transformations connecting these
  functors are described in \cref{sec:transformations}.}\label{fig:main-diagram}
\end{figure}
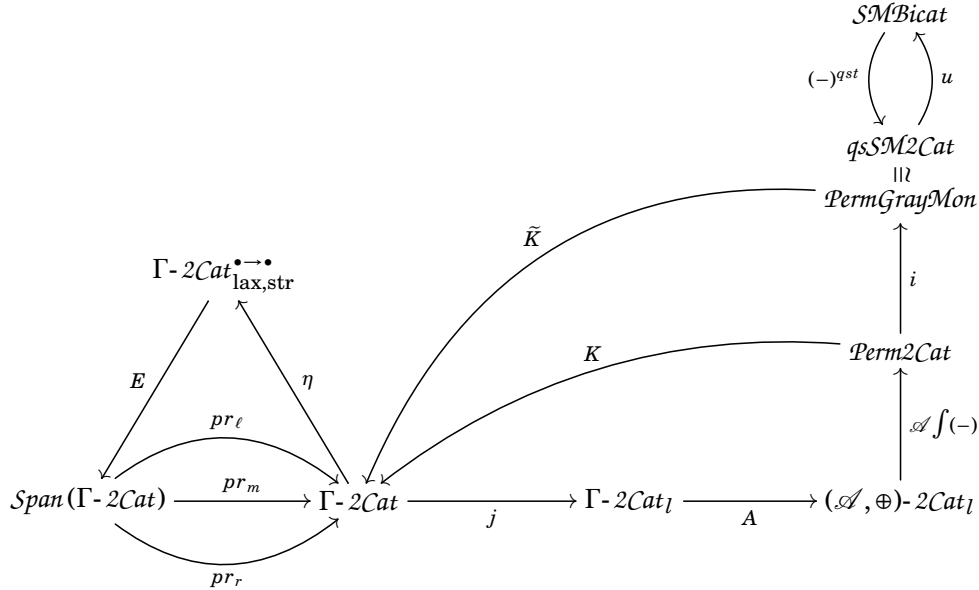

\subsection{Categories}\label{sec:categories}
\ 

\noindent $\GaIICat_{l}$\\
\indent $\Ga$-2-categories with lax maps.   \cref{defn:D2Cat-lax-category}
\ \\

\noindent $\GaIICat$\\
\indent $\Ga$-2-categories with strict maps.  \cref{defn:gammaobjs}
\ \\

\noindent $\Ga\mh\IICat^{\bullet \to \bullet}_{\mathrm{lax},\mathrm{str}}$\\
\indent Arrow category whose objects are lax maps of
$\Ga$-2-categories but whose morphisms are pairs of strict maps which
give a strictly commuting square. \cref{D2Cat-lax-arrow}
\ \\

\noindent $\mathpzc{Span}\,(\Ga\mh\IICat)$\\
\indent Spans of $\Ga$-2-categories and strict maps. \cref{defn:span} 
\ \\

\noindent $(\sA,\oplus)\mh\IICat_{l}$\\
\indent Symmetric monoidal diagrams on $\sA$ taking values in
$\IICat$, and monoidal  $\sA$-lax maps.  \cref{defn:cat-of-sm-diagrams}
\ \\


\noindent $\PIICat$\\
\indent Permutative 2-categories with strict functors.  \cref{defn:perm-2-cat,defn:PIICat}
\ \\

\noindent $\PGM$\\
\indent Permutative Gray-monoids with strict functors.
\cref{defn:pgm,prop:defn-PGM}\\ 
\indent Variants given in \cref{prop:PGM-variants}.
\ \\

\noindent $\qsSMIICat$\\
\indent Quasi-strict symmetric monoidal bicategories. \cref{notn:smb-qst}\\
\indent Isomorphic to $\PGM$ by \cref{p2isoqs2}.
\ \\

\noindent $\SMBicat$\\
\indent Symmetric monoidal bicategories with strict functors. \cref{notn:smb-qst}
\ \\

\subsection{Homotopy theories}\label{sec:homotopy-theories}

\begin{itemize}
\item Weak equivalences of 2-categories are denoted $\cW$ and are
  created by the nerve functor to simplicial sets. 
\item The prefix $\ho$ (lowercase) denotes a homotopy category obtained
  by inverting weak equivalences $\cW$.
\item Stable equivalences of 2-categories are defined only for
  symmetric monoidal 2-categories and are created by the functors
  $\Ko$ and $\Kt$.  In cases where both versions of $K$-theory are
  defined, these result in the same class of stable equivalences.
\item The prefix $\Ho$ (capitalized) denotes a homotopy category
  obtained by inverting stable equivalences.
\end{itemize}

\subsection{Functors}\label{sec:functors}
\ \\

\noindent $i$, $j$, $u$\\
\indent Inclusions of subcategories (this notation is not used in the
text, only in this diagram)
\ \\

\noindent $pr_{\ell}$, $pr_m$, $pr_r$\\
\indent Three forgetful functors from spans to the left, respectively
middle, respectively right, objects of a span (this notation is not
used in the text, only in this diagram)
\ \\

\noindent $(-)^{qst}$\\
\indent Defined in \cref{cohqs2cats}
\ \\

\noindent $P = \sA \wre (A(-))\cn \GaIICat \to (\sA,\oplus)\mh\IICat_l$\\
\indent Defined in \cref{sec:permutative-2-cats-from-Gamma-2-categories}, see
summary there
\ \\

\noindent $\Ko\cn \PIICat \to \GaIICat$\\
$\Kt\cn \PGM \to \GaIICat$\\
\indent Defined in \cref{sec:K-epz-fluffy,sec:K-epz-strict}, respectively
\ \\

\noindent $E\cn \Ga\mh\IICat^{\bullet \to \bullet}_{\mathrm{lax},\mathrm{str}} 
\to \mathpzc{Span}\,(\Ga\mh\IICat)$\\
\indent Given by \cref{cor:E-on-Ga-arrow-category}
\ \\

\noindent $\eta \cn \GaIICat \to \mathpzc{Span}\,(\Ga\mh\IICat)$.\\
\indent Defined in the proof of \cref{lem:eta-zig}
\ \\

\subsection{Transformations}\label{sec:transformations}
\ \\

\noindent $\varphi \cn  \Kt \Rightarrow \Ko i$ 
\begin{itemize} 
\item induced by inclusion (\cpageref{Kt-to-Ko,Kt-to-Ko-formulas});
\item componentwise strict map of $\Ga$-2-categories; 
\item natural with respect to strict symmetric monoidal 2-functors;
\item componentwise weak equivalence
\end{itemize}
\ \\

\noindent $\nu \cn  (u(-))^{qst} \Rightarrow \Id$ \\
\indent Quasi-strictification of a quasi-strict symmetric monoidal
  bicategory is biequivalent to the identity. \cref{thm:cohqs2cats2}
\ \\

\noindent $\epzo \cn  P\Ko \Rightarrow \Id$
\begin{itemize}
\item defined in \cref{prop:epzo-pseudofunctor};
\item componentwise symmetric monoidal pseudofunctor
  (\cref{thm:epzo-symm-mon});
\item strictly natural with respect to strict functors
  (\cref{prop:epzo-natural});
\item componentwise weak equivalence (\cref{prop:epzo-weak-equiv})
\end{itemize}

Note the components of $\epzo$ are generally not 2-functors, and
therefore generally not maps in $\PGM_v$ for any morphism variant
$v$.  This means that $\epzo$ is a coherent collection of
pseudofunctors, but not actually a natural transformation.
\ \\

\noindent $\epz \cn  P\Kt \Rightarrow \Id$ 
\begin{itemize}
\item $\epz = \epzo \circ P\varphi$ (\cref{defn:epz});
\item strict symmetric monoidal 2-functor
  (\cref{prop:epz-times-is-strict-SM});
\item natural with respect to strictly-monoidal 2-functors and lax-natural with respect
  to lax-monoidal 2-functors (\cref{prop:epz-is-natural}); 
\item componentwise weak equivalence (\cref{prop:epz-weak-equiv})
\end{itemize}
\ \\

\noindent $\eta \cn  \Id \Rightarrow \Kt P$
\begin{itemize}
\item defined in \cref{defn:eta};
\item componentwise $\Ga$-lax map
  (\cref{sec:eta-is-lax-transformation});
\item strictly natural with respect to strict $\Ga$-maps
  (\cref{cor:eta-nat-strict,rmk:eta-lax-nat-maybe});
\item componentwise stable equivalence (\cref{thm:eta-stable-equiv})
\end{itemize}
\ \\

\noindent $(\om,\nu)\cn \Id \Leftarrow E\eta_{(-)} \Rightarrow \Kt P$\\
Zigzag defined in \cref{lem:eta-zig} using the span construction of
\cref{defn:E}.  \Cref{lem:eta-zig} proves the following about this zigzag:
\begin{itemize}
\item componentwise strict $\Ga$-maps;
\item strictly natural with respect to strict $\Ga$-maps;
\item componentwise stable equivalences
\end{itemize}

\clearpage





\bibliographystyle{amsalpha2}
\bibliography{Refs}%

\end{document}